\documentclass[USenglish]{article}	

\usepackage[utf8]{inputenc}				
\usepackage[big,online]{dgruyter2}	
\usepackage{lmodern} 
\usepackage{microtype}
\usepackage[numbers,square,sort&compress]{natbib}
\usepackage{mathtools}
\usepackage{amsmath}
\usepackage{enumitem}
\theoremstyle{dgthm}
\newtheorem{theorem}{Theorem}
\newtheorem{lemma}[theorem]{Lemma}
\newtheorem{assumption}[theorem]{Assumption}
\newtheorem{problem}[theorem]{Problem}

\theoremstyle{dgdef}

\newtheorem{remark}[theorem]{Remark}

\def\div{\operatorname{div}}

\def\RR{\mathbb{R}}
\def\LL{\mathbf{L}}

\def\Th{\mathcal{T}_h}
\def\Itau{\mathcal{I}_\tau}
\def\Id{\mathbf{I}}

\def\la{\langle}
\def\ra{\rangle}

\def\Th{\mathcal{T}_h}
\def\Vh{\mathcal{V}_h}

\DeclarePairedDelimiter{\norm}{\|}{\|}
\DeclarePairedDelimiter{\snorm}{|}{|}
\newcommand{\na}{\nabla}

\def\u{\mathbf{u}}

\def\na{\nabla}
\def\ww{\mathbf{w}}

\def\dt{\partial_t}

\def\ddt{\frac{\mathrm{d}}{\mathrm{d}t}}

\usepackage{xcolor}
\def\changesone#1{{#1}}
\def\changestwo#1{{#1}}

\begin{document}

	\articletype{Research Article}
  \startpage{1}

\title{Variational approximation for a non-isothermal coupled phase-field system: \\  Structure-preservation \& Nonlinear stability }
\runningtitle{Non-isothermal coupled phase-field models}

\author*[1]{Aaron Brunk}
\author[2]{Oliver Habrich}
\author[3]{Timileyin David Oyedeji} 
\author[3]{Yangyiwei Yang} 
\author[3]{Bai-Xiang Xu} 
\runningauthor{A.~Brunk et al.}
\affil[1]{\protect\raggedright 
Institute of Mathematics, Johannes Gutenberg-University Mainz, Germany, e-mail: abrunk@uni-mainz.de}
\affil[2]{\protect\raggedright 
Institute for Numerical Mathematics, Johannes Kepler University Linz, Austria, e-mail:  oliver.habrich@jku.at}
\affil[3]{\protect\raggedright 
Mechanics of Functional Materials Division, Technical University Darmstadt, Germany, e-mail: timileyin.oyedeji@tu-darmstadt.de,
yangyiwei.yang@mfm.tu-darmstadt.de, xu@mfm.tu-darmstadt.de}	
	
\abstract{A Cahn-Hilliard-Allen-Cahn phase-field model coupled with a heat transfer
equation, particularly with full non-diagonal mobility matrices, is studied. After reformulating the problem w.r.t. the inverse of temperature, we proposed and analysed a structure-preserving approximation for the semi-discretisation in space and then a fully discrete approximation using conforming finite elements and time-stepping methods. We prove structure-preserving property and discrete stability using relative entropy methods for the semi-discrete and fully discrete case. The theoretical results are illustrated by numerical experiments.}


\maketitle

\section{Introduction}

Non-isothermal sintering involves the densification of packed powders into bulk materials under non-isothermal conditions \cite{german2014sintering, kang2004}. Thus, the properties of sintered products are influenced by both conventional factors, such as the size and chemical composition of powders and pressure, and also non-isothermal factors, such as temperature inhomogeneity and heating/cooling rate. Understanding the interplay between these factors and the different physical effects in the non-isothermal sintering process is crucial to bridge process parameters, microstructure, and properties of the sintered materials \cite{Timi23}. In this regard, phase-field models have been widely employed to study the complex microstructure evolution and the intricate multi-physics in non-isothermal sintering \cite{Yang2019, yangscripta2020}. 

In a typical phase-field sintering model, a conserved order parameter $\rho$ is used to differentiate the grain and the pore/atmosphere regions, and a set of non-conserved order parameters $\eta_i$ are used to represent the orientations of the grains. While the temporal evolution of $\rho$ is governed by a form of the Cahn-Hilliard equation describing mass transfer, the evolution of $\eta_i$ is governed by a form of the Allen-Cahn equation depicting grain growth \cite{Kazaryan1999, wang2006}. Despite its extensive use, a major theoretical issue of conventional phase-field models is the quantitative validity associated with artificial interface effects like violation of conservation laws and discontinuity of the chemical/thermal potentials at the interface \cite{echebarria2004}. Variational quantitative phase-field models like in \eqref{eq:s1}--\eqref{eq:s3} are formulated to contain cross-coupling terms that are essential for the elimination of artificial interface effects \cite{brener2012, Boussinot2013}.\\
In this work, we consider a temperature dependent Cahn-Hilliard-Allen-Cahn system with cross-coupling given by
\begin{alignat}{2}
\dt\rho &= \div\Big(\LL_{11}\nabla\tfrac{\mu_\rho}{T} - \LL_{12}\nabla\tfrac{1}{T}  +  \LL_{13}\tfrac{\mu_\eta}{T}\Big), \qquad &\tfrac{\mu_\rho}{T} &= -\gamma_\rho\Delta\rho + \tfrac{1}{T}\tfrac{\partial f}{\partial\rho}(\rho,T,\eta), \label{eq:s1}\\
\dt e(\rho,T,\eta) &= \div\Big( \LL_{12}\nabla\tfrac{\mu_\rho}{T} -  \LL_{22}\nabla\tfrac{1}{T} + \LL_{23}\tfrac{\mu_\eta}{T}  \Big),\label{eq:s2} \\
\dt\eta &= -\LL_{13}\cdot\nabla\tfrac{\mu_\rho}{T} + \LL_{23}\nabla\tfrac{1}{T}  - \LL_{33}\tfrac{\mu_\eta}{T}, \qquad & \tfrac{\mu_\eta}{T} &= -\gamma_\eta\Delta\eta + \tfrac{1}{T}\tfrac{\partial f}{\partial\eta}(\rho,T,\eta).\label{eq:s3}
\end{alignat}
The system is complemented by periodic boundary conditions and suitable initial conditions and endowed with a free energy density given by \begin{align*}
f(\rho,\eta,T) = \tilde f(\rho,\eta,T) + \tfrac{T\gamma_\rho}{2}\snorm{\nabla\rho}^2+ \tfrac{T\gamma_\eta}{2}\snorm{\nabla\eta}^2.
\end{align*}
From this energy density $f$ by standard arguments, we obtain the internal energy density $e$ and the entropy density $s$. 

The above model is a mathematical simplification of the phase-field model for non-isothermal sintering \cite{Timi23}, i.e. only one $\eta$ is employed instead of a whole set. However, the structure can be translated into the complete model in a straightforward fashion. 
In a broader context, there have been several different generalisations of phase-field models into the non-isothermal setting; see \cite{Charach1998,Fabrizio2006,Pawlow2016} for some general ansatz and a quite intensive overview. They typically deal either with a conserved or non-conserved phase-field, i.e. Allen-Cahn or Cahn-Hilliard type.
The well-known generalisations are for the non-conserved phase-field introduced by Penrose and Five \cite{Penrose1990}, while for a conserved phase-field, a similar model is introduced by Pawlow and Alt \cite{Alt1992} where the system driving force is the entropy. Both models and related systems can also be derived within the well-known GENERIC or SNET formalism, see \cite{Gladkov2016}. Modelling approaches based on the micro-force balance in the spirit of Gurtin can be found in \cite{Marveggio2021} or using the least action principle in \cite{deanna2022temperature}. Similar models are derived by Caginalp using finite prorogation of temperature and afterwards extended see \cite{Caginalp1986,colli2022cahnhilliard}.\\
%
Many existence results, weak or classical, are attributed to Colli and co-workers for various of the above models, see \cite{Colli_memory_smooth,Colli_memory_weak,Colli_AC,Colli_memory,Colli_Penrose,Colli_nonlinear_multi,colli1998weak}.
Existence results and some further analysis for the model of Pawlow and Alt can be found in \cite{Alt1990,Alt1991,Alt1992,KENMOCHI19941163,zheng1992global}. We note that only the work of Pawlow and Alt \cite{Alt1992,alt1992existence} deal with a cross-coupled system, where the maximum principle is expected to be lost, while most other results heavily rely on this principle. \changesone{In principle, they are able to prove existence of generalized solutions for the non-isothermal Cahn-Hilliard model using inverse temperature as main variable, i.e. system \eqref{eq:s1}--\eqref{eq:s2}, using suitable growth conditions on the driving potential, cf. \cite{alt1992existence}. The solutions are governed by the usual weak formulation, but with an defect measure in the equation for the internal energy, due to the weak limit in the internal energy itself. In principle, we expect that the proof can be extended to the model in consideration here, \eqref{eq:s1}--\eqref{eq:s2}.}\\

For a numerical approximation of such models a purely implicit Euler method was introduced by \cite{Pawlow2016}, while Gonzalez's and coworkers \cite{GuillnGonzlez2009,GonzalezFerreiro2014} generalised this to an energy-stable method via the average vector field methods \cite{McLachlanEtAl99,Gonzales96}. The system in consideration, as well as suitable subsystems, are endowed with a gradient structure. While for the Cahn-Hilliard and Allen-Cahn equations, the typical methods are convex-concave splitting \cite{Jie}, scalar auxiliary variable approach (SAV) \cite{SHEN2018407,LI23,Akrivis19,CHEN2019} and Energy Quadratisation (EQ) \cite{GONG2019,Yang2020,CHEN2022,Zhang2022}, where the latter two relax the driving functional energy or entropy by introduction of another variable, either only time or space-time-dependent. For the internal energy equation, some works use Energy Quadratisation, see \cite{Guo_2015,Sun2020}.\\

The aim of this work is to study the variational properties of the above system and its systematic discretisation in space and time. We extend the classical convex-concave splitting from the isothermal case \cite{Jie} to the non-isothermal regime. The novelties of this manuscript are
\begin{itemize}
\item Problem-adapted reformulation w.r.t. the inverse of the temperature to reveal the variational structure
    \item Structure-preserving approximation via suitable application of standard discretisations in space and afterwards also in time.
    \item Nonlinear stability estimate for the semi-discrete and the fully discrete scheme by means of relative entropy.
\end{itemize} 

\changesone{We underline that fundamentally, both numerical approximations—the semi-discrete and the fully discrete—can potentially form a foundation for establishing the existence of generalized solutions through appropriate weak limits. While both discretizations maintain the thermodynamic structure, ensuring the conservation of mass and internal energy, as well as entropy production, they also fulfill the nonlinear stability estimate via the relative entropy method. This nonlinear stability estimate, analogous to the well-established weak-strong uniqueness principle in the continuous case, can be utilized for investigating convergence and error estimates. We anticipate that a similar stability estimate can be inferred for the generalized solutions leading to a weak-strong uniqueness principle. This aspect is currently under investigation and necessitates significant adjustments to the proofs presented in the manuscript.}
\changestwo{We observe that the semi-discretization, along with its associated nonlinear stability estimate, demonstrates that discretization using conforming finite elements in space maintains the intrinsic structure of the model, including thermodynamic consistency and stability as indicated by the relative entropy method. Essentially, this formulation allows for the development of time discretization, and we suggest a relatively simple time discretization method based on classical convex-concave splitting. However, the semi-discretization can serve as a promising starting point for investigating alternative time discretization approaches without concerning space discretisation. Another advantage of the space discretization lies in the straightforward proof of the stability estimate and illustrates what kind of estimate is to be expected. Conversely, in the fully discrete scheme, this structure is complicated by the non-conforming approximation in time, resulting in more technical estimates as detailed in Appendix \ref{app:1}.}

The manuscript is structured as follows. In Section \ref{sec:not}, we will reformulate the system by introducing the inverse temperature as a main variable instead of the typically used temperature. Afterwards, in Section \ref{sec:semidiscrete}, we will use the reformulated system to introduce a semi-discretisation in space using piecewise linear conforming finite elements and a nonlinear stability estimate, which is proven in Section \ref{sec:proof_stab_semi}. In Section \ref{sec:fulldiscrete}, the fully discrete scheme using implicit time discretisation and a convex-concave splitting of the driving potential is proposed and analysed, and a nonlinear stability results is given, which is proven in the Appendix \ref{app:1}. Finally, we present a convergence test and illustrate the evolution of the system through an application-inspired example.

\section{Notation \& auxiliary results}\label{sec:not}

Before we present our discretisation method and main results in detail, let us briefly introduce our notation and main assumptions, and recall some basic facts.

\subsection{Notation}
The system \eqref{eq:s1}--\eqref{eq:s3} is investigated on a finite time interval $(0,T_f)$. 
To avoid the discussion of boundary conditions, we consider a spatially periodic setting, i.e., 
\begin{itemize}
\item[(A0)] $\Omega \subset \RR^d$, $d=2,3$ is a cube and identified with the $d$-dimensional torus $\mathcal{T}^d$.\\
Moreover, functions on $\Omega$ are assumed to be periodic throughout the paper. 
\end{itemize}
By $L^p(\Omega)$, $W^{k,p}(\Omega)$, we denote the corresponding Lebesgue and Sobolev spaces of periodic functions with norms $\norm{\cdot}_{L^{p}}$ and $\norm{\cdot}_{W^{k,p}}$. 
As usual, we abbreviate $H^k(\Omega)=W^{k,2}(\Omega)$ and write $\norm{\cdot}_{H^{k}} = \norm{\cdot}_{W^{k,2}}$. 
The corresponding dual spaces are denoted by $H^{-k}(\Omega)=H^k(\Omega)^*$, with norm defined by
\begin{align} \label{eq:dualnorm}
    \norm{r}_{H^{-k}} = \sup_{v \in H^s(\Omega)} \frac{\la r, v\ra}{\|v\|_{k}}.
\end{align}
Here $\langle \cdot, \cdot\rangle$ denotes the duality product on $H^{-k}(\Omega) \times H^k(\Omega)$.
The same symbol is also used for the scalar product on $L^2(\Omega)$, which is defined by
\begin{align*}
\la u, v \ra = \int_\Omega u \cdot v \, dx \qquad \forall u,v \in L^2(\Omega).    
\end{align*}
By $L^2_0(\Omega) \subset L^2(\Omega)$, we denote the spaces of square-integrable functions with zero average.
As usual, we denote by $L^p(a,b;X)$, $W^{k,p}(a,b;X)$, and
$H^k(a,b;X)$, the Bochner spaces of integrable or differentiable functions on the time interval $(a,b)$ with values in some Banach space $X$. If $(a,b)=(0,T)$, we omit reference to the time interval and briefly write $L^p(X)$. The corresponding norms are denoted, e.g., by $\|\cdot\|_{L^p(X)}$ or $\|\cdot\|_{H^k(X)}$.
\section{Variational structure and reformulation}
To have a physically meaningful model, we consider basic physical properties, in this case, conservation of mass, internal energy and entropy production.
For the investigation, we recall that the entropy and internal energy densities are defined as follows
\begin{align*}
s(\rho,\eta,T)&=-\tfrac{\partial f(\rho,\eta,T)}{\partial T}=   -\tfrac{\partial \tilde f(\rho,\eta,T)}{\partial T} - \tfrac{\gamma_\rho}{2}\snorm{\nabla\rho}^2- \tfrac{\gamma_\eta}{2}\snorm{\nabla\eta}^2  , \\  
e(\rho,\eta,T) &= f(\rho,\eta,T) + T s(\rho,\eta,T) = \tilde f(\rho,\eta,T) - T\tfrac{\partial \tilde f(\rho,\eta,T)}{\partial T}.
\end{align*}
Finally, we observe that the entropy can be rewritten as
\begin{equation}
s(\rho,\eta,T) =  \tfrac{e(\rho,\eta,T)}{T} - \tfrac{1}{T}\tilde f(\rho,\eta,T)  - \tfrac{\gamma_\rho}{2}\snorm{\nabla\rho}^2- \tfrac{\gamma_\eta}{2}\snorm{\nabla\eta}^2. \label{eq:entrophelp}
\end{equation}
\begin{remark}
 Following the typical thermodynamic derivations, the Helmholtz free energy $f$ and the temperature $T$ are introduced via the free variable $e$ and the function $s(e)$. Following the computations in Penrose and Five \cite{Penrose1990}, one can see that $\frac{\delta}{\delta \phi} (\frac{f}{T}) = -\frac{\delta}{\delta \phi} s$, hence the main driving force of the system is the entropy.
\end{remark}
\noindent 
\textbf{Weak formulation \& conservation}
The weak form of the problem is obtained by testing the equations with appropriate test functions and using integration-by-parts for some terms. This leads to the variational equations,
\begin{align*}
\la \dt \rho,v_1\ra + \la \LL_{11}\nabla\tfrac{\mu_\rho}{T} - \LL_{12}\nabla\tfrac{1}{T}  +  \LL_{13}\tfrac{\mu_\eta}{T}, \nabla v_1\ra &= 0 \\
\la \tfrac{\mu_\rho}{T}, v_2\ra -\la \gamma_\rho \nabla \rho,\nabla v_2\ra - \la \tfrac{f_\rho}{T}, v_2\ra&=0  \\
\la \dt e(\rho,\eta,T), v_3\ra + \la \LL_{12}\nabla\tfrac{\mu_\rho}{T} -  \LL_{22}\nabla\tfrac{1}{T} + \LL_{23}\tfrac{\mu_\eta}{T}, \nabla v_3\ra &= 0 \\
\la \dt \eta, w_1\ra + \la \LL_{13}\cdot\nabla\tfrac{\mu_\rho}{T} - \LL_{23}\nabla\tfrac{1}{T}  + \LL_{33}\tfrac{\mu_\eta}{T}, w_1\ra &= 0 \\
\la \tfrac{\mu_\eta}{T}, w_2 \ra -\la \gamma_\eta\nabla\eta,\nabla w_2 \ra - \la \tfrac{f_\eta}{T}, w_2 \ra&=0
\end{align*}
which hold for all sufficiently smooth periodic test functions $v_1,v_2,v_3$ and $w_1,w_2$.
The internal energy $\int_\Omega e(\rho,\eta,T) $ and the total mass $\int_\Omega \rho$ are then automatically conserved by testing with $v_1=1$ and $\xi=1$ respectively.

\bigskip 
\noindent 
\textbf{Entropy production.}
For the evolution of the entropy we use \eqref{eq:entrophelp} and insert $(v_1,v_2,\xi,w_1,w_2)=(0,\dt\rho,0,0,\dt\eta)$ and $(v_1,v_2,\xi,w_1,w_2)=(-\mu_\rho,0,\frac{1}{T},\mu_\eta,0)$, to obtain
\begin{align*}
\ddt\la s(\rho,\eta,T),1\ra 
=& \la \dt e,\tfrac{1}{T} \ra + \la e(\rho,\eta,T) -f(\rho,\eta,T)+T\tfrac{\partial f(\rho,\eta,T)}{\partial T},\dt(\tfrac{1}{T}) \ra \\
&-  \gamma_\rho\la \nabla\rho,\nabla\dt\rho \ra - \la \partial_\rho\tilde f,\dt\rho \ra - \gamma_\eta\la \nabla\eta,\nabla\dt\eta \ra - \la \partial_\eta\tilde f,\dt\eta \ra\\
=&  \la \dt e,\tfrac{1}{T} \ra - \la \dt \rho,\tfrac{\mu_\rho}{T} \ra - \la \dt \eta,\tfrac{\mu_\eta}{T} \ra \\
=& - \la \LL_{12}\na\tfrac{\mu_\rho}{T}, \na\tfrac{1}{T}\ra + \la \LL_{22}\na\tfrac{1}{T}, \na\tfrac{1}{T}\ra - \la \tfrac{\mu_\eta}{T}\LL_{23}, \na\tfrac{1}{T}\ra \\
&+ \la \LL_{11}\na\tfrac{\mu_\rho}{T}, \na\tfrac{\mu_\rho}{T}\ra - \la \LL_{12}\na\tfrac{1}{T}, \na\tfrac{\mu_\rho}{T}\ra + \la \tfrac{\mu_\eta}{T}\LL_{13}, \na\tfrac{\mu_\rho}{T}\ra \\
&+ \la \LL_{13}\na\tfrac{\mu_\rho}{T}, \tfrac{\mu_\eta}{T}\ra - \la \LL_{23}\na\tfrac{1}{T}, \tfrac{\mu_\eta}{T}\ra + \la \LL_{33}\tfrac{\mu_\eta}{T}, \tfrac{\mu_\eta}{T}\ra \\
=&  \int_\Omega \begin{pmatrix} \na\tfrac{\mu_\rho}{T} \\ \na\tfrac{1}{T} \\ \tfrac{\mu_\eta}{T}\end{pmatrix}^\top\begin{pmatrix} \LL_{11} & -\LL_{12} & \LL_{13}\\
-\LL_{12} & \LL_{22} & -\LL_{23}\\
\LL_{13}^\top\ & -\LL_{23}^\top\ & \LL_{33}\\
\end{pmatrix}\begin{pmatrix} \na\tfrac{\mu_\rho}{T} \\ \na\tfrac{1}{T} \\ \tfrac{\mu_\eta}{T}\end{pmatrix}.
\end{align*}

Entropy production can then be guaranteed by assuming that the mobility/conductivity matrix $\LL$ is positive (semi-)definite. 

\bigskip 
\noindent 
\textbf{Variable transformation.}
The weak formulation and entropy-production equality suggests to introduce new variables 
\begin{align*}
\theta=1/T, \qquad \tilde\mu_\rho=\mu_\rho/T, \qquad \tilde\mu_\eta = \mu_\eta/T. 
\end{align*} 
To this end, we consider reduced free energy
\begin{align*}
 \Phi(\rho,\theta,\eta) := \theta f(\rho,\tfrac{1}{\theta},\eta) = \psi(\rho,\theta,\eta) + \frac{\gamma_\rho}{2}\snorm{\nabla\rho}^2  + \frac{\gamma_\eta}{2}\snorm{\nabla\eta}^2.    
\end{align*}
With this, we can compute the entropy and internal energy in the new variables and find
\begin{align*}
  \tilde e(\rho,\theta,\eta) &= \frac{\partial\Phi(\rho,\theta,\eta)}{\partial\theta} = \frac{\psi(\rho,\theta,\eta)}{\partial\theta}  , \\
  \tilde s(\rho,\theta,\eta) &= \theta\frac{\partial\Phi(\rho,\theta,\eta)}{\partial\theta} - \Phi(\rho,\theta,\eta) = \theta\frac{\psi(\rho,\theta,\eta)}{\partial\theta}  - \psi(\rho,\theta,\eta) - \frac{\gamma_\rho}{2}\snorm{\nabla\rho}^2  - \frac{\gamma_\eta}{2}\snorm{\nabla\eta}^2
\end{align*}

Again, we express the entropy in terms of the internal energy and find
\begin{align}
\tilde s(\rho,\theta,\eta) = \theta \tilde e(\rho,\theta,\eta)  - \psi(\rho,\theta,\eta) - \frac{\gamma_\rho}{2}\snorm{\nabla\rho}^2  - \frac{\gamma_\eta}{2}\snorm{\nabla\eta}^2  \label{eq:entropy}
\end{align}

Using this the rewritten system in weak formulation reads
\begin{align}
\la \dt\rho,v_1 \ra &+ \la \LL_{11}\nabla \tilde\mu_\rho,\nabla v_1 \ra - \la \LL_{12}\nabla\theta,\nabla v_1 \ra + \la \tilde\mu_\eta\LL_{13},\nabla v_1 \ra =0, \label{eq:noniso1}\\   
\la \tilde\mu_\rho,v_2 \ra &- \gamma_\rho\la \nabla\rho,\nabla v_2 \ra - \la  \partial_\rho\psi,v_2 \ra = 0, \label{eq:noniso2}\\
\la \dt \tilde e(\rho,\theta,\eta),\xi \ra &+ \la \LL_{12}\nabla \tilde\mu_\rho,\nabla \xi \ra - \la \LL_{22}\nabla\theta,\nabla \xi \ra + \la \tilde\mu_\eta\LL_{23},\nabla \xi \ra =0, \label{eq:noniso3}\\   
\la \dt\eta,w_1 \ra &+ \la w_1\LL_{13},\nabla\tilde\mu_\rho \ra - \la w_1\LL_{23},\nabla\theta \ra + \la \LL_{33}\tilde\mu_\eta,w_1 \ra = 0,\label{eq:noniso4}\\
\la \tilde\mu_\eta,w_2 \ra &- \gamma_\eta\la \nabla\eta,\nabla w_2 \ra - \la \partial_\eta\psi,w_2 \ra = 0. \label{eq:noniso5}
\end{align}
Furthermore, we will directly drop the superscript, i.e. $\tilde e,\tilde s, \tilde\mu_\rho,\tilde\mu_\eta$ will be denoted by $e,s, \mu_\rho,\mu_\eta$.
The entropy-production now again follows directly by testing with  $(v_1v_2,\xi,w_1,w_2)=(\mu_\rho,\dt\rho,\theta,\mu_\eta,\dt\eta)$ and in accordance with \eqref{eq:entropy} yields
\begin{align*}
\la\dt  s(\rho,\theta,\eta),1\ra =&  \la \dt \theta ,  e(\rho,\theta,\eta)  - \frac{\partial\psi(\rho,\theta,\eta)}{\partial \theta}\ra  + \la \dt \tilde e(\rho,\theta,\eta), \theta \ra \\
&-\gamma_\rho\la \nabla\dt\rho, \nabla\rho \ra -\la \dt\rho,\psi_\rho(\rho,\theta,\eta) \ra   - \gamma_\rho\la \nabla\dt\eta, \nabla\eta \ra - \la \dt\eta, \psi_\eta(\rho,\theta,\eta) \ra \\
 =& \la \dt \theta , e(\rho,\theta,\eta) - \la \dt\rho,\mu_\rho \ra - \la \dt\eta,\mu_\eta \ra \\
  =& \Big\la \begin{pmatrix} \na\mu_\rho \\ \na\theta \\ \mu_\eta\end{pmatrix},\begin{pmatrix} \LL_{11} & -\LL_{12} & \LL_{13}\\
-\LL_{12} & \LL_{22} & -\LL_{23}\\
\LL_{13}^\top\ & -\LL_{23}^\top\ & \LL_{33}\\
\end{pmatrix}\begin{pmatrix} \na\mu_\rho \\ \na\theta \\ \mu_\eta\end{pmatrix} \Big\ra\\
 =&: \mathcal{D}_\LL(\mu_\rho,\theta,\mu_\eta),
\end{align*}
which is non-negative provided that $\LL$ is positive semi-definite.

In the following, we collect the minimal assumptions for the rest of this manuscript.
\begin{assumption}\label{ass}
    
\begin{itemize} \phantom{1}
\item[(A0)] $\Omega \subset \RR^d$, $d=2,3$ is a cube and identified with the $d$-dimensional torus $\mathcal{T}^d$.
    \item[(A1)] The interface parameters $\gamma_\rho,\gamma_\eta\in\RR$ are positive.
    \item[(A2)] The matrix $\LL(\rho,\nabla\rho,\theta,\eta,\nabla\eta)=:\LL\in\mathbb{R}^{(2d+1) \times (2d+1)}$ is assumed to be positive definite such that for the positive constants $\lambda_0,\lambda_1>0$ we have
    \begin{equation*}
      \lambda_1\snorm{\xi}^2 \geq \mathbf{\xi}^\top\LL\mathbf{\xi} \geq \lambda_0\snorm{\xi}^2, \quad \forall \mathbf{\xi}\in\mathbb{R}^{2d+1}
    \end{equation*}
    Furthermore, we assume every component of $\LL$ to be a $C^1(\RR)$ function of all arguments, such that $|\frac{\partial \LL}{\partial \rho}|,|\frac{\partial \LL}{\partial \theta}|,|\frac{\partial \LL}{\partial \eta}|,|\frac{\partial \LL}{\partial \nabla\rho}|,|\frac{\partial \LL}{\partial \nabla\eta}|\leq \LL_3$.
    \item[(A3)] For the smooth driving potential $\psi(\cdot,\cdot,\cdot):\RR\times\RR_+\times\RR\to \RR$ we assume that for every fixed $(\rho,\eta)$ the potential $\psi(\rho,\cdot,\eta):\RR_+\to\RR$ is concave and goes to infinity for $\theta\to 0$. Furthermore, for every fixed $\theta\in\RR_+$ the potential $\psi(\cdot,\theta,\cdot)$ can be decomposed in a convex and a concave part, denoted by $\psi_{vex}, \psi_{cav}.$
    \item[(A4)] We assume that there exists a constant $\alpha\in\RR$ such that $\psi(\rho,\theta,\eta) + \frac{\alpha}{2}(\norm{\rho}_0^2 + \norm{\eta}_0^2)$ is strictly convex for every fixed, positive $\theta$. 
\end{itemize}
\end{assumption}
Note that in order to make physical sense out of the internal energy and entropy, one would require additional properties, especially the growth conditions for $\theta$ approaching to $0$ or $\infty$.

\section{Semi-discrete in space:}\label{sec:semidiscrete}
In this section, we will introduce a finite element discretisation of the weak formulation model \eqref{eq:noniso1}--\eqref{eq:noniso5}.
\subsection*{Space discretisation}

As another preparatory step, let us introduce the relevant notation and assumptions for our discretisation strategy. We require that
\begin{itemize}\itemsep1ex
    \item[(A5)] $\Th$ is a geometrically conforming partition of $\Omega$ into simplices that can be extended periodically to periodic extensions of $\Omega$, where $h$ denotes the maximal diameter of a triangle in $\Vh$.
\end{itemize}
We denote by 
\begin{align*}
    \Vh &:= \{v \in H^1(\Omega)\cap C^0(\bar\Omega) : v|_K \in P_1(K) \quad \forall K \in \Th\},\quad
    \Vh^+ &:= \{v \in \Vh : v(x) > 0,\; \forall x\in\Omega\}
\end{align*}
the space of continuous and piecewise linear functions over $\Th$ and its positive cone. Note that since every value in an element $K$ is deduced from a convex combination of the nodal values, it is equivalent to assume that all nodal values are positive. The semi-discrete formulation then reads.
\begin{problem}\label{prob_semi}
  Let $(\rho_{0,h},\theta_{0,h},\eta_{0,h})\in \Vh\times\Vh^+\times\Vh$ be given. Find $(\rho_h,\theta_h,\eta_h)\in C^1(0,T;\Vh\times\Vh^+\times\Vh)$ and $(\mu_{\rho,h},\mu_{\eta,h}) \in C^0(0,T;\Vh\times\Vh)$  such that $(\rho_h,\theta_h,\eta_h)(0)=(\rho_{0,h},\theta_{0,h},\eta_{0,h})$ and such that for all $v_{1,h}, v_{2,h}, \xi_h, w_{1,h},w_{2,h} \in \Vh$ and all $0 \le t \le T$, there holds 
\begin{align}
\la \dt\rho_h,v_{1,h} \ra &+ \la \LL_{11}\nabla \mu_{\rho,h},\nabla v_{1,h} \ra - \la \LL_{12}\nabla\theta_h,\nabla v_{1,h} \ra + \la \mu_{\eta,h}\LL_{13},\nabla v_{1,h} \ra =0, \label{eq:nonisoh1}\\   
\la \mu_{\rho,h},v_{2,h} \ra &- \gamma_\rho\la \nabla\rho_h,\nabla v_{2,h}\ra - \la  \partial_\rho\psi(\rho_h,\theta_h,\eta_h),v_{2,h} \ra = 0, \label{eq:nonisoh2}\\
\la \dt  e(\rho_h,\theta_h,\eta_h),\xi_h \ra &+ \la \LL_{12}\nabla \mu_{\rho,h},\nabla \xi_h \ra - \la \LL_{22}\nabla\theta_h,\nabla \xi_h \ra + \la \mu_{\eta,h}\LL_{23},\nabla \xi_h \ra =0, \label{eq:nonisoh3}\\   
\la \dt\eta_h,w_{1,h} \ra &+ \la w_{1,h}\LL_{13},\nabla\mu_{\rho,h} \ra - \la w_{1,h}\LL_{23},\nabla\theta \ra + \la \LL_{33}\mu_{\eta,h},w_{1,h} \ra = 0,\label{eq:nonisoh4}\\
\la \mu_{\eta,h},w_{2,h} \ra &- \gamma_\eta\la \nabla\eta_h,\nabla w_{2,h} \ra - \la \partial_\eta\psi(\rho_h,\theta_h,\eta_h),w_{2,h} \ra = 0. \label{eq:nonisoh5}
\end{align}
\end{problem}

\begin{lemma}
   Let (A0)--(A3), (A5) hold, cf. Assumption \ref{ass}. Solution of \eqref{eq:nonisoh1}--\eqref{eq:nonisoh5}. The solutions are mass-, energy-conservative and produce entropy, i.e.
   \begin{align*}
    \la \dt\phi_h, 1\ra= \la \dt e(\rho_h,\theta_h,\eta_h), 1\ra =0, \qquad \la \dt s(\rho_h,\theta_h,\eta_h),1\ra = \mathcal{D}_\LL(\mu_{\rho,h},\theta_h,\mu_{\eta,h})
   \end{align*}
\end{lemma}
\begin{proof}
 Conservation of mass and energy follows by inserting $(v_{1,h},v_{2,h},\xi_h,w_{1,h},w_{2,h})=(1,0,0,0,0)$ and $(v_{1,h},v_{2,h},\xi_h,w_{1,h},w_{2,h})=(0,0,1,0,0)$, i.e.
 \begin{equation*}
   \la \dt\phi_h(t),1\ra =0, \qquad\la \dt e(\rho_h(t),\theta_h(t),\eta_h(t),1\ra =0.  
 \end{equation*}
 Similarly by inserting $(v_{1,h},v_{2,h},\xi_h,w_{1,h},w_{2,h})=(\mu_{\rho,h},\dt\phi_h,\theta_h,\mu_{\eta,h},\dt\eta_h)$ we obtain
 \begin{align*}
    \la \dt s(\rho_h(t),\theta_h(t),\eta_h(t),1\ra =\mathcal{D}_\LL(\mu_{\rho,h},\theta_h,\mu_{\eta,h}). 
 \end{align*}
\end{proof}

Note that the existence of discrete solutions can be established via Picard-Lindelöff requiring some growth conditions on $\psi,e,s$.

\section{Stability estimate}\label{sec:stability1}
In this section, we consider the stability of the semi-discrete scheme with respect to a suitable perturbed system. The \emph{distance} between the solutions will be measurable by a problem-tailored relative entropy functional shown later on.

We introduce the perturbed system as follows. For the functions $(\hat\rho_h,\hat\mu_{\rho,h},\hat\theta_h,\hat\eta_h,\hat\mu_{\eta,h})$ we introduce the residuals $r_{i,h}, i=1,\ldots,5$ via the following variational identities
\begin{align}
\la \dt\hat\rho_h,v_{1,h} \ra &+ \la \LL_{11}\nabla \hat\mu_{\rho,h}- \LL_{12}\nabla\hat\theta_h+ \hat\mu_{\eta,h}\LL_{13},\nabla v_{1,h} \ra =\la r_{1,h},v_{1,h} \ra, \label{eq:nonisoph1}\\   
\la \hat\mu_{\rho,h},v_{2,h} \ra &- \gamma_\rho\la \nabla\hat\rho_h,\nabla v_{2,h}\ra - \la  \partial_\rho\hat\psi(\hat\rho_h,\hat\theta_h,\hat\eta_h),v_{2,h} \ra = \la r_{2,h},v_{2,h} \ra, \label{eq:nonisoph2}\\
\la \dt e(\hat\rho_h,\hat\theta_h,\hat\eta_h),\xi_h \ra &+ \la \LL_{12}\nabla \hat\mu_{\rho,h}- \LL_{22}\nabla\hat\theta_h+\hat\mu_{\eta,h}\LL_{23},\nabla \xi_h \ra =\la r_{3,h},\xi_{h} \ra, \label{eq:nonisoph3}\\   
\la \dt\hat\eta_h,w_{1,h} \ra &+ \la \LL_{13}\nabla\hat\mu_{\rho,h} - \LL_{23}\nabla\hat\theta_h+\LL_{33}\hat\mu_{\eta,h},w_{1,h} \ra = \la r_{4,h},w_{1,h} \ra,\label{eq:nonisoph4}\\
\la \hat\mu_{\eta,h},w_{2,h} \ra &- \gamma_\eta\la \nabla\hat\eta_h,\nabla w_{2,h} \ra - \la \partial_\eta\hat\psi(\hat\rho_h,\hat\theta_h,\hat\eta_h),w_{2,h} \ra = \la r_{5,h},w_{2,h} \ra. \label{eq:nonisoph5}
\end{align}
The relative entropy we propose is given by
\begin{align*}
\mathcal{W}(\rho,\theta,\eta|\hat\rho,\hat\theta,\hat\eta):= W_{\hat\theta}(\rho,\theta,\eta) - W_{\hat\theta}(\hat\rho,\hat\theta,\hat\eta) - \partial_\rho(W_{\hat\theta}(\hat\rho,\hat\theta,\hat\eta))(\rho-\hat\rho)  - \partial_\eta(W_{\hat\theta}(\hat\rho,\hat\theta,\hat\eta))(\eta-\hat\eta).   
\end{align*}
where $W_{\hat\theta}(\rho,\theta,\eta)$ denotes a rescaled exergy with respect to a reference "temperature" $\hat\theta$ 
\begin{equation*}
W_{\hat\theta}(\rho,\theta,\eta) = \hat\theta e(\rho,\theta,\eta) - s(\rho,\theta,\eta) = (\hat\theta-\theta)\partial_\theta \psi(\rho,\theta,\eta) + \psi(\rho,\theta,\eta) + \frac{\gamma_\rho}{2}\snorm{\nabla\rho}^2  + \frac{\gamma_\eta}{2}\snorm{\nabla\eta}^2.    
\end{equation*}
To see that the above relative entropy is a meaningful object for comparison, we expand as follows 
\begin{align}
 \mathcal{W}(\rho,\theta,\eta|\hat\rho,\hat\theta,\hat\eta) =& \frac{\gamma_\rho}{2}\snorm{\nabla(\rho-\hat\rho)}^2 + \frac{\gamma_\eta}{2}\snorm{\nabla(\eta-\hat\eta)}^2+ \psi(\rho,\theta,\eta) - \psi(\hat\rho,\hat\theta,\hat\eta)\notag\\
 & - \partial_\theta\psi(\rho,\theta,\eta)(\theta-\hat\theta) - \partial_\rho\psi(\hat\rho,\hat\theta,\hat\eta)(\hat\rho-\hat\rho)  - \partial_\eta\psi(\hat\rho,\hat\theta,\hat\eta)(\hat\eta-\hat\eta)  \label{eq:expexergy}\\
 =& \frac{\gamma}{2}\snorm{\nabla(\rho-\hat\rho)}^2 + \frac{\gamma_\eta}{2}\snorm{\nabla(\eta-\hat\eta)}^2 + \psi(\rho,\hat\theta,\eta|\hat\rho,\hat\theta,\hat\eta) - \psi(\rho,\hat\theta,\eta|\rho,\theta,\eta)\notag
\end{align}
Recalling assumptions (A3) and (A4) $\psi(\rho,\hat\theta,\eta|\hat\rho,\hat\theta,\hat\eta)$ is a Bregman divergence up to a quadratic perturbation. While $ - \psi(\rho,\hat\theta,\eta|\rho,\theta,\eta)$ is strictly convex by assumption. In order to obtain a complete Bregman divergence, we use (A4) as follows
\begin{align*}
\mathcal{W}_{\lambda} (\rho,\theta,\eta|\hat\rho,\hat\theta,\hat\eta) := \mathcal{W}(\rho,\theta,\eta|\hat\rho,\hat\theta,\hat\eta) + \frac{\lambda}{2}(\norm{\rho_h-\hat\rho_h}_0^2 + \norm{\eta_h-\hat\eta_h}_0^2)  
\end{align*}
where $\lambda$ is sufficiently large. Note that this functional can be obtained by penalisation of the exergy with quadratic terms in $\rho$ and $\eta$. Furthermore, in the case that (A4) holds for $\lambda\leq 0$, i.e. the potential is strictly convex, the remaining estimates are valid for $\lambda=0$.
Let us summarize some properties of the relative entropy
\begin{lemma}
Let $(\rho,\theta,\eta), (\hat\rho,\hat\theta,\hat\eta)\in H^1(\Omega)\times L^2(\Omega)\times H^1(\Omega)$ be given and assume that the lower bound $\hat\theta \geq c_+ >0$ holds. The following inequality holds
 \begin{align*}
  \mathcal{W}_{\lambda}(\rho,\theta,\eta|\hat\rho,\hat\theta,\hat\eta) \geq C(\norm{\rho-\hat\rho}_1^2  + \norm{\eta-\hat\eta}_1^2).   
 \end{align*}   
\end{lemma}
\begin{proof}
 This directly follows from the expansion in \eqref{eq:expexergy} and (A3).   
\end{proof}

So far, we have not imposed concrete details on the free energy density $\psi$. In the following, we compute the temporal evolution of the relative entropy in time.

\begin{lemma}\label{lem:proof_stab_semi}
   Let $(\rho_h,\mu_{\rho,h},\theta_h,\eta_h,\mu_{\eta,h})$ be a solution of Problem \ref{prob_semi} and $r_{i,h}, i=1,\ldots,5$ denote the residuals defined by \eqref{eq:nonisoph1}--\eqref{eq:nonisoph5} for a set $(\hat\rho_h,\hat\mu_{\rho,h},\hat\theta_h,\hat\eta_h,\hat\mu_{\eta,h})$ of sufficiently smooth discrete functions.
Then the following stability estimate holds
    \begin{align*}
 \mathcal{W}_{\lambda}(\rho_h,\theta_h,\eta_h&|\hat\rho_h,\hat\theta_h,\hat\eta_h)(t)+ \frac{1}{2}\int_0^t\mathcal{D}_{\LL}(\mu_{\rho,h}-\hat\mu_{\rho,h},\theta_h-\hat\theta_h,\mu_{\eta,h}-\hat\mu_{\eta,h}) \\
 \leq& \mathcal{W}_{\lambda}(\rho_h,\theta_h,\eta_h|\hat\rho_h,\hat\theta_h,\hat\eta_h)(0) + \int_0^t C\mathcal{W}_{\lambda}(\rho_h,\theta_h,\eta_h|\hat\rho_h,\hat\theta_h,\hat\eta_h) \\
 &+ \la \psi_\rho(\rho_h,\theta_h,\eta_h)-\psi_\rho(\hat\rho_h,\hat\theta_h,\hat\eta_h),1 \ra^2 + \norm{\theta_h-\hat\theta_h}_0^2  \\
&+ \la \partial\psi_\theta(\rho_h,\theta_h,\eta_h|\hat\rho_h,\hat\theta_h,\hat\eta_h), \dt\hat\theta\ra + \la \partial_\rho\psi(\rho_h,\theta_h,\eta_h|\hat\rho_h,\hat\theta_h,\hat\eta_h), \dt\hat\rho\ra \\
&+ \la \partial_\eta\psi(\rho_h,\theta_h,\eta_h|\hat\rho_h,\hat\theta_h,\hat\eta_h), \dt\hat\eta\ra \\
&+ C(\norm{r_{1,h}}_{-1}^2 + \norm{r_{2,h}}_{1}^2 + \norm{r_{3,h}}_{-1}^2 + \norm{r_{4,h}}_{0}^2 + \norm{r_{5,h}}_{0}^2) ds.
\end{align*}
\end{lemma}

\begin{proof}
    The proof is given in the Section \ref{sec:proof_stab_semi}.
\end{proof}

\begin{theorem}\label{thm:semistab}
 Let Lemma \ref{lem:proof_stab_semi} hold. Furthermore, we assume the following uniform bounds
\begin{align*}
  \norm{(\rho_h,\theta_h,\eta_h)}_{L^\infty(\Omega_T)} &+ \norm{(\hat\rho_h,\hat\theta_h,\hat\eta_h)}_{L^\infty(\Omega_T)} + \norm{(\dt\hat\rho_h,\dt\hat\theta_h,\dt\hat\eta_h)}_{L^\infty(\Omega_T)} \leq C_\infty, \\  
  \theta_h, \hat\theta_h &\geq c_\infty \qquad \text{ for every } t\in [0,T], 
\end{align*}
with positive constants $c_\infty,C_\infty$ independent of $h$.  Then for all $t\in(0,T)$ we have
  \begin{align*}
 \mathcal{W}_{\lambda}(\rho_h,\theta_h,\eta_h|\hat\rho_h,\hat\theta_h,\hat\eta_h)(t)+& \tfrac{1}{2}\int_0^t\mathcal{D}_{\LL}(\mu_{\rho,h}-\hat\mu_{\rho,h},\theta_h-\hat\theta_h,\mu_{\eta,h}-\hat\mu_{\eta,h})\\
 \leq&  C(T)\mathcal{W}_{\lambda}(\rho_h,\theta_h,\eta_h|\hat\rho_h,\hat\theta_h,\hat\eta_h)(0) \\
 &+ \int_0^t C(\norm{r_{1,h}}_{-1}^2 + \norm{r_{2,h}}_{1}^2 + \norm{r_{3,h}}_{-1}^2 + \norm{r_{4,h}}_{0}^2 + \norm{r_{5,h}}_{0}^2) ds.
\end{align*}
\end{theorem}
\begin{proof}
The proof can be found in the Section \ref{sec:full_stab_semi}. 
\end{proof}

\begin{remark}
  For a fixed $h$, such $L^\infty$-bounds are available. However, to use the estimate for error analysis, $h$ independent bounds are required. To overcome the necessity of these bounds for the semi-discrete solution, we will consider in future work the technique of Feireisl et al see for instance \cite{feireisl2021numerical}.  The crucial idea is to decompose the space-time domain in a part where these bounds hold and the complement. In the complement, more refined estimates for the relative entropy are employed together with suitable growth conditions on the potential $\psi$.  
\end{remark}

\section{Proof of Lemma \ref{lem:proof_stab_semi}}
\label{sec:proof_stab_semi}
For simplicity, we neglect the $h$ in this section.
To compute the temporal evolution of the relative entropy, we consider the following splitting
\begin{equation*}
  \mathcal{W}_{\lambda}(\rho,\theta,\eta|\hat\rho,\hat\theta,\hat\eta) = \mathcal{W}(\rho,\theta,\eta|\hat\rho,\hat\theta,\hat\eta) + \frac{\lambda}{2}(\norm{\rho-\hat\rho}_0^2 + \norm{\eta-\hat\eta}_0^2).
\end{equation*}

\subsection*{Main part}
We start expanding the first term of the relative entropy and find
\begin{align*}
\frac{d}{dt} \int_\Omega \mathcal{W}(\rho,\theta,\eta|\hat\rho,\hat\theta,\hat\eta) =& \gamma_\rho\la \nabla(\rho-\hat\rho),\dt\nabla(\rho-\hat\rho) \ra + \la \partial_\rho\psi - \partial_\rho\hat\psi,\dt(\rho-\hat\rho) \ra \\
&+ \gamma_\eta\la \nabla(\eta-\hat\eta),\dt\nabla(\eta-\hat\eta) \ra + \la \partial_\eta\psi - \partial_\eta\hat\psi,\dt(\eta-\hat\eta) \ra \\
&- \la \partial_\rho\psi - \partial_\rho\hat\psi,\dt(\rho-\hat\rho) \ra - \la \partial_\eta\psi - \partial_\eta\hat\psi,\dt(\eta-\hat\eta) \ra \\
& + \la \partial_\rho\psi,\dt\rho \ra - \la \partial_\rho\hat\psi,\dt\hat\rho \ra - \la \partial_\rho\hat\psi, \dt(\rho-\hat\rho)\ra \\
& + \la \partial_\eta\psi,\dt\eta \ra - \la \partial_\eta\hat\psi,\dt\hat\eta \ra - \la \partial_\eta\hat\psi, \dt(\eta-\hat\eta)\ra \\
& + \la \partial_\theta\psi,\dt\theta \ra - \la \partial_\theta\hat\psi,\dt\hat\theta \ra - \la \partial_\theta\psi, \dt(\theta-\hat\theta)\ra \\
& - \la \rho-\hat\rho, \partial_\rho^2\hat\psi \dt\hat\rho + \partial_\rho\partial_\eta\hat\psi \dt\hat\eta + \partial_\rho\partial_\theta\hat\psi \dt\hat\theta \ra \\
&- \la \eta-\hat\eta, \partial_\eta\partial_\rho\hat\psi \dt\hat\rho + \partial_\eta^2\hat\psi \dt\hat\eta + \partial_\eta\partial_\theta\hat\psi \dt\hat\theta \ra\\
&- \la \theta-\hat\theta, \partial_\theta\partial_\rho\psi \dt\rho + \partial_\theta\partial_\eta\psi \dt\eta + \partial_\theta^2\psi \dt\theta \ra.
\end{align*}
The first two lines can be rewritten by inserting into the variational formulations as
\begin{align*}
(i) + \ldots + (iv) = \la \mu_\rho-\hat\mu_\rho+r_2,\dt(\rho-\hat\rho) \ra + \la \mu_\eta-\hat\mu_\eta+r_4,\dt(\eta-\hat\eta) \ra    .
\end{align*}
Next we recall that $\partial_\theta \psi=e, \partial_\theta \hat\psi =\hat e$ and with suitable cancellations we find
\begin{align*}
(v) + \ldots + (xix) =& - \la \dt(e-\hat e),\theta-\hat\theta \ra + \la \partial_\rho\psi(\rho,\theta,\eta|\hat\rho,\hat\theta,\hat\eta), \dt\hat\rho\ra  \\
&+ \la \partial_\eta\psi(\rho,\theta,\eta|\hat\rho,\hat\theta,\hat\eta), \dt\hat\theta\ra + \la \partial_\theta\psi(\rho,\theta,\eta|\hat\rho,\hat\theta,\hat\eta), \dt\hat\theta\ra   .
\end{align*}
Hence, together we find 
\begin{align*}
    \frac{d}{dt} \int_\Omega \mathcal{W}(\rho,\theta,\eta|\hat\rho,\hat\theta,\hat\eta) =&\la \mu_\rho-\hat\mu_\rho+r_2,\dt(\rho-\hat\rho) \ra + \la \mu_\eta-\hat\mu_\eta+r_5,\dt(\eta-\hat\eta) \ra \\
    & - \la \dt(e-\hat e),\theta-\hat\theta \ra + \la \partial_\rho\psi(\rho,\theta,\eta|\hat\rho,\hat\theta,\hat\eta), \dt\hat\rho\ra  \\
&+ \la \partial_\eta\psi(\rho,\theta,\eta|\hat\rho,\hat\theta,\hat\eta), \dt\hat\theta\ra + \la \partial_\theta\psi(\rho,\theta,\eta|\hat\rho,\hat\theta,\hat\eta), \dt\hat\theta\ra   \\
=& (a) + (b) + (c) + (d) + (e) + (f).
\end{align*}
While $(d), (e) , (f)$ appear in the final result we only have to consider $(a),(b),(c)$
\begin{align*}
(a) + (b) + (c) =&\la \mu_\rho-\hat\mu_\rho+r_2,\dt(\rho-\hat\rho) \ra   - \la \dt(e-\hat e),\theta-\hat\theta \ra + \la \mu_\eta-\hat\mu_\eta+r_5,\dt(\eta-\hat\eta) \ra\\
    =&-\la \LL_{11}\nabla(\mu_\rho-\hat\mu_\rho) - \LL_{12}\nabla(\theta-\hat\theta) + \LL_{13}(\mu_\eta-\hat\mu_\eta), \nabla(\mu_\rho-\hat\mu_\rho+r_2)\ra\\
    & +\la\LL_{12}\nabla(\mu_\rho-\hat\mu_\rho) - \LL_{22}\nabla(\theta-\hat\theta) + \LL_{23}(\mu_\eta-\hat\mu_\eta), \nabla(\theta-\hat\theta)\ra  \\
    &-\la \LL_{13}\nabla(\mu_\rho-\hat\mu_\rho) - \LL_{12}\nabla(\theta-\hat\theta) + \LL_{13}(\mu_\eta-\hat\mu_\eta),(\mu_\eta-\hat\mu_\eta+r_5)\ra \\ 
    &+ \la r_1,\mu_\rho-\hat\mu_\rho+r_2 \ra- \la r_3,\theta-\hat\theta \ra + \la r_4,\mu_\eta-\hat\mu_\eta+r_5 \ra \\
    \leq& - (1-\delta)\mathcal{D}_\LL( \mu_{\rho}-\hat\mu_\rho,\theta-\hat\theta,\mu_{\eta}-\hat\mu_\eta) + C(\delta)(\norm{r_2}_1^2 + \norm{r_5}_0^2) \\
    &+\la r_1,\mu_\rho-\hat\mu_\rho+r_2 \ra - \la r_3,\theta-\hat\theta \ra + \la r_4,\mu_\eta-\hat\mu_\eta+r_5 \ra.
\end{align*}
The last line can be estimated in a standard manner, i.e.
\begin{align*}
 \la& r_1,\mu_\rho-\hat\mu_\rho+r_2 \ra - \la r_3,\theta-\hat\theta \ra + \la r_4,\mu_\eta-\hat\mu_\eta+r_5 \ra \\
 \leq& C\norm{r_1}_{-1}(\norm{\nabla(\mu_\rho-\hat\mu_\rho)}_0 +  \norm{r_2}_1^2 + \la \mu_\rho-\hat\mu_\rho + r_2,1 \ra) \\
 &+ \norm{r_3}_{-1}^2(\norm{\theta-\hat\theta}_0 + \norm{\nabla(\theta-\hat\theta)}_0) + \norm{r_4}_0(\norm{\mu_\eta-\hat\mu_\eta}_0 + \norm{r_5}_0)  \\
 \leq& \delta\mathcal{D}_\LL( \mu_{\rho}-\hat\mu_\rho,\theta-\hat\theta,\mu_{\eta}-\hat\mu_\eta) + C(\norm{r_1}_{-1}^2 + \norm{r_2}_1^2 + \norm{r_3}_{-1}^2 + \norm{r_4}^2 + \norm{r_5}_0^2 ) \\
 & + \la \mu_\rho-\hat\mu_\rho + r_2,1 \ra^2 + \norm{\theta-\hat\theta}_0^2.
\end{align*}
We insert into the variational identities we can simply
\begin{equation*}
  \la \mu_\rho-\hat\mu_\rho + r_2,1 \ra^2 = \la \partial_\rho \psi(\rho,\theta,\eta) - \partial_\rho(\hat\rho,\hat\theta,\hat\eta),1 \ra^2.   
\end{equation*}

\subsection*{Quadratic part}
We compute the evolution of the quadratic terms and find
\begin{align*}
\frac{d}{dt} &\frac{\lambda}{2}\int_\Omega \snorm{\rho-\hat\rho}^2 + \snorm{\eta-\hat\eta}^2  \\
&= \lambda\la \rho-\hat\rho,\dt(\rho-\hat\rho) \ra + \lambda\la \eta-\hat\eta,\dt(\eta-\hat\eta) \ra \\
&= \delta\mathcal{D}_\LL( \mu_{\rho}-\hat\mu_\rho,\theta-\hat\theta,\mu_{\eta}-\hat\mu_\eta) + C(\delta)\norm{\na(\rho-\hat\rho)}_0^2 + C(\delta)\norm{\na(\eta-\hat\eta)}_0^2 + \norm{r_1}_{-1}^2 + \norm{r_4}_{0}^2 \\
& \leq \delta\mathcal{D}_\LL( \mu_{\rho}-\hat\mu_\rho,\theta-\hat\theta,\mu_{\eta}-\hat\mu_\eta) + C(\delta)\mathcal{W}_{\lambda}(\rho,\theta,\eta|\hat\rho,\hat\theta,\hat\eta)  + \norm{r_1}_{-1}^2 + \norm{r_4}_{0}^2.
\end{align*}

\subsection*{Full estimate}

Combination of the above paragraphs and choosing $\delta=1/6$ we find
\begin{align}
\ddt \mathcal{W}_{\lambda}(\rho,\theta,\eta|\hat\rho,\hat\theta,\hat\eta)&+ \tfrac{1}{2}\mathcal{D}_{\LL}(\mu_\rho,\theta,\mu_\eta|\hat\mu_\rho,\hat\theta,\hat\mu_\eta) \leq C\mathcal{W}_{\lambda}(\rho,\theta,\eta|\hat\rho,\hat\theta,\hat\eta) \label{eq:relentcollect}\\
&+ \la \partial\psi_\theta(\rho,\theta,\eta|\hat\rho,\hat\theta,\hat\eta), \dt\hat\theta\ra + \la \partial_\rho\psi(\rho,\theta,\eta|\hat\rho,\hat\theta,\hat\eta), \dt\hat\rho\ra \\
&+ \la \partial_\eta\psi(\rho,\theta,\eta|\hat\rho,\hat\theta,\hat\eta), \dt\hat\eta\ra \\
& + \la \psi_\rho(\rho_h,\theta_h,\eta_h)-\psi_\rho(\hat\rho_h,\hat\theta_h,\hat\eta_h),1 \ra^2 + \norm{\theta-\hat\theta}_0^2  \\
&+ C(\norm{r_1}_{-1}^2 + \norm{r_2}_{1}^2 + \norm{r_3}_{-1}^2 + \norm{r_4}_{0}^2 + \norm{r_5}_{0}^2).
\end{align}

Hence, the result follows by integration in time and application of the Gronwall lemma.
\section{Proof of Theorem \ref{thm:semistab}}
\label{sec:full_stab_semi}

Recall that $(|\rho_h|,\theta_h,|\eta_h|)$ as well as $(|\hat\rho_h|,\hat\theta_h,|\hat\eta_h|)$ are bounded in $L^\infty(\Omega\times(0,T))$ by $\max(C_\infty,c_\infty)$ and $\theta_h,\hat\theta_h \geq c_\infty$.
Using the Hessian representation of the relative exergy \eqref{eq:expexergy}, we find that 
\begin{align}
 c(\norm{\rho_h-\hat\rho_h}_1^2 &+ \norm{\theta_h-\hat\theta_h}_0^2 + \norm{\eta_h-\hat\eta_h}_1^2) \label{eq:rel_ent_equiv}\\
 &\leq  \mathcal{W}_{\lambda}(\rho_h,\theta_h,\eta_h|\hat\rho_h,\hat\theta_h,\hat\eta_h) \leq C(\norm{\rho_h-\hat\rho_h}_1^2 + \norm{\theta_h-\hat\theta_h}_0^2 + \norm{\eta_h-\hat\eta_h}_1^2). \notag
\end{align}
Using the equivalence, we can consider the second to sixth terms on the right-hand-side of \eqref{eq:relentcollect} and estimate
\begin{align*}
 &\la \partial\psi_\theta(\rho_h,\theta_h,\eta_h|\hat\rho_h,\hat\theta_h,\hat\eta_h), \dt\hat\theta_h\ra + \la \partial_\rho\psi(\rho_h,\theta_h,\eta_h|\hat\rho_h,\hat\theta_h,\hat\eta_h), \dt\hat\rho_h\ra \\
 &+ \la \partial_\eta\psi(\rho_h,\theta_h,\eta_h|\hat\rho_h,\hat\theta_h,\hat\eta_h), \dt\hat\eta_h\ra  + \la \psi_\rho(\rho_h,\theta_h,\eta_h)-\psi_\rho(\hat\rho_h,\hat\theta_h,\hat\eta_h),1 \ra^2 + \norm{\theta_h-\hat\theta_h}_0^2  \\
& \leq C(1+\norm{\dt\hat\rho_h}_{0,\infty}+\norm{\dt\hat\theta_h}_{0,\infty}+\norm{\dt\hat\eta_h}_{0,\infty})\mathcal{W}_{\lambda}(\rho_h,\theta_h,\eta_h|\hat\rho_h,\hat\theta_h,\hat\eta_h)
\end{align*}
where the constants $C$ depends on $C_\infty,c_\infty$. The final result then follows from an application of the Gronwall lemma.

\section{Full discrete approximation:}\label{sec:fulldiscrete}
Let us now consider a fully discrete approximation using backward differences in time, together with splitting into convex and concave parts for the phase-fields $\rho,\eta$.  First, we introduce the following notations.
\subsection*{Time discretisation}
We partition the time interval $[0,T]$ into uniformly with step size $\tau>0$ and introduce $\Itau:=\{0=t^0,t^1=\tau,\ldots, t^{n_T}=T\}$, where $n_T=\tfrac{T}{\tau}$ is the absolute number of time steps. We denote by $\Pi^1_c(\Itau), \Pi^0(\Itau)$ the spaces of continuous piecewise linear and piecewise constant functions on $\Itau$. For a function $g\in\Pi^1_c(\tau)$ we denote by $g^{n+1},g^n$ the evaluation at $t^{n+1}$ and $t^n$, while for functions $g\in\Pi^0(\Itau)$ $g^{n+1}$ is the constant value on the interval $(t^n,t^{n+1}]$. Finally, we introduce the time difference and the discrete-time derivative via
\begin{equation*}
    d^{n+1}g = g^{n+1} - g^n, \qquad d^{n+1}_\tau g = \frac{g^{n+1}-g^n}{\tau}.
\end{equation*}
Note that in the case of continuous piecewise linear function, the discrete time derivative coincides with the normal time derivative. We consider the following discretisation.
\begin{problem} \label{prob:ful}
    Let $(\rho_{0,h},\theta_{0,h},\eta_{0,h}) \in \Vh\times\Vh^+\times\Vh$ be given. Find $(\rho_{h},\theta_{h},\eta_{h})\in \Pi^1_c(\Itau;\Vh\times\Vh^+\times\Vh)$ and $(\mu_{\rho,h},\mu_{\eta,h}) \in  \Pi^0(\Itau;\Vh\times\Vh)$  such that $(\rho_{h}^0,\theta_{h}^0,\eta_{h}^0)=(\rho_{0,h},\theta_{0,h},\eta_{0,h})$ and
such that for all $v_{1,h},v_{2,h},\xi_{h}, w_{1,h}, w_{2,h} \in  \Vh\times\Vh\times\Vh\times\Vh$ and for all $n\geq 0$, there holds
\begin{align}
&\la d_\tau^{n+1}\rho_h,v_{1,h} \ra + \la \LL_{11}\nabla \mu_{\rho,h}^{n+1}-\LL_{12}\nabla\theta_h^{n+1}+\mu_{\eta,h}^{n+1}\LL_{13},\nabla v_{1,h} \ra =0, \label{eq:pg1}\\   
&\la \mu_{\rho,h}^{n+1},v_{2,h} \ra - \gamma_\rho\la \nabla\rho_h^{n+1},\nabla v_{2,h} \ra - \la \partial_\rho \psi(\rho_h,\theta_h^{n+1},\eta_h),v_{2,h} \ra = 0, \label{eq:pg2}\\
&\la d_\tau^{n+1} e(\rho_h,\theta_h,\eta_h),\xi_h \ra + \la \LL_{12}\nabla \mu_{\rho,h}^{n+1}- \LL_{22}\nabla\theta_h^{n+1}+ \mu_{\eta,h}^{n+1}\LL_{23},\nabla \xi_h \ra =0,\label{eq:pg3} \\   
&\la d_\tau^{n+1}\eta_{h},w_{1,h} \ra + \la\LL_{13}\cdot\nabla\mu_{\rho,h}^{n+1}-\LL_{23}\cdot\nabla\theta_h^{n+1}+ \LL_{33}\mu_{\eta,h}^{n+1},w_{1,h} \ra = 0, \label{eq:pg4}\\
&\la \mu_{\eta,h}^{n+1},w_{2,h} \ra - \gamma_\eta\la \nabla\eta_h^{n+1},\nabla w_{2,h} \ra - \la \partial_\eta \psi(\rho_h,\theta_h^{n+1},\eta_h),w_{2,h} \ra = 0. \label{eq:pg5}
\end{align}
\end{problem}

We employ the convex-concave splitting for $\partial_\rho \psi, \partial_\eta\psi$, i.e.
\begin{align*}
  \partial_\rho \psi(\rho_h,\theta_h^{n+1},\eta_h)&=\partial_\rho\psi ^{vex}(\rho_h^{n+1},\theta_h^{n+1},\eta_h^{n+1}) + \partial_\rho\psi^{cav}(\rho_h^{n},\theta_h^{n+1},\eta_h^{n}) \\
  \partial_\eta \psi(\rho_h,\theta_h^{n+1},\eta_h) &= \partial_\eta\psi^{vex}(\rho_h^{n+1},\theta_h^{n+1},\eta_h^{n+1}) + \partial_\eta\psi^{cav}(\rho_h^{n},\theta_h^{n+1},\eta_h^{n}).
\end{align*}

\begin{remark}
Note that the structural properties are independent of the concrete time discretisation of $\LL$. One can evaluate $\LL$ at any linear combination of the new and old-time levels to obtain a first-order scheme.
\end{remark}

\begin{lemma}
   Let (A0)--(A5) hold, cf. Assumption \ref{ass} and $(\rho_h,\mu_{\rho,h},\theta_h,\eta_h,\mu_{\eta,h})$ a solution of Problem \ref{prob:ful}. Then this solution is mass-, energy-conservative and entropy-productive, i.e.
   \begin{align*}
    \la \rho_{h}^{n+1}, 1\ra &= \la \rho_{h}^0, 1\ra, \quad  \la e_{h}^{n+1}, 1\ra =\la e_{h}^{0}, 1\ra, \qquad \forall n\geq0, \\
   \la  s_{h}^{n+1},1\ra &=\la  s_{h}^{0},1\ra + \sum_{k=0}^{n_T-1}\mathcal{D}^{k}_{num}(\rho_h,\theta_h,\eta_h) + \tau\sum_{k=0}^{n_T-1}\mathcal{D}_\LL(\mu_{\rho,h}^{k+1},\theta_{h}^{k+1},\mu_{\eta,h}^{k+1}) \qquad \forall n\geq0.
   \end{align*}
   with numerical dissipation
   \begin{align*}
       \mathcal{D}^{n}_{num}(\rho_h,\theta_h,\eta_h) := & \frac{\gamma_\rho}{2}\norm{\nabla d^{n+1}\rho_h}^2 + \frac{\gamma_\eta}{2}\norm{\nabla d^{n+1}\eta_h}^2 - \tfrac{1}{2}\la \partial_{\theta\theta}\psi^{n,\xi_1,n}d^{n+1}\theta_h,d^{n+1}\theta_h \ra \\
 - & \tfrac{1}{2}\la (H_{(\rho_h,\eta_h)}^{cav,\chi_1,{n+1},\chi_2}-H^{vex,\chi_3,{n+1},\chi_4}_{(\rho_h,\eta_h)})d^{n+1}(\rho_h,\eta_h)^\top,d^{n+1}(\rho_h,\eta_h)^\top \ra \geq 0
   \end{align*}
   for some $\xi_1\in(\theta_h^n,\theta_h^{n+1}) , \chi_1,\chi_3\in (\rho_h^n,\rho_h^{n+1}) , \chi_2,\chi_4\in (\eta_h^n,\eta_h^{n+1}). $
\end{lemma}
\begin{proof}
Mass and energy conservation follows directly by setting $v_{1,h}=1=\xi_h$ in \eqref{eq:pg1}, \eqref{eq:pg3}. For entropy dissipation, we compute as follows
\begin{align*}
\la s^{n+1}-s^n, 1\ra =& \la e^{n+1}-e^n, \theta^{n+1}\ra + \la \theta_h ^{n+1}-\theta_h^n,e^n \ra - \la \psi^{n+1}-\psi^n,1\ra  \\
&- \frac{\gamma_\rho}{2}\norm{\nabla\rho^{n+1}}_0^2 + \frac{\gamma_\rho}{2}\norm{\nabla\rho^{n}}_0^2  -\frac{\gamma_\rho}{2}\norm{\nabla\eta^{n+1}}_0^2 + \frac{\gamma_\rho}{2}\norm{\nabla\eta^{n}}_0^2 \\
=& \la e^{n+1}-e^n, \theta^{n+1}\ra + \la \theta_h ^{n+1}-\theta_h^n,e^n \ra- \la \psi^{n+1} - \psi^n, 1 \ra \\
&- \la \partial_\rho \psi(\rho_h,\theta_h^{n+1},\eta_h),\rho_h^{n+1}-\rho_h^n \ra - \gamma_\rho\la \nabla \rho_h^{n+1}, \nabla(\rho_h^{n+1}- \rho_h^n)\ra \\
& - \la \partial_\eta \psi(\rho_h,\theta_h^{n+1},\eta_h),\eta_h^{n+1}- \eta_h^n \ra - \gamma_\eta\la \nabla \eta_h^{n+1},\nabla( \eta_h^{n+1}- \eta_h^n)\ra \\
& + \la \partial_\rho \psi(\rho_h,\theta_h^{n+1},\eta_h),\rho_h^{n+1}- \rho_h^n \ra + \la \partial_\eta\psi(\rho_h,\theta_h^{n+1},\eta_h),\eta_h^{n+1}- \eta_h^n\ra\\
&+ \frac{\gamma_\rho}{2}\norm{\nabla(\rho_h^{n+1}- \rho_h^n)}^2 + \frac{\gamma_\eta}{2}\norm{\nabla(\eta_h^{n+1}- \eta_h^n)}^2 . 
\end{align*}
Inserting $v_{2,h}=\rho_h^{n+1}-\rho_h^n$ and $w_{2,h}=\eta_h^{n+1}- \eta_h^n$ yield
\begin{align*}
 \la s^{n+1}-s^n, 1\ra \geq& \la e^{n+1}-e^n, \theta^{n+1}\ra - \la \mu_{\rho,h}^{n+1},\rho_h^{n+1}-\rho_h^n \ra - \la \mu_{\eta,h}^{n+1},\eta_h^{n+1}-\eta_h^n \ra  \\
&- \la \psi(\rho_h^n,\theta_h^{n+1},\eta_h^n) - \psi(\rho_h^n,\theta_h^{n},\eta_h^n) - \partial_\theta\psi(\rho_h^n,\theta_h^{n},\eta_h^n)(\theta^{n+1}-\theta^n),1 \ra \\
& - \la \psi(\rho_h^{n+1},\theta_h^{n+1},\eta_h^{n+1})- \psi(\rho_h^n,\theta_h^{n+1},\eta_h^n)  -  \partial_\rho \psi(\rho_h,\theta_h^{n+1},\eta_h)(\rho_h^{n+1}- \rho_h^n) \\
& -  \partial_\eta \psi(\rho_h,\theta_h^{n+1},\eta_h)(\eta_h^{n+1}- \eta_h^n),1\ra  + \frac{\gamma_\rho}{2}\norm{\nabla(\rho_h^{n+1}- \rho_h^n)}^2 + \frac{\gamma_\eta}{2}\norm{\nabla(\eta_h^{n+1}- \eta_h^n)}^2 .
\end{align*}
Using the convexity and concavity properties of $\psi$, i.e. (A3) and the convex-concave splitting we find
\begin{align*}
\la s^{n+1}-s^n, 1\ra \geq & - \la \psi(\rho_h^n,\theta_h^{n+1},\eta_h^n) - \psi(\rho_h^n,\theta_h^{n},\eta_h^n) - \partial_\theta\psi(\rho_h^n,\theta_h^{n},\eta_h^n)(\theta^{n+1}-\theta^n),1 \ra \\   
&-\la \psi(\rho_h^{n+1},\theta_h^{n+1},\eta_h^{n+1})- \psi(\rho_h^n,\theta_h^{n+1},\eta_h^n) - \partial_\rho \psi(\rho_h,\theta_h^{n+1},\eta_h)(\rho_h^{n+1}- \rho_h^n) \\
 & \quad-  \partial_\eta \psi(\rho_h,\theta_h^{n+1},\eta_h)(\eta_h^{n+1}- \eta_h^n),1\ra \\
 =& - \tfrac{1}{2}\la \partial_{\theta\theta}\psi^{n,\xi_1,n}d^{n+1}\theta,d^{n+1}\theta \ra \\
 &-  \tfrac{1}{2}\la (H_{(\rho,\eta)}^{cav,\chi_1,{n+1},\chi_2}-H^{vex,\chi_3,{n+1},\chi_4}_{(\rho,\eta)})(d^{n+1}\rho,d^{n+1}\eta)^\top,(d^{n+1}\rho,d^{n+1}\eta)^\top \ra \\
  \geq & 0.
\end{align*}
Therefore, obtained by direct estimation
\begin{align*}
 \la s^{n+1}-s^n, 1\ra &\geq \la e^{n+1}-e^n, \theta^{n+1}\ra - \la \mu_{\rho,h}^{n+1},\rho_h^{n+1}-\rho_h^n \ra - \la \mu_{\eta,h}^{n+1},\eta_h^{n+1}- \eta_h^n \ra  .
 \end{align*}
Inserting $v_{1,h}=\mu_{\rho,h}^{n+1}, \xi_h=\theta^{n+1}_h, w_{1,h}=\mu_{\eta,h}^{n+1}$ yields
\begin{align*}
\la s^{n+1}-s^n,1\ra = \mathcal{D}_{num}^{n+1}(\phi_h,\theta_h,\eta_h) +  \tau\mathcal{D}_\LL( \mu_{\rho,h}^{n+1},\theta_h^{n+1},\mu_{\eta,h}^{n+1}).
\end{align*}
The statement of the theorem is followed by a summation over the time steps.
\end{proof}

\begin{remark}
  The existence of discrete solutions, as well as discretisation parameter independent bounds, is a more delicate topic. In the future, we like to consider this in more detail. However, for a suitable form of the Helmholtz free energy $\psi$, this should be possible by a suitable fixed point argument. 
\end{remark}

In a similar spirit as in the semi-discrete case, we can employ the relative entropy to derive a stability result with respect to a perturbed system. For given $(\hat\rho_h,\hat\theta_h,\hat\eta_h)\in\Pi^1_c(\Itau;\Vh\times\Vh^+\times\Vh)$ and $(\hat\mu_{\rho,h},\hat\mu_{\eta,h})\in \Pi^0(\Itau;\Vh\times\Vh)$ we define the residuals $r_{i,h}\in \Pi^0(\Itau;\Vh), i=1,\ldots, 4$ for all test function $(v_{1,h},v_{2,h},\xi_h,w_{1,h},w_{2,h})\in \Vh\times\Vh\times\Vh^+\times\Vh\times\Vh$ via the following perturbed system
\begin{align}
&\la d_\tau^{n+1}\hat\rho_h,v_{1,h} \ra + \la \LL_{11}\nabla \hat\mu_{\rho,h}^{n+1}- \LL_{12}\nabla\hat\theta_h^{n+1}+ \hat\mu_{\eta,h}^{n+1}\LL_{13},\nabla v_{1,h} \ra =\la r_{1,h}^{n+1},v_{1,h} \ra, \label{eq:ppg1}\\   
&\la \hat\mu_{\rho,h}^{n+1},v_{2,h} \ra - \gamma_\rho\la \nabla\hat\rho_h^{n+1},\nabla v_{2,h} \ra - \la \partial_\rho \psi(\hat\rho_h,\hat\theta_h^{n+1},\hat\eta_h),v_{2,h} \ra = \la r_{2,h}^{n+1},v_{2,h} \ra, \label{eq:ppg2}\\
&\la d_\tau^{n+1} e(\hat\rho_h,\hat\theta_h,\hat\eta_h),\xi_h \ra + \la \LL_{12}\nabla \hat\mu_{\rho,h}^{n+1}- \LL_{22}\nabla\hat\theta_h^{n+1}+ \hat\mu_{\eta,h}^{n+1}\LL_{23},\nabla \xi_h \ra =\la r_{3,h}^{n+1},\xi_{h} \ra,\label{eq:ppg3} \\   
&\la d_\tau^{n+1}\hat\eta_{h},w_{1,h} \ra + \la\LL_{13}\cdot\nabla\hat\mu_{\rho,h}^{n+1}-\LL_{23}\cdot\nabla\hat\theta_h^{n+1}+\LL_{33}\hat\mu_{\eta,h}^{n+1},w_{1,h} \ra = \la r_{4,h}^{n+1},w_{1,h} \ra, \label{eq:ppg4}\\
&\la \hat\mu_{\eta,h}^{n+1},w_{2,h} \ra - \gamma_\eta\la \nabla\hat\eta_h^{n+1},\nabla w_{2,h} \ra - \la \partial_\eta \psi(\hat\rho_h,\hat\theta_h^{n+1},\hat\eta_h),w_{2,h} \ra = \la r_{5,h}^{n+1},w_{2,h} \ra. \label{eq:ppg5}
\end{align}

With the same ansatz as in the semi-discrete case, we can deduce the following result.

\begin{theorem}\label{thm:stabfullgron}
Let $(\rho_h,\mu_{\rho,h},\theta_h,\eta_h,\mu_{\eta,h})$ be a solution of Problem \ref{prob:ful} and $(\hat\rho_h,\hat\mu_{\rho,h},\hat\theta_h,\hat\eta_h,\hat\mu_{\eta,h})$ defining the residuals via \eqref{eq:ppg1}-\eqref{eq:ppg5}. Furthermore, we assume the following uniform bounds
\begin{align*}
  \norm{(\rho_h,\theta_h,\eta_h)}_{L^\infty(\Omega_T)} &+ \norm{(\hat\rho_h,\hat\theta_h,\hat\eta_h)}_{L^\infty(\Omega_T)} + \norm{(d_\tau^{n+1}\hat\rho_h,d_\tau^{n+1}\hat\theta_h,d_\tau^{n+1}\hat\eta_h)}_{L^\infty(\Omega_T)} \leq C_\infty, \\  
  \theta_h, \hat\theta_h &\geq c_\infty \qquad \text{ for every } t\in [0,T], 
\end{align*}
with positive constants $c_\infty,C_\infty$ independent of $h,\tau$. Then the following result holds  
\begin{align}
 \mathcal{W}_{\lambda}(\rho_h,\theta_h,\eta_h|\hat\rho_h,\hat\theta_h,\eta_h)(t^{n+1}) &+ \sum_{k=0}^n\frac{\tau}{2}\mathcal{D}_\LL(\mu_{\rho,h}^{k+1}-\hat\mu_{\rho,h}^{n+1},\theta_h^{k+1}-\hat\theta_h^{k+1},\mu_{\eta,h}^{k+1}-\hat\mu_{\eta,h}^{k+1}) \notag\\
 &+ \sum_{k=0}^n\tilde{\mathcal{D}}_{num}^{k+1}(\rho_h-\hat\rho_h,\theta_h-\hat\theta_h,\eta_h-\hat\eta_h)\label{eq:relentfullcoll}\\
 \leq C\mathcal{W}_{\lambda}(\rho_h,\theta_h,\eta_h|\hat\rho_h,\hat\theta_h,\hat\eta_h)(0)& + C\tau^2 \notag\\
+\tau\sum_{k=0}^n C(\LL)(\norm{r_{1,h}^{k+1}}_{-1}^2 +& \norm{r_{2,h}^{k+1}}_1^2+ \norm{r_{3,h}^{k+1}}_{-1}^2 + \norm{r_{4,h}^{k+1}}_0^2 + \norm{r_{5,h}^{k+1}}_0^2). \notag
\end{align} 
\end{theorem}
\begin{proof}
 The main ideas of the proof are analogous to the semi-discrete case. For completeness, we give the proof in the appendix.  Note that the main new difficulty lies in the optimal analysis of the non-conformity error of the time discretisation.
\end{proof}

As in the semi-discrete case, the uniform bounds on the discrete solutions are typically not available independent of $h,\tau$. 
\begin{remark}
 In principle, one could perform the error analysis based on the above stability result. Therefore, one has to identify the perturbed solution as a suitable projection of a sufficiently regular solution of the continuous problem. The residuals can then be identified and estimated using projection errors and the relative entropy, cf. \cite{Brunk23,Brunk23NS} for the error analysis of another system in this spirit. 
\end{remark}

\section{Numerical tests}\label{sec:num}
For confirmation of our theoretical findings, we now present some computational results and convergence rates for a typical test problem.

\subsection*{Model problem}

We consider the domain $\Omega=(0,1)^2$, which is identified with the two-torus $\mathbb{T}^2$, i.e., \eqref{eq:noniso1}--\eqref{eq:noniso5} is complemented by periodic boundary conditions. 
We further set $T_f=0.16$ for the final time and choose the  following initial conditions
\begin{equation*}
 \rho_0 =\eta_0 = \tfrac{1}{2}+\tfrac{1}{100}\cos(2\pi x)\cos(2\pi y), \quad 
 \theta_0 = 1 + \tfrac{6}{10}\sin(2\pi x)\sin(2\pi y).
\end{equation*}
With the Helmholtz potential
\begin{align*}
 \Psi(\rho,\theta,\eta) =& \log(\theta) + \frac{\gamma_\rho}{2}\snorm{\nabla\rho}^2 + \frac{\gamma_\eta}{2}\snorm{\nabla\eta}^2 +  (C_1\theta-C_2)\rho^2(1-\rho)^2 + (D_1\theta-D_2)[\rho^2 \\
 &+ 6(1-\rho)(\eta^2 + (1-\eta)^2) - 4(2-\rho)(\eta^3 + (1-\eta)^3)  + 3(\eta^2 + (1-\eta)^2 )^2 ]
\end{align*}
which implies that in the original variables, we consider the Helmholtz free energy
\begin{align*}
f(\rho,T,\eta) =& -T\log(T) + \frac{\gamma_\rho T}{2}\snorm{\nabla\rho}^2 + \frac{\gamma_\eta T}{2}\snorm{\nabla\eta}^2 +  (C_1-C_2T)\rho^2(1-\rho)^2 + (D_1-D_2T)[\rho^2 \\
 &+ 6(1-\rho)(\eta^2 + (1-\eta)^2) - 4(2-\rho)(\eta^3 + (1-\eta)^3)  + 3(\eta^2 + (1-\eta)^2 )^2 ]   .
\end{align*}
For the diffusion matrix, we choose 
\begin{align*}
    \LL = \begin{pmatrix}
        \tfrac{1}{10}\Id_{2\times 2} & 0\Id_{2\times 2} & 0\Id_{2\times 1} \\
        0\Id_{2\times 2} &  \tfrac{1}{10}\Id_{2\times 2} & 0\Id_{2\times 1} \\
        0\Id_{1\times 2} & 0\Id_{1\times 2} & 10
    \end{pmatrix}
\end{align*}
The remaining model parameters are chosen as $C_1=2, C_2=1, D_1=0.141, D_2=0.062, \gamma_\rho=\gamma_\eta=0.001$. 

\subsection*{Convergence results}

We now discuss the convergence rates observed in our computations.
Since no analytical solution is available, the discretisation error is estimated by comparing the computed solutions $(\phi_{h,\tau},\mu_{\rho,h,\tau},\theta_{h,\tau},\eta_{h,\tau},\mu_{\eta,h,\tau})$ with those computed on uniformly refined grids, i.e. $(\phi_{h/2,\tau/2},\mu_{\rho,h/2,\tau/2},\theta_{h/2,\tau/2},\eta_{h/2,\tau/2},\mu_{\eta,h/2,\tau/2})$.
The error quantities for the fully-discrete scheme, which we report in the following, are defined as 
\begin{align*}
e_{h,\tau} &= \norm*{\rho_{h,\tau} - \rho_{h/2,\tau/2}}_{L_\tau^\infty(H^1)}^2 + \norm*{\theta_{h,\tau} - \theta_{h/2,\tau/2}}_{L_\tau^\infty(L^2)}^2 + \norm*{\eta_{h,\tau} - \eta_{h/2,\tau/2}}_{L_\tau^\infty(H^1)}^2 \\
&\quad+ \norm*{\mu_{\rho,h,\tau} - \mu_{\rho,h/2,\tau/2}}_{L^2(H^1)}^2+ \norm*{\theta_{h,\tau} - \theta_{h/2,\tau/2}}_{L^2(H^1)}^2 + \norm*{\mu_{\eta,h,\tau} - \mu_{\eta,h/2,\tau/2}}_{L^2(L^2)}^2.
\end{align*}
Here, we employed the time-discrete norm
\begin{align*}
 \norm*{g_{h,\tau} - g_{h/2,\tau/2}}_{L_\tau^\infty(X)}^2:=\max_{t^n\in\Itau} \norm{g_{h,\tau}(t^n) - g_{h/2,\tau/2}(t^n)}_X^2.  
\end{align*}
It should be noted that in the case of uniform bounds on the solutions and the perturbed system, this is equivalent to the error measure which is provided by the relative entropy by Theorem \ref{thm:stabfullgron}.
In Table~\ref{tab:rates_time_chns1} and Table~\ref{tab:rates_time_chns2}, we report the results of our computations obtained on a sequence of uniformly refined meshes with mesh size $h_k=2^{-k-1}$, $k=0,\ldots,6$, and time steps $\tau_k =0.001\cdot h_k$. 
We observe second-order convergence for the squared norms of the errors in all solution components, which yields, as expected, a first-order scheme with respect to the energy norm. 

\begin{table}[htbp!]
\centering
\small
\caption{Part 1: Errors and experimental orders of convergence. \label{tab:rates_time_chns1}} 
\begin{tabular}{|c||c|c|c|c|c|c|c|c|c|c}
$ k $ & $ e_{h,\tau} $  &  eoc & $e^\rho_{h,\tau}$ & eoc & $e^\theta_{h,\tau}$ & eoc & $e^{\eta}_{h,\tau}$ & eoc  \\
\hline
$ 0 $ & $4.66 \cdot 10^{-1}$  &   ---    & $3.08 \cdot 10^{-1}$  &   ---  & $9.32\cdot 10^{-3}$ &  --- & $9.29\cdot 10^{-3}$ &  --- \\
$ 1 $ & $2.74 \cdot 10^{-1}$  &   0.76   & $2.13 \cdot 10^{-1}$  & 0.53   & $8.03\cdot 10^{-4}$ & 3.54 & $5.00\cdot 10^{-3}$ &  0.89 \\
$ 2 $ & $8.97 \cdot 10^{-2}$  &   1.61   & $7.29 \cdot 10^{-2}$  & 1.55   & $5.55\cdot 10^{-5}$ & 3.86 & $1.50\cdot 10^{-3}$ &  1.74 \\
$ 3 $ & $2.69 \cdot 10^{-2}$  &   1.74   & $2.26 \cdot 10^{-2}$  & 1.69   & $3.64\cdot 10^{-6}$ & 3.93 & $3.99\cdot 10^{-4}$ &  1.91 \\
$ 4 $ & $7.89 \cdot 10^{-3}$  &   1.77   & $6.79 \cdot 10^{-3}$  & 1.73   & $2.75\cdot 10^{-7}$ & 3.73 & $1.02\cdot 10^{-4}$ &  1.97 \\
$ 5 $ & $2.19 \cdot 10^{-3}$  &   1.85   & $1.91 \cdot 10^{-3}$  & 1.83   & $3.53\cdot 10^{-8}$ & 2.96 & $2.58\cdot 10^{-5}$ &  1.99 \\
\end{tabular}
\end{table}

\begin{table}[htbp!]
\centering
\small
\caption{Part 2: Errors and experimental orders of convergence. \label{tab:rates_time_chns2}} 
\begin{tabular}{|c||c|c|c|c|c|c|}
$ k $ & $ e^{\mu_\rho}_{h,\tau} $  &  eoc & $e^{\nabla\theta}_{h,\tau}$ & eoc & $e^{\mu_\eta}_{h,\tau}$ & eoc  \\
\hline
$ 0 $   & $8.44 \cdot 10^{-3}$  & ---    & $1.31\cdot 10^{-1}$ &  --- & $2.94\cdot 10^{-6}$ &  ---   \\
$ 1 $   & $3.92 \cdot 10^{-3}$  & 1.11   & $5.15\cdot 10^{-2}$ & 1.34 & $1.02\cdot 10^{-6}$ &  1.52  \\
$ 2 $   & $1.20 \cdot 10^{-3}$  & 1.70   & $1.41\cdot 10^{-2}$ & 1.87 & $1.62\cdot 10^{-7}$ &  2.66  \\
$ 3 $   & $3.47 \cdot 10^{-4}$  & 1.80   & $3.60\cdot 10^{-3}$ & 1.97 & $2.49\cdot 10^{-8}$ &  2.70  \\
$ 4 $   & $9.46 \cdot 10^{-5}$  & 1.86   & $9.04\cdot 10^{-4}$ & 1.99 & $4.08\cdot 10^{-9}$ &  2.50  \\
$ 5 $   & $2.53 \cdot 10^{-5}$  & 1.92   & $2.26\cdot 10^{-4}$ & 2.00 & $8.93\cdot 10^{-10}$ & 2.30  \\
\end{tabular}
\end{table}

\subsection{Applied experiment}

In this section, we consider the influence of the temperature on the evolution of the coupled phase-field system. Therefore, we consider the same parameters as before but with a slightly different diffusion matrix
\begin{align*}
    \LL = \begin{pmatrix}
        \tfrac{1}{10}\Id_{2\times 2} & 0\Id_{2\times 2} & 0\Id_{2\times 1}\\
        0\Id_{2\times 2} &  \tfrac{1}{1000}\Id_{2\times 2} & 0\Id_{2\times 1} \\
        0\Id_{1\times 2} & 0\Id_{1\times 2} & 1000.
    \end{pmatrix}
\end{align*}
and the initial data
\begin{align*}
w_1(x,y) &= \sqrt{ (x-\tfrac{11}{40})^2 + (y-\tfrac{1}{2})^2} - \tfrac{1}{10}, \quad w_2(x,y)= \sqrt{(x-\tfrac{5}{8})^2 + (y-\tfrac{1}{2})^2}- \tfrac{1}{5}, \\
 \rho_0(x,y) &= 1-\tfrac{1}{2}\tanh\left(\frac{\max(-w_1,w_2)}{\sqrt{2\gamma_\rho}}\right) -\tfrac{1}{2}\tanh\left(\frac{\max(-w_2,w_1)}{\sqrt{2\gamma_\rho}}\right), \\
 \eta_0(x,y) &= \tfrac{1}{2}\rho_0(x,y)\chi(x\geq \tfrac{70}{200}) - \tfrac{1}{2}\rho_0(x,y)\chi(x < \tfrac{70}{200}), \\
 \theta_{0,A}(x,y)&=1, \quad \theta_{0,B}(x,y)=1 + \tfrac{3}{5}\sin(2\pi y), \quad \theta_{0,C}(x,y)= 1- \tfrac{3}{5}\sin(2\pi y)
\end{align*}
The three different initial temperature fields (A),(B), and (C) defined above are chosen in order to allow us to deduce the effects of the temperature in the coupled system. The evolution up to time $t=10$ is shown in Figures \ref{fig:evorho}--\ref{fig:evotheta}. We can see that the conserved phase-field, cf. Figure \ref{fig:evorho}, $\rho$ evolves to the same final state at $t=10$, but with different speed. While the non-conserved phase-field, cf. Figure \ref{fig:evoeta}, we see that experiments (A), and (B) behave quite similarly, in (C) we see a completely different behaviour. The inverse temperature $\theta$,  cf. Figure \ref{fig:evotheta}, evolves in all three cases quite fast to an equilibrium value, which is determined by the initial internal energy. In the sintering context, one can obtain information on the two different grains by the meta-quantity $\rho(2\eta-1)$, which is $1$ for one grain and $-1$ for the other grain. The evolution is depicted in  Figure \ref{fig:evophase}, where we can clearly see that one grain absorbs the other grain over time, and this process is only finished in experiment (C), while in the other two cases, this is still an active process. 

In conclusion, the temperature profile allows for change in the time scales of the process and enhances or suppresses the evolution.

\begin{figure}[htbp!]
\centering
\footnotesize
\begin{tabular}{cccc}
    \includegraphics[trim={1.5cm 0.0cm 3.0cm 0.5cm},clip,scale=0.32]{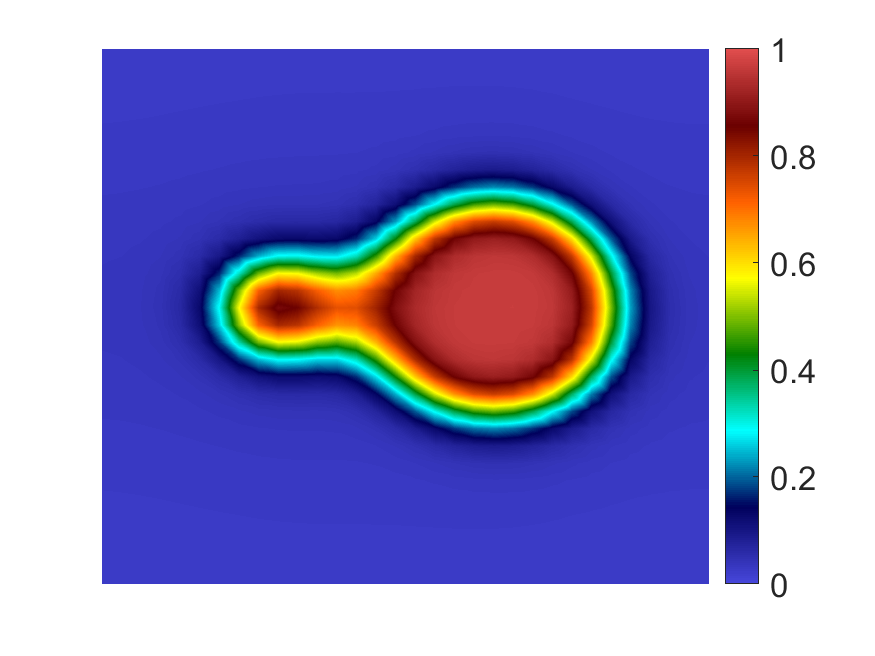} 
    &
    \includegraphics[trim={2.0cm 0.0cm 3.0cm 0.5cm},clip,scale=0.32]{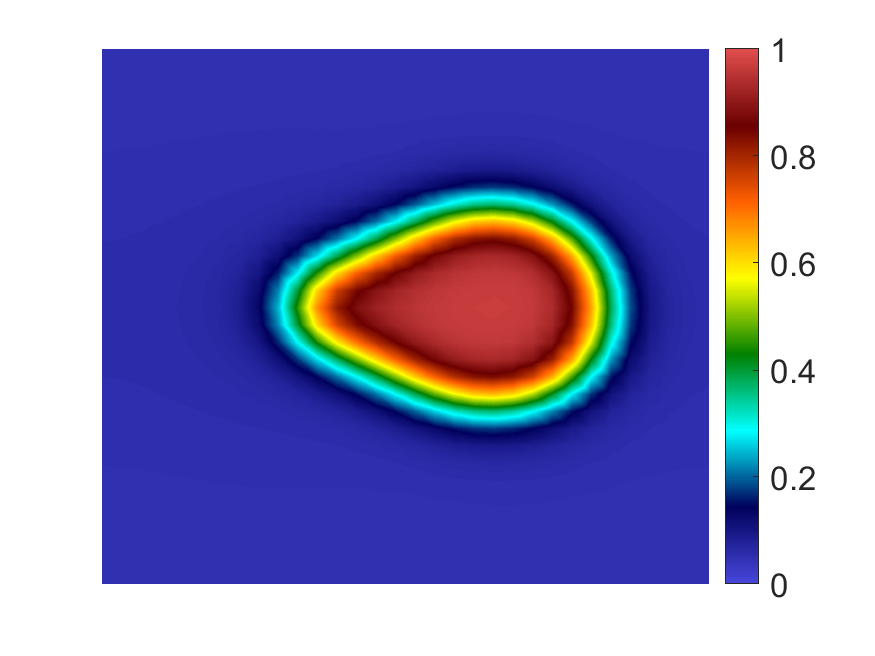}  
    &
    \includegraphics[trim={2.0cm 0.0cm 3.0cm 0.5cm},clip,scale=0.32]{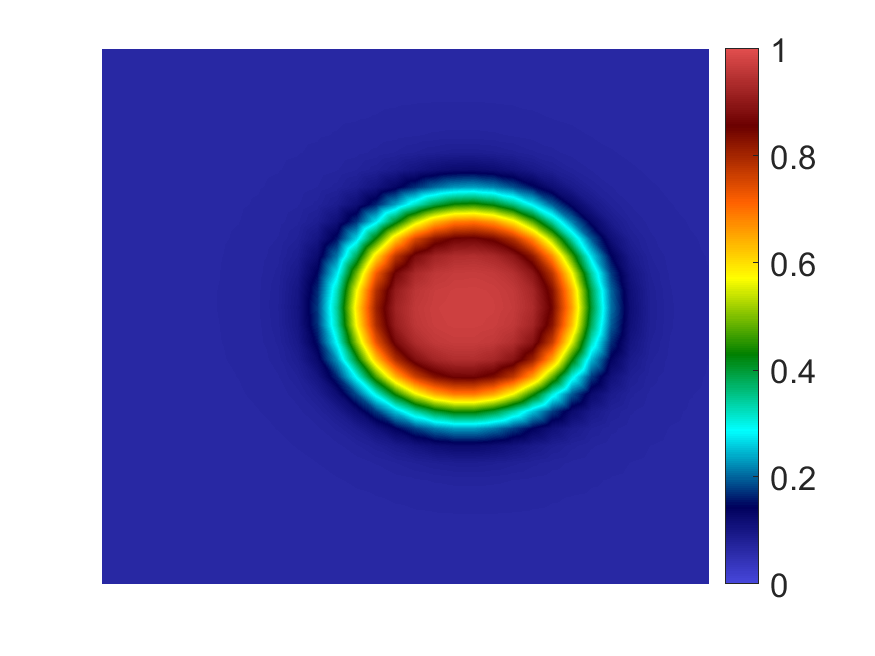}
    &
    \includegraphics[trim={2.0cm 0.0cm 0.0cm 0.5cm},clip,scale=0.32]{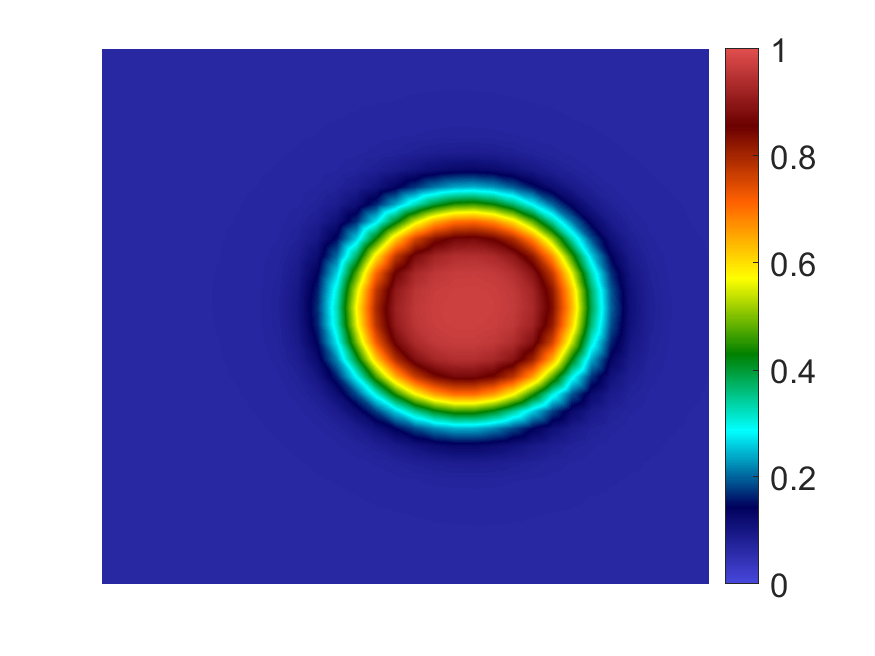} \\[-0.5em]
    \includegraphics[trim={1.5cm 0.0cm 3.0cm 0.5cm},clip,scale=0.32]{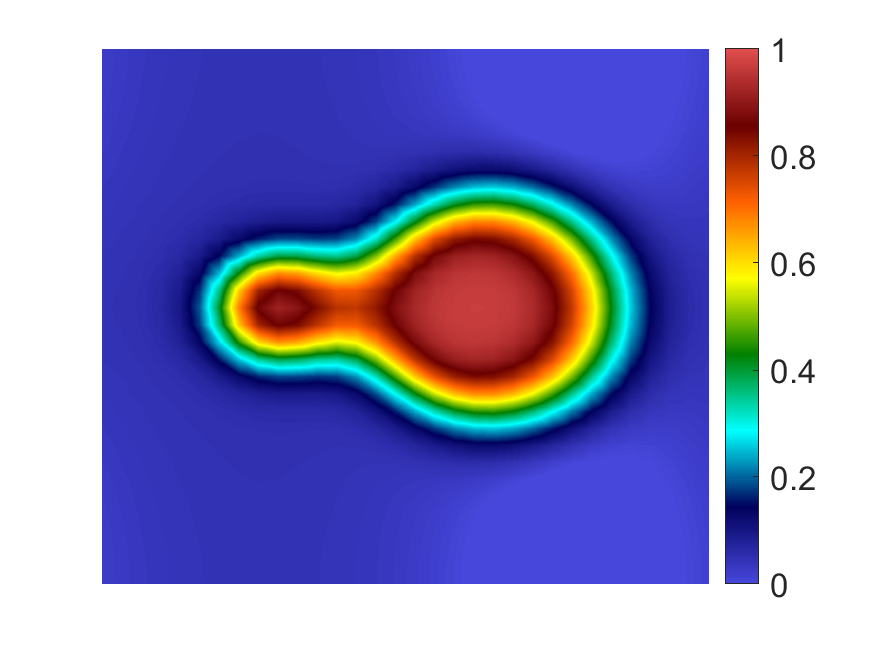} 
    &
    \includegraphics[trim={2.0cm 0.0cm 3.0cm 0.5cm},clip,scale=0.32]{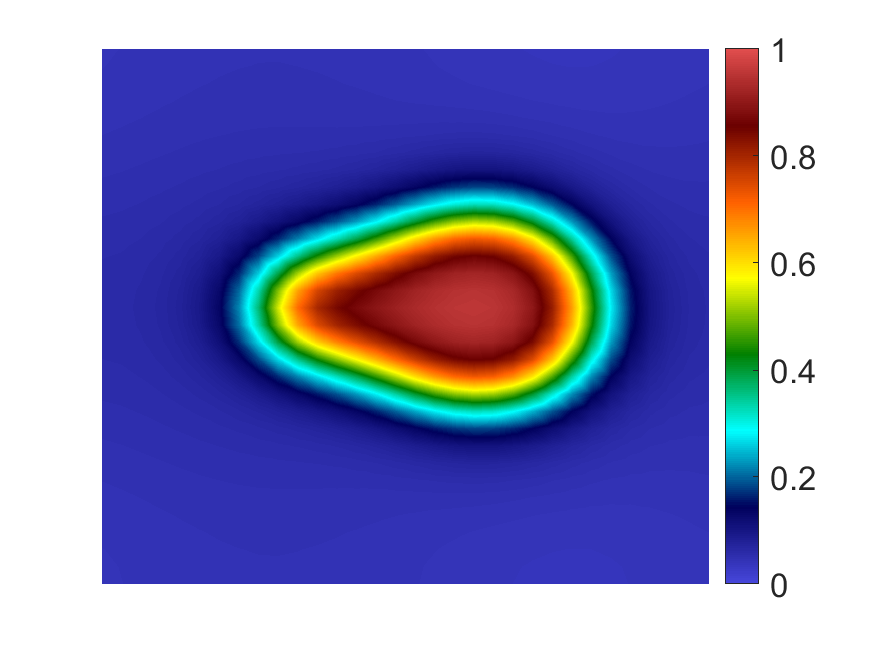}  
    &
    \includegraphics[trim={2.0cm 0.0cm 3.0cm 0.5cm},clip,scale=0.32]{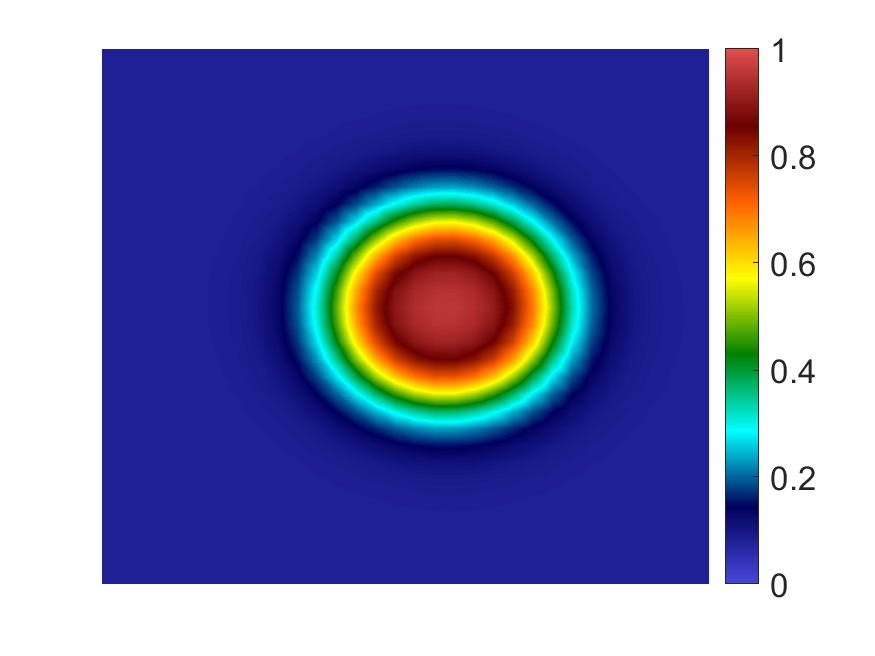}
    &
    \includegraphics[trim={2.0cm 0.0cm 0.0cm 0.5cm},clip,scale=0.32]{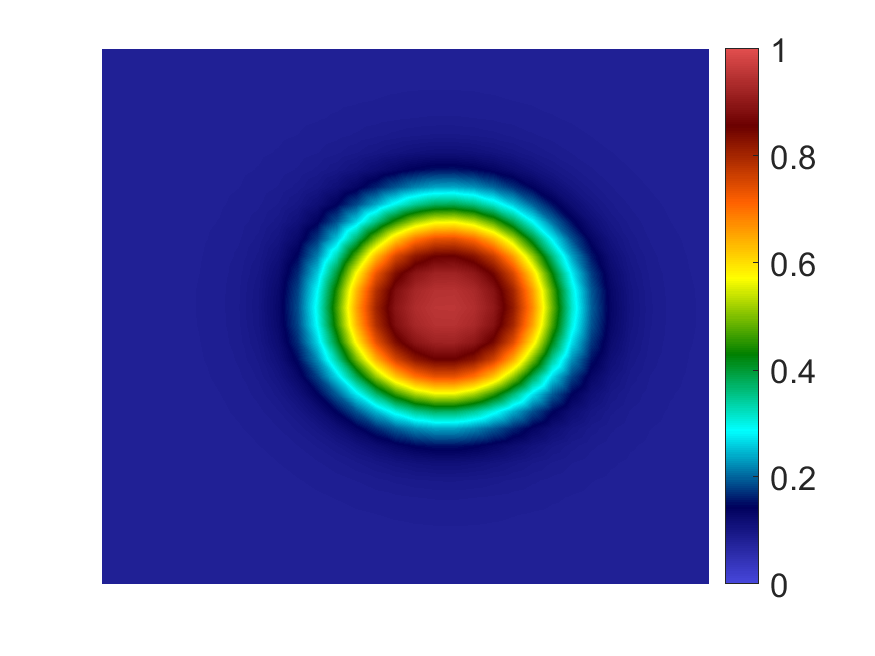} \\[-0.5em]
    \includegraphics[trim={1.5cm 0.0cm 3.0cm 0.5cm},clip,scale=0.32]{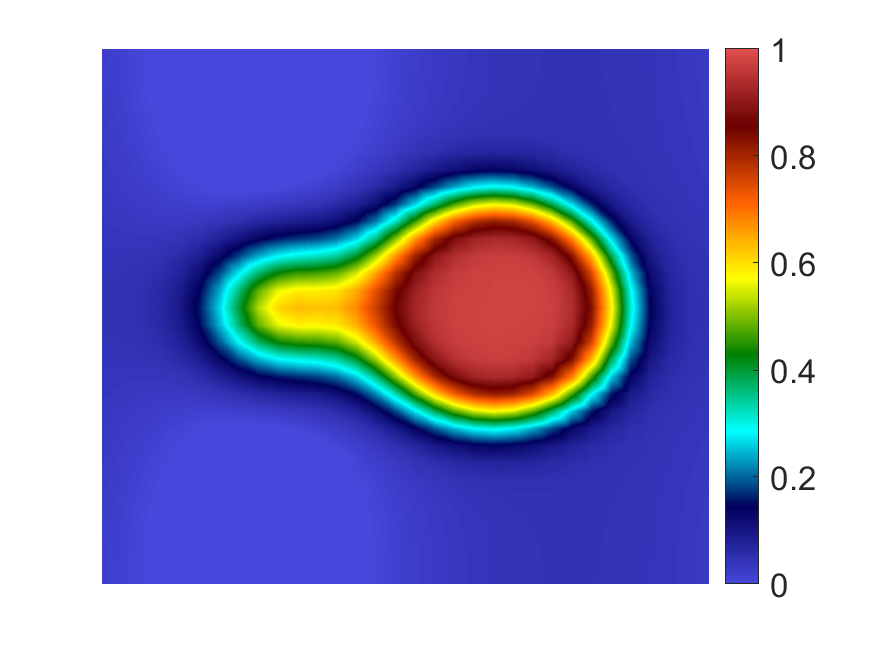} 
    &
    \includegraphics[trim={2.0cm 0.0cm 3.0cm 0.5cm},clip,scale=0.32]{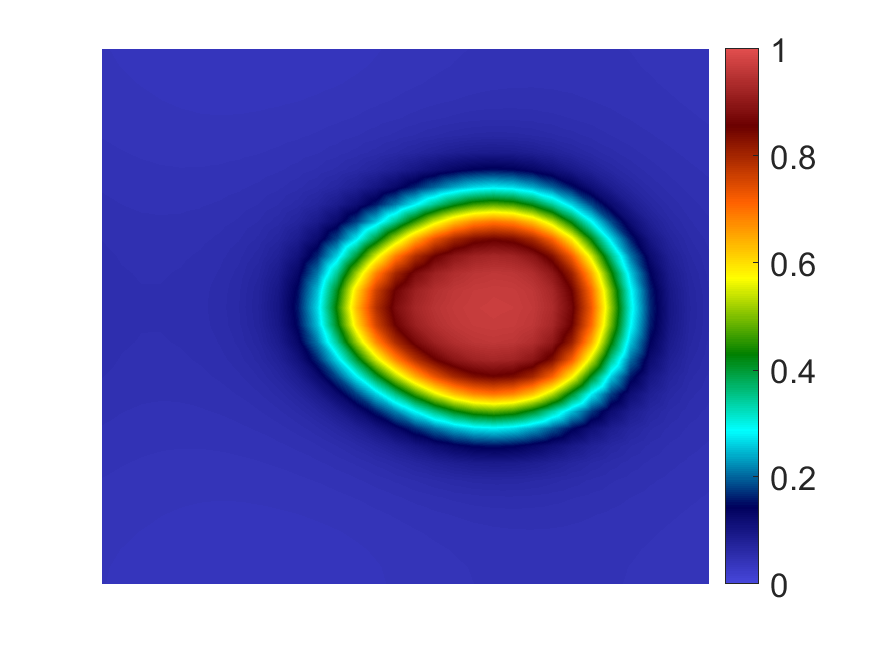}  
    &
    \includegraphics[trim={2.0cm 0.0cm 3.0cm 0.5cm},clip,scale=0.32]{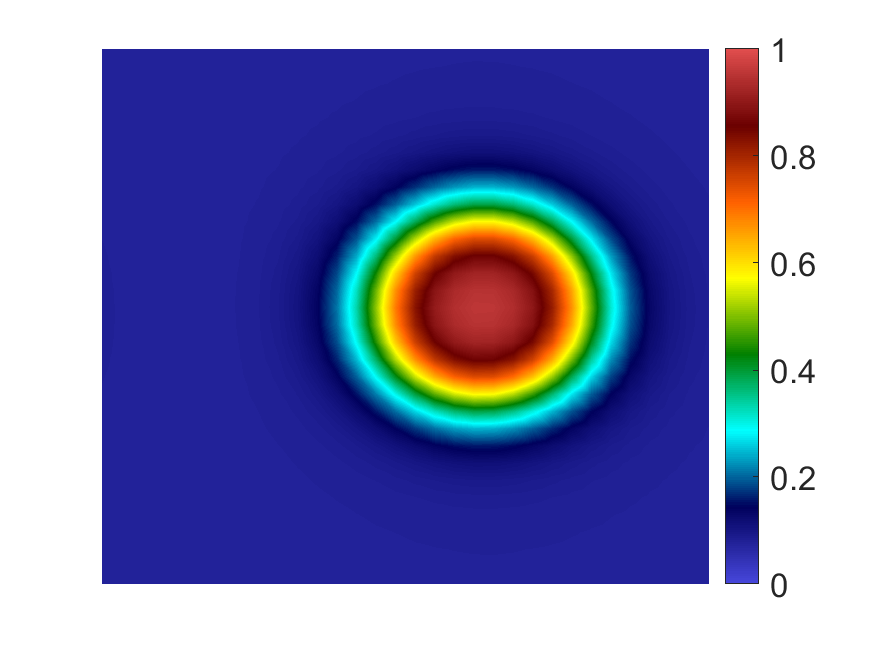}
    &
    \includegraphics[trim={2.0cm 0.0cm 0.0cm 0.5cm},clip,scale=0.32]{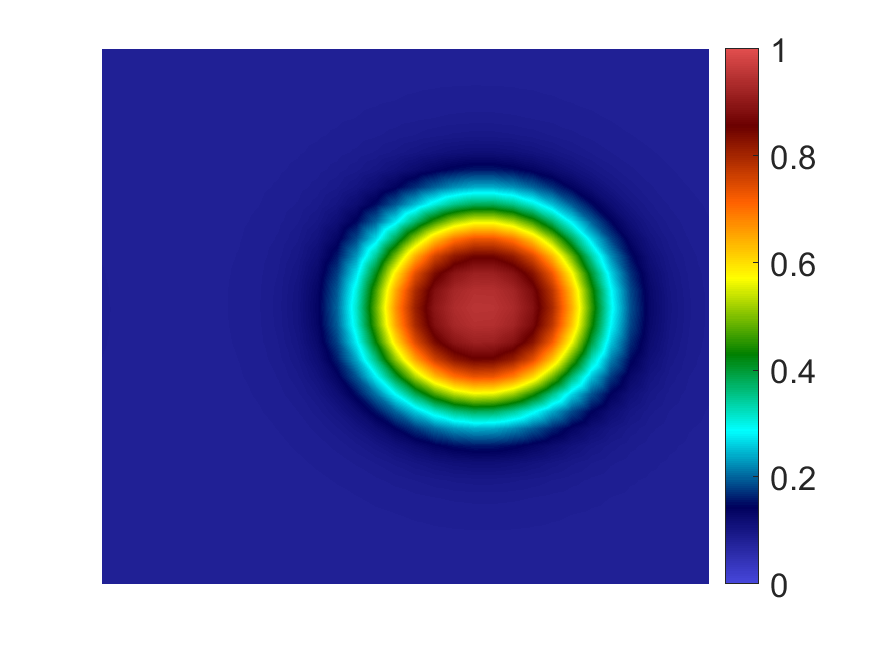} \\
\end{tabular}
    \caption{Snapshots of the conserved phase-field $\rho$ at the times $t\in\{0.5,1.5,7.5,10\}$ with the three different test cases. (Upper row) Initial temperature profile A (Middle row) Initial temperature profile B (Lower row) Initial temperature profile C.\label{fig:evorho}}
\end{figure}

\begin{figure}[htbp!]
\centering
\footnotesize
\begin{tabular}{cccc}
    \includegraphics[trim={1.5cm 0.0cm 3.0cm 0.5cm},clip,scale=0.32]{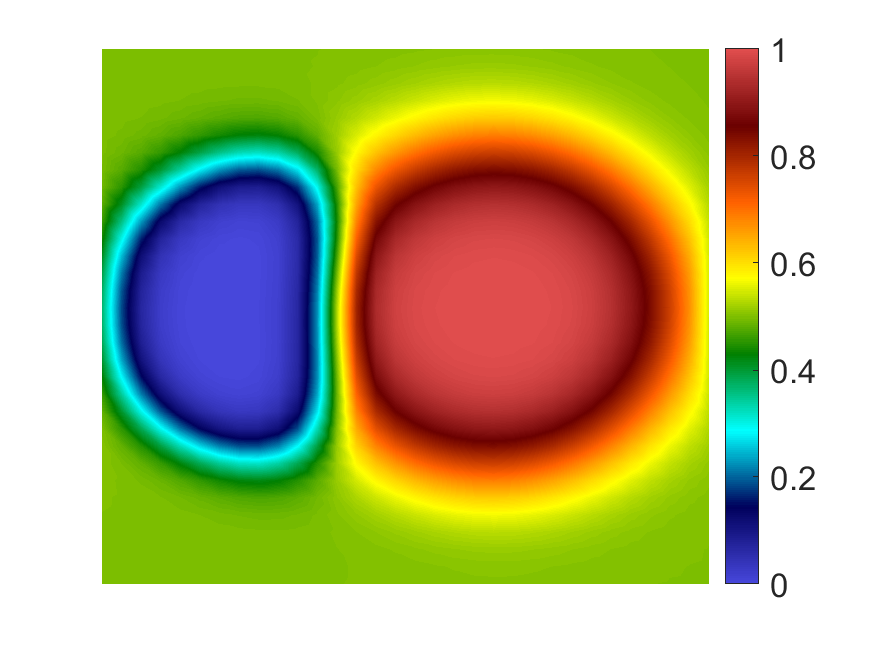} 
    &
    \includegraphics[trim={2.0cm 0.0cm 3.0cm 0.5cm},clip,scale=0.32]{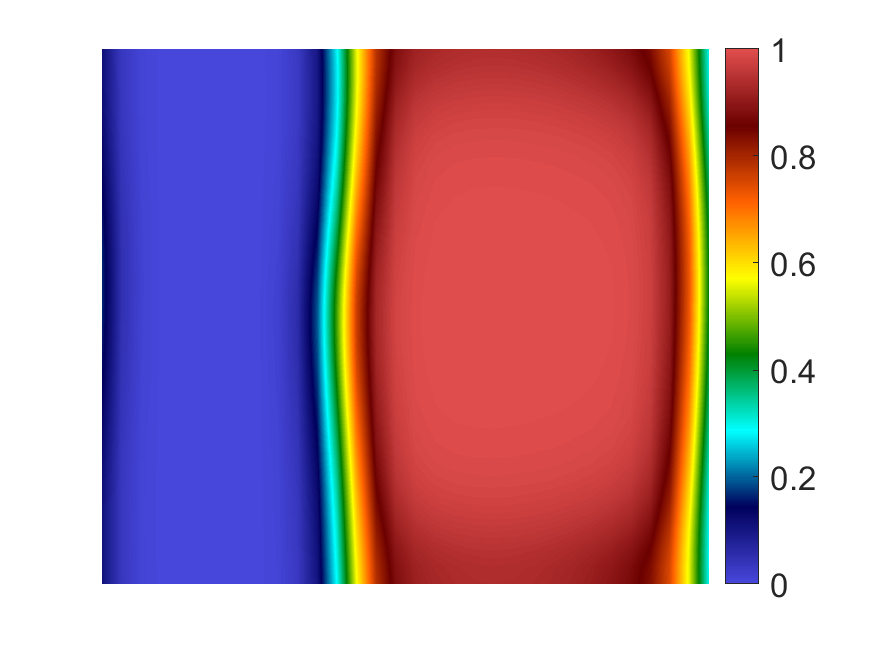}  
    &
    \includegraphics[trim={2.0cm 0.0cm 3.0cm 0.5cm},clip,scale=0.32]{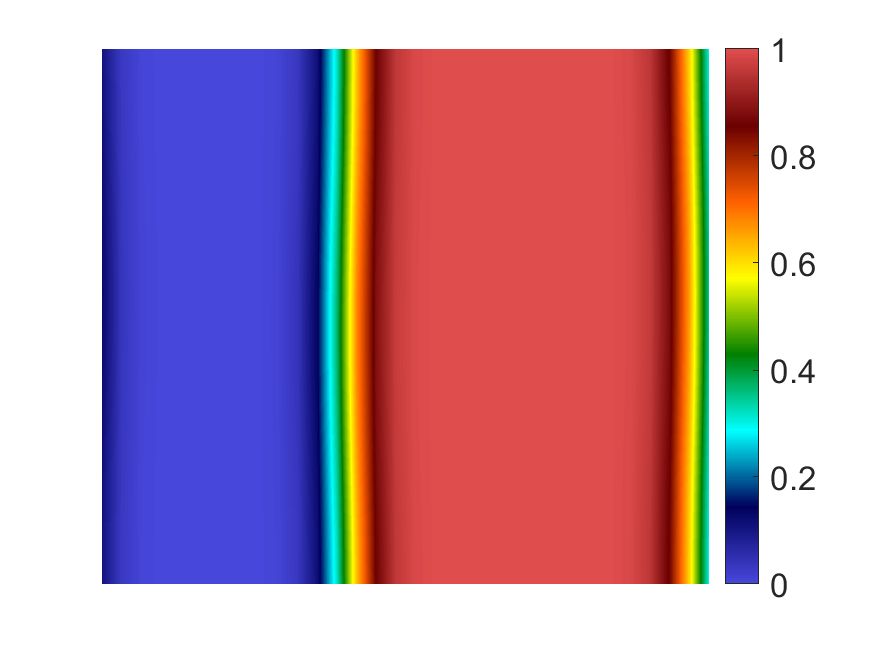}
    &
    \includegraphics[trim={2.0cm 0.0cm 0.0cm 0.5cm},clip,scale=0.32]{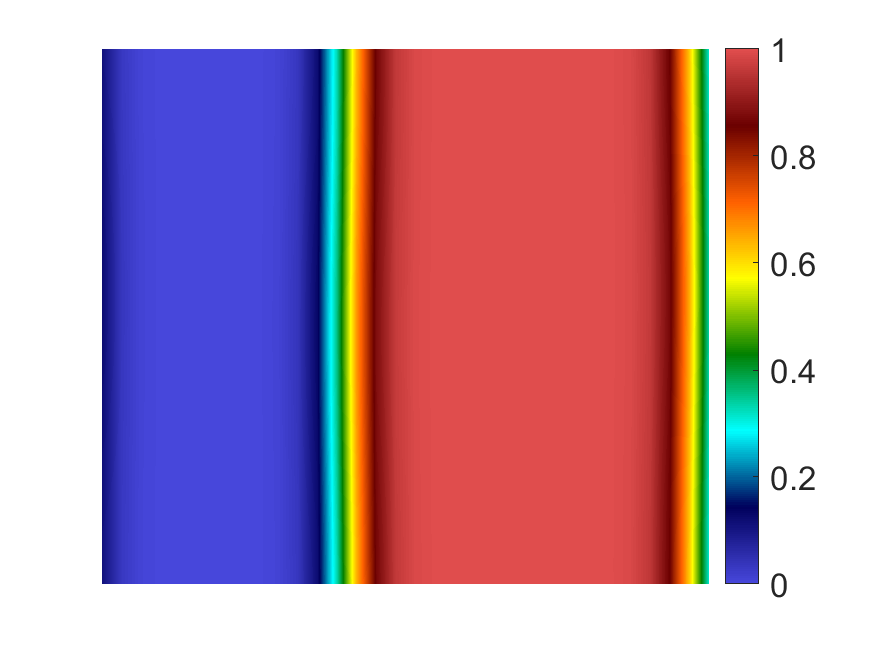} \\[-0.5em]
    \includegraphics[trim={1.5cm 0.0cm 3.0cm 0.5cm},clip,scale=0.32]{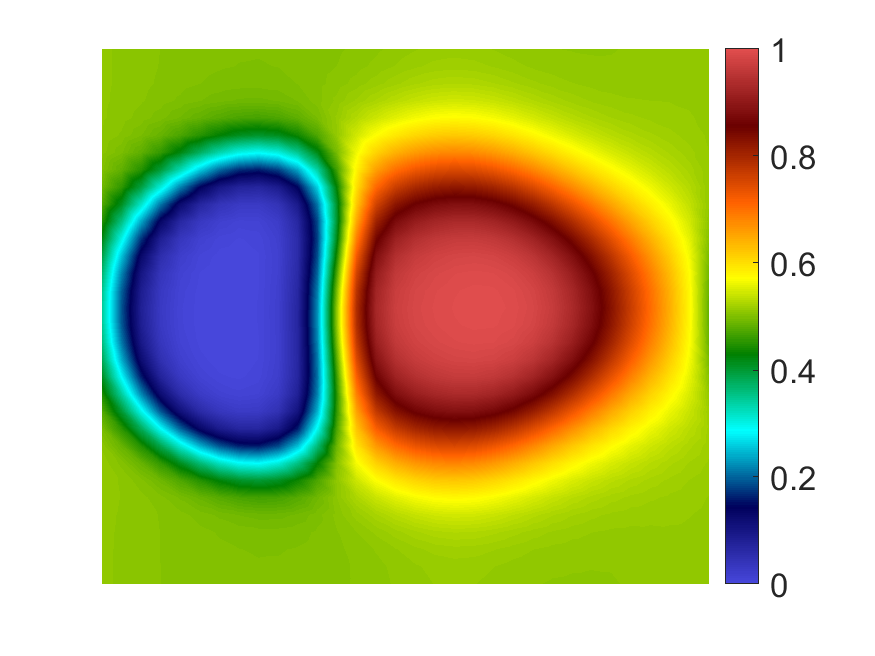} 
    &
    \includegraphics[trim={2.0cm 0.0cm 3.0cm 0.5cm},clip,scale=0.32]{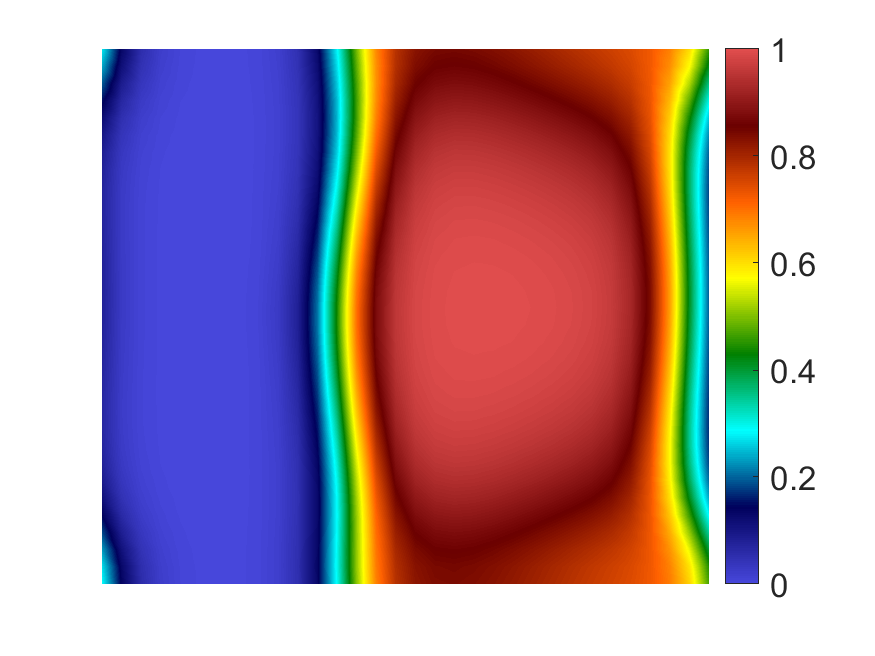}  
    &
    \includegraphics[trim={2.0cm 0.0cm 3.0cm 0.5cm},clip,scale=0.32]{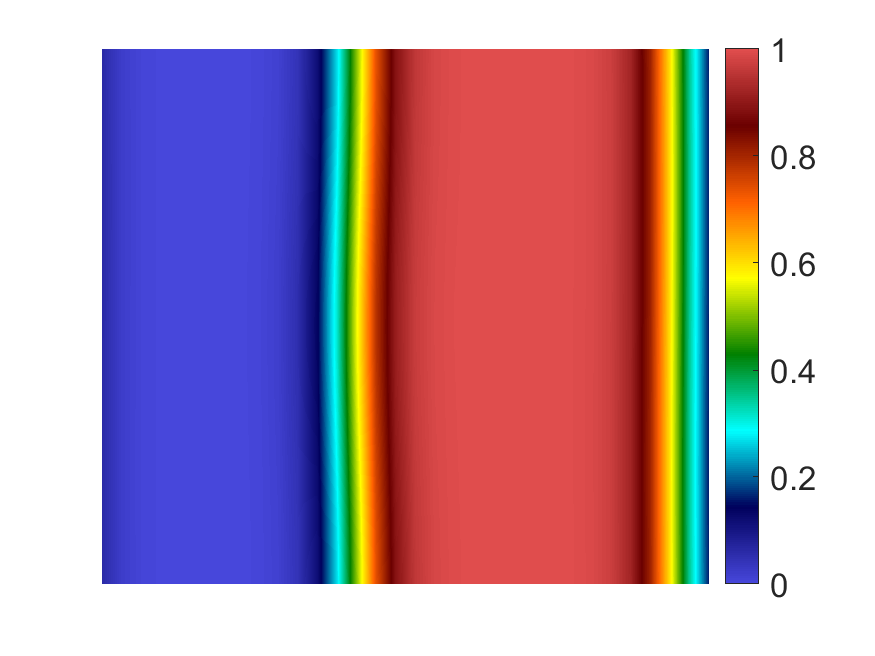}
    &
    \includegraphics[trim={2.0cm 0.0cm 0.0cm 0.5cm},clip,scale=0.32]{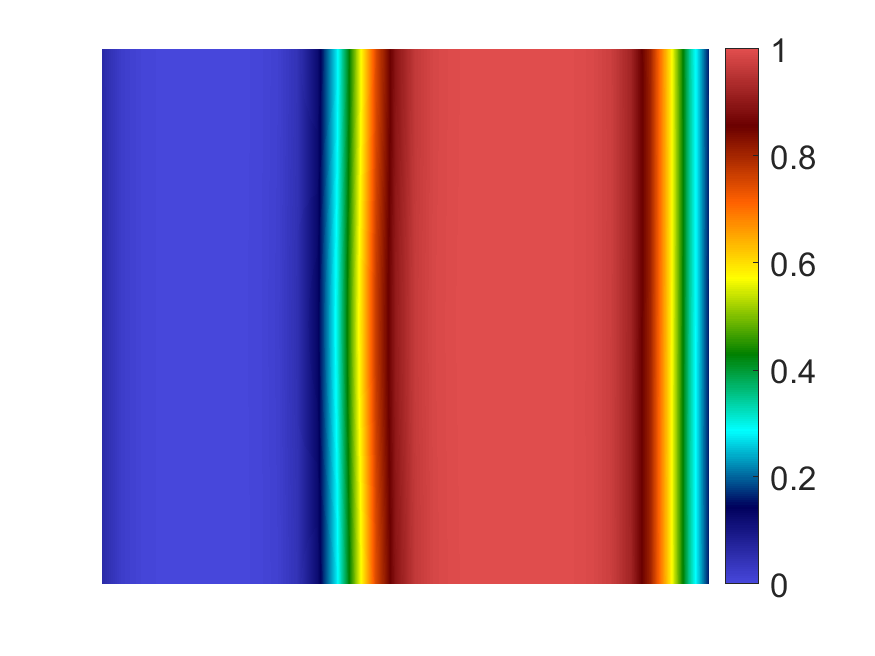} \\[-0.5em]
    \includegraphics[trim={1.5cm 0.0cm 3.0cm 0.5cm},clip,scale=0.32]{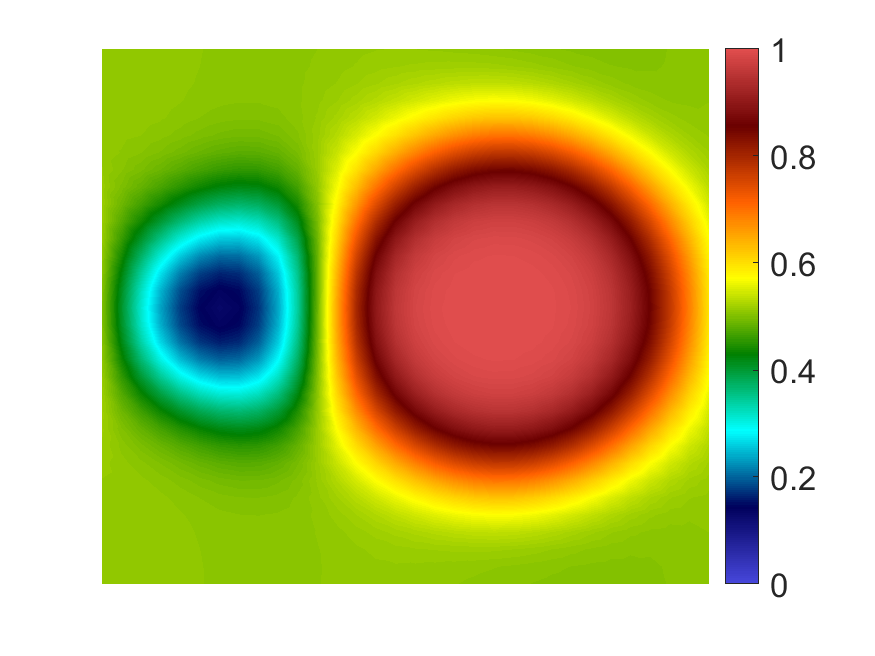} 
    &
    \includegraphics[trim={2.0cm 0.0cm 3.0cm 0.5cm},clip,scale=0.32]{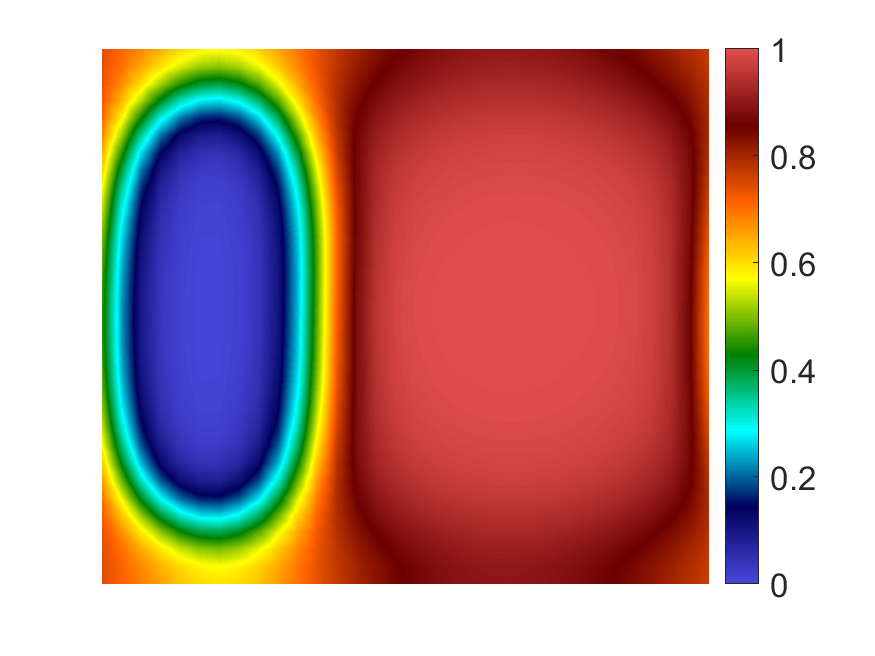}  
    &
    \includegraphics[trim={2.0cm 0.0cm 3.0cm 0.5cm},clip,scale=0.32]{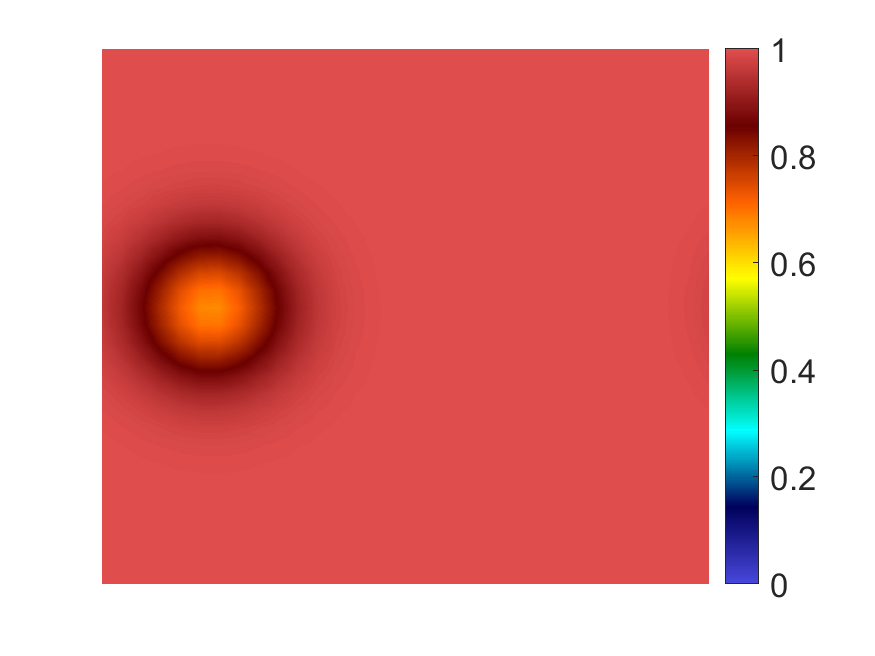}
    &
    \includegraphics[trim={2.0cm 0.0cm 0.0cm 0.5cm},clip,scale=0.32]{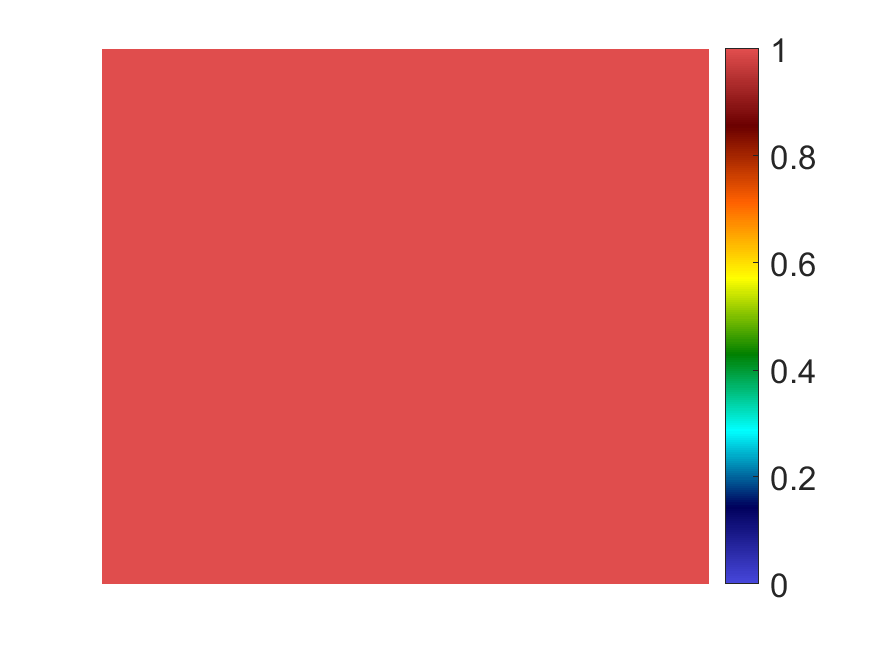} \\
\end{tabular}
    \caption{Snapshots of the non-conserved phase-field $\eta$ at the times $t\in\{0.5,1.5,7.5,10\}$ with the three different test cases. (Upper row) Initial temperature profile A (Middle row) Initial temperature profile B (Lower row) Initial temperature profile C.\label{fig:evoeta}}
\end{figure}

\begin{figure}[htbp!]
\centering
\footnotesize
\begin{tabular}{cccc}
    \includegraphics[trim={1.5cm 0.0cm 3.0cm 0.5cm},clip,scale=0.32]{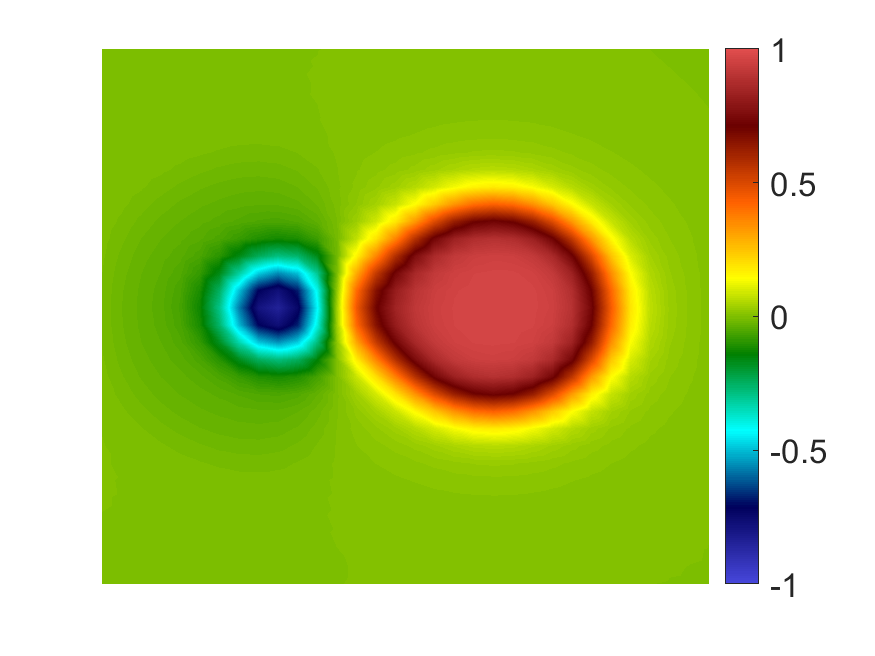} 
    &
    \includegraphics[trim={2.0cm 0.0cm 3.0cm 0.5cm},clip,scale=0.32]{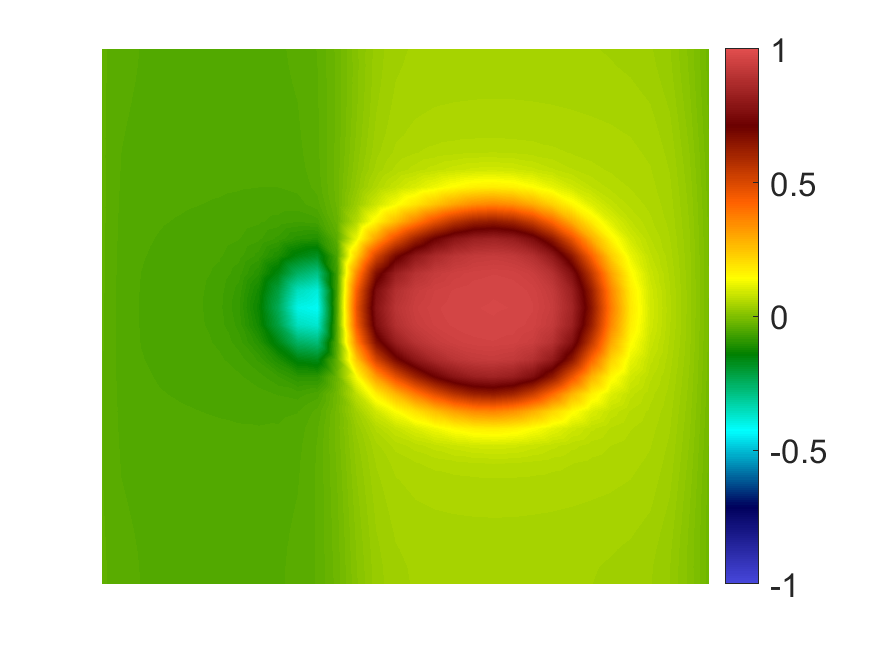}  
    &
    \includegraphics[trim={2.0cm 0.0cm 3.0cm 0.5cm},clip,scale=0.32]{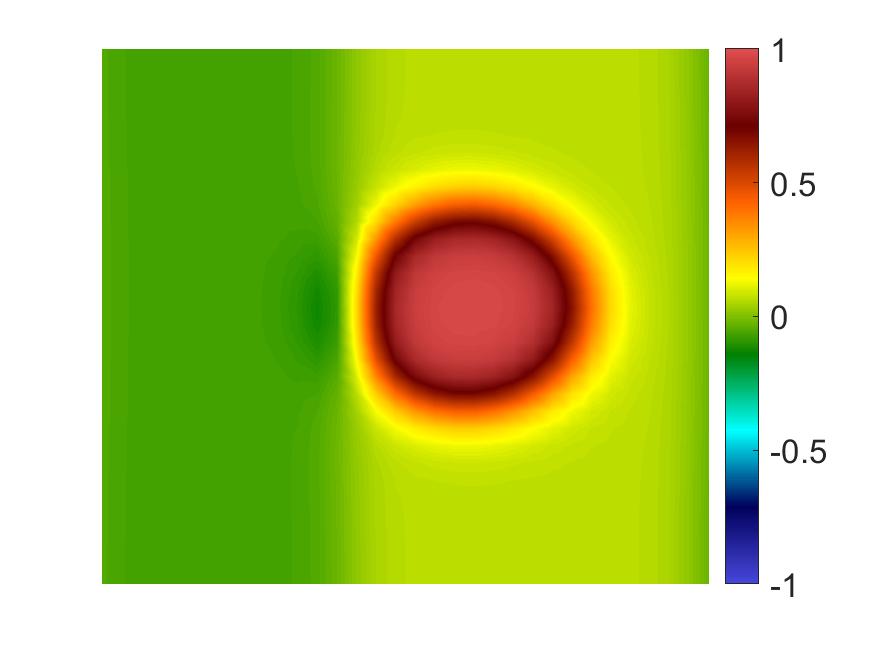}
    &
    \includegraphics[trim={2.0cm 0.0cm 0.0cm 0.5cm},clip,scale=0.32]{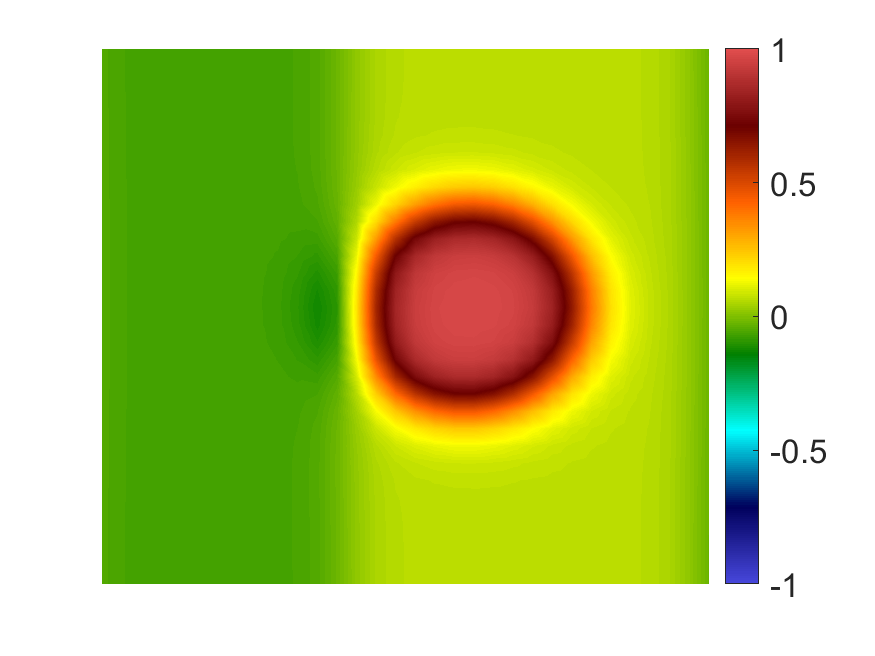} \\[-0.5em]
    \includegraphics[trim={1.5cm 0.0cm 3.0cm 0.5cm},clip,scale=0.32]{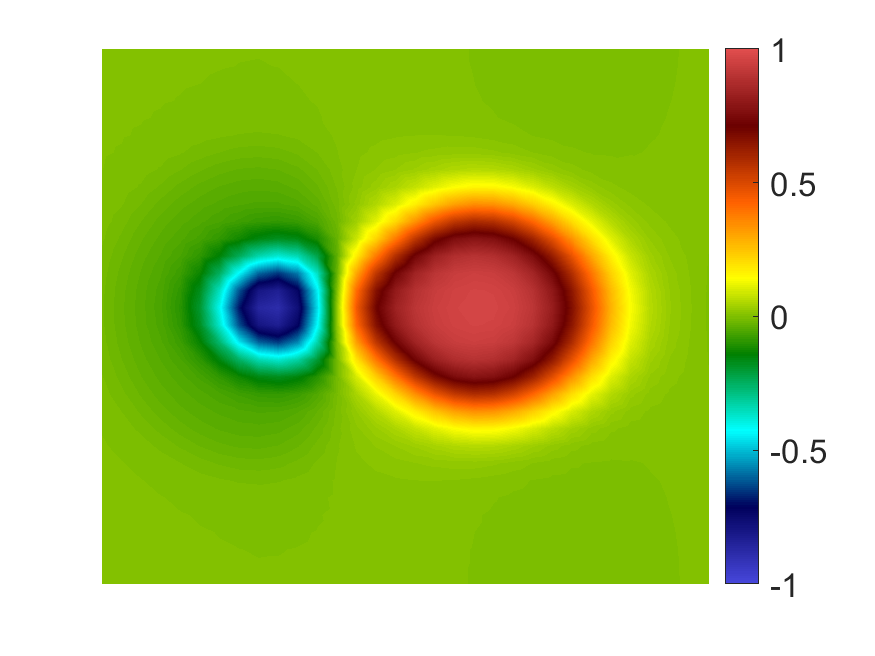} 
    &
    \includegraphics[trim={2.0cm 0.0cm 3.0cm 0.5cm},clip,scale=0.32]{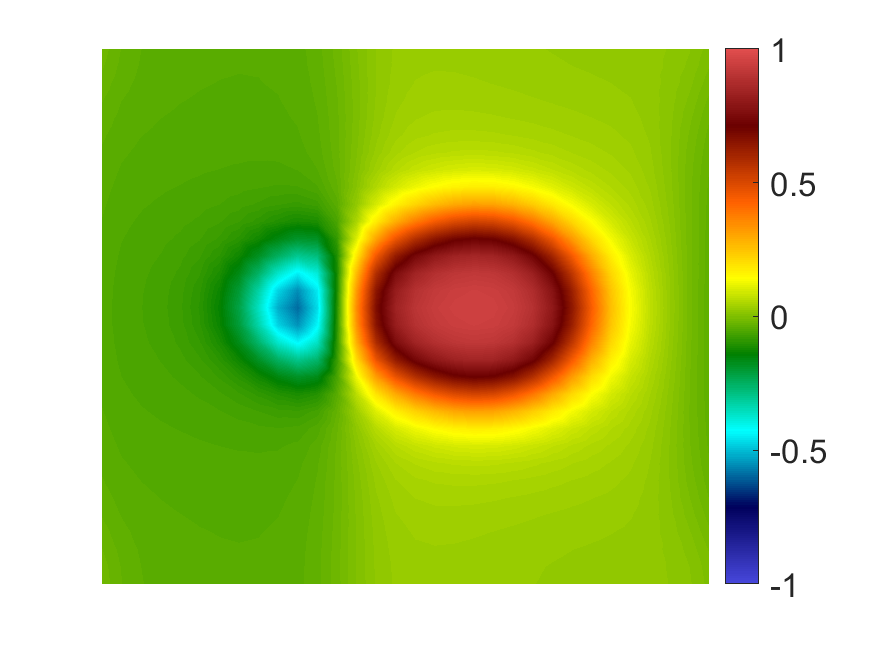}  
    &
    \includegraphics[trim={2.0cm 0.0cm 3.0cm 0.5cm},clip,scale=0.32]{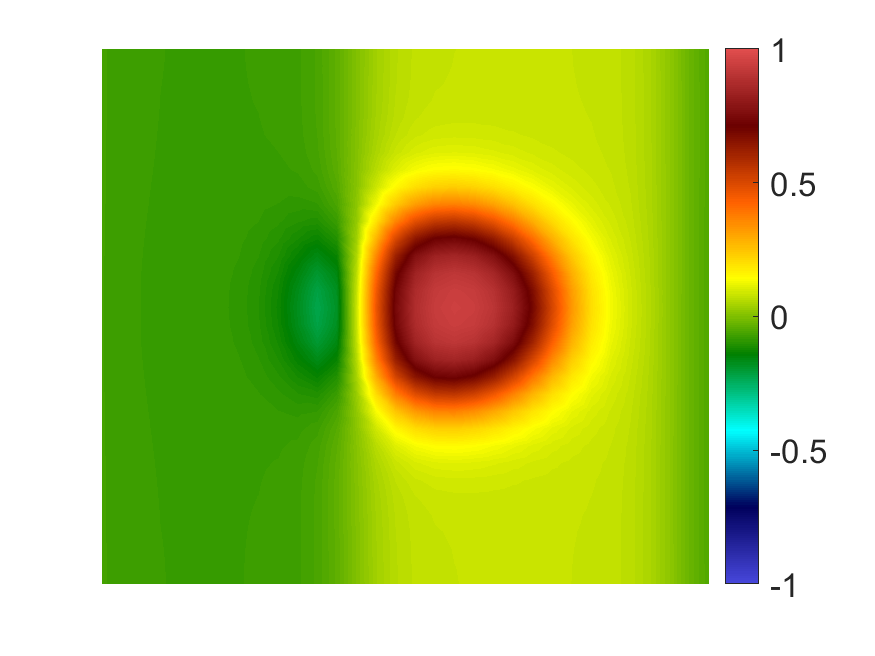}
    &
    \includegraphics[trim={2.0cm 0.0cm 0.0cm 0.5cm},clip,scale=0.32]{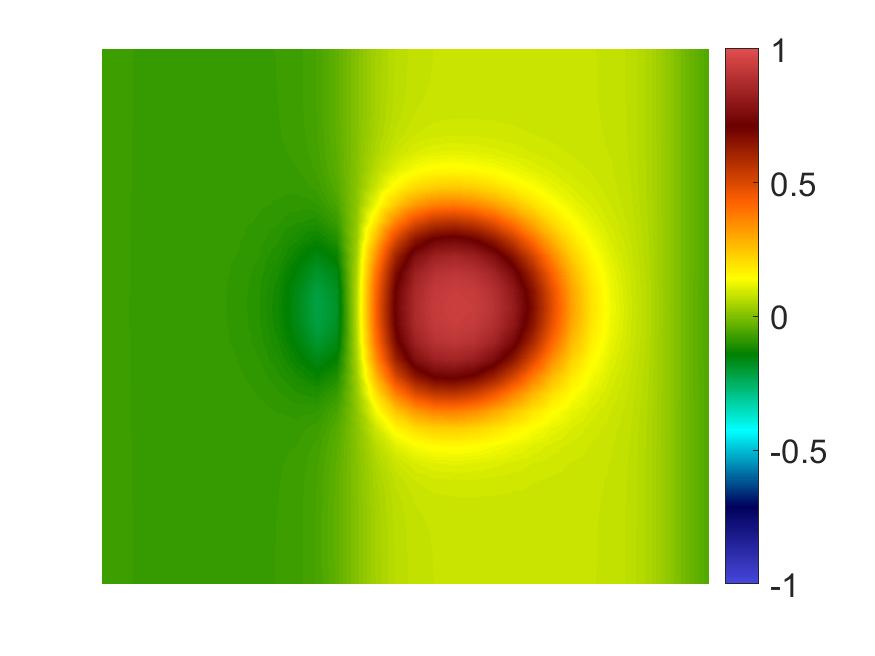} \\[-0.5em]
    \includegraphics[trim={1.5cm 0.0cm 3.0cm 0.5cm},clip,scale=0.32]{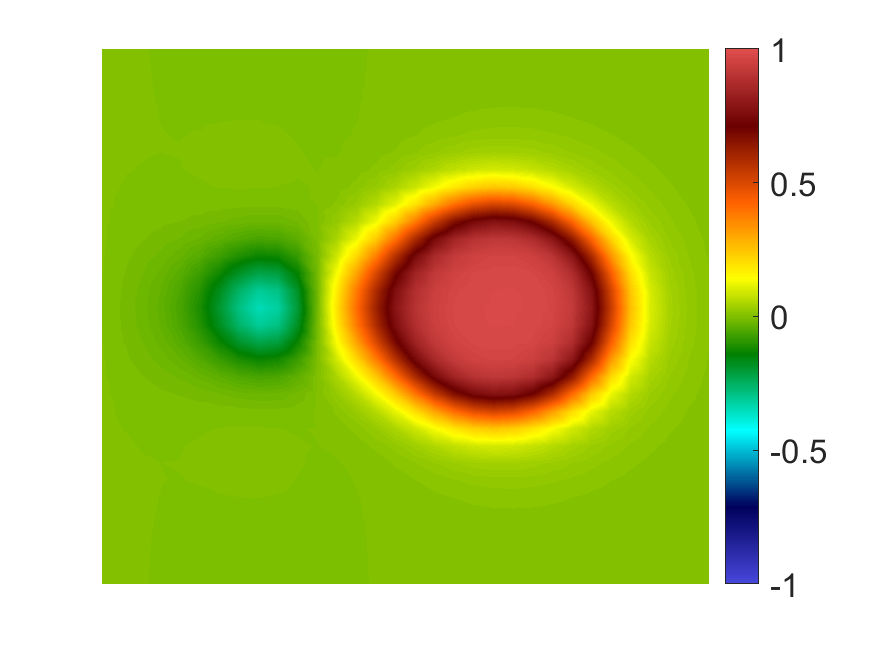} 
    &
    \includegraphics[trim={2.0cm 0.0cm 3.0cm 0.5cm},clip,scale=0.32]{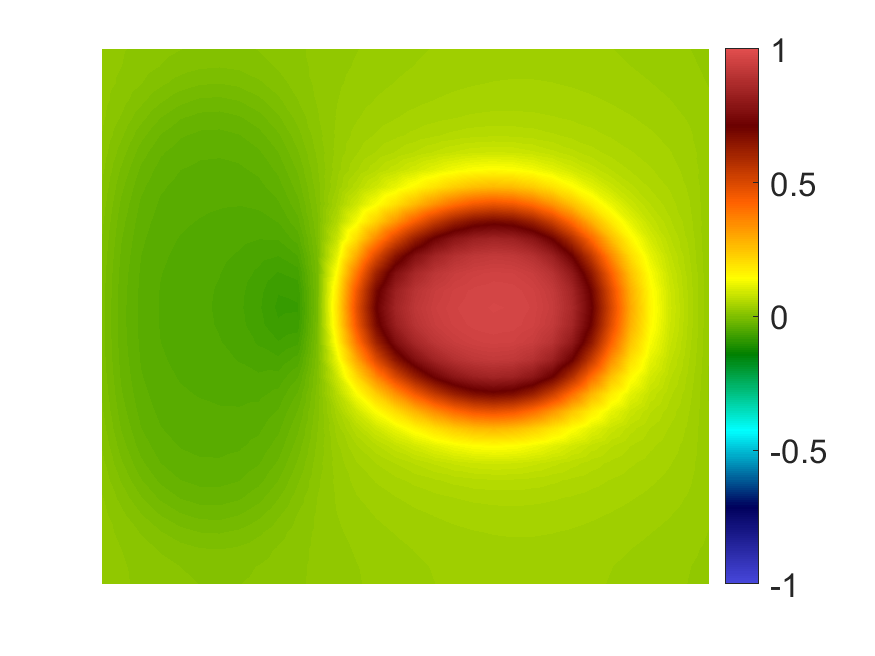}  
    &
    \includegraphics[trim={2.0cm 0.0cm 3.0cm 0.5cm},clip,scale=0.32]{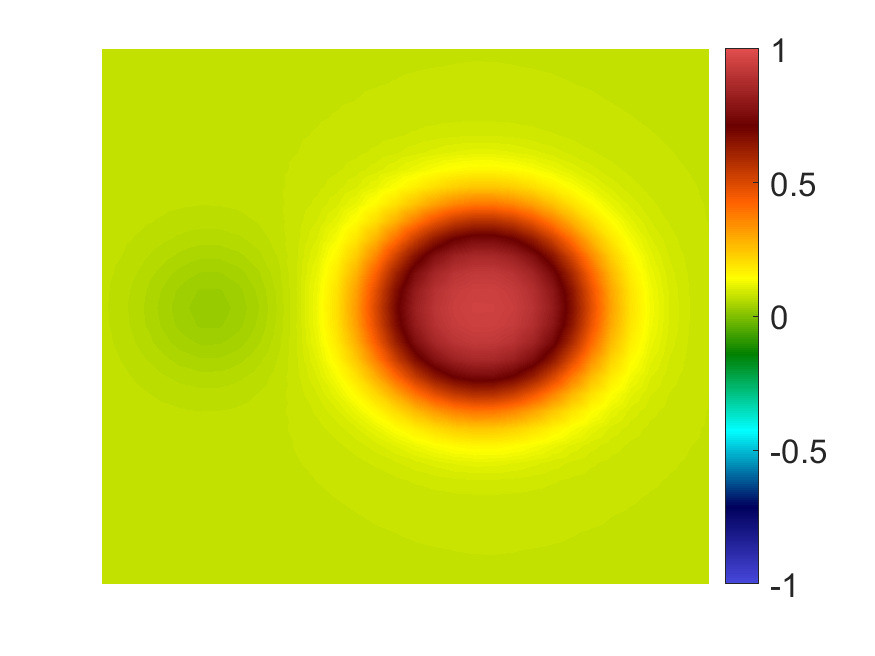}
    &
    \includegraphics[trim={2.0cm 0.0cm 0.0cm 0.5cm},clip,scale=0.32]{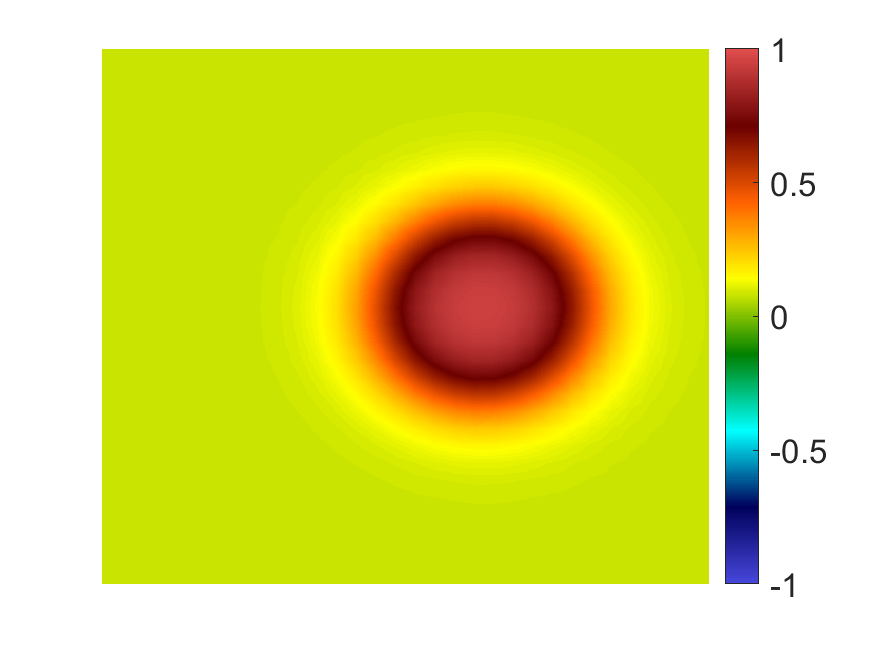} \\
\end{tabular}
    \caption{Snapshots of  $\rho(2\eta-1)$ at the times $t\in\{0.5,1.5,7.5,10\}$ with the three different test cases. (Upper row) Initial temperature profile A (Middle row) Initial temperature profile B (Lower row) Initial temperature profile C.\label{fig:evophase}}
\end{figure}

\begin{figure}[htbp!]
\centering
\footnotesize
\begin{tabular}{ccc}
   
    \includegraphics[trim={2.0cm 0.0cm 0.0cm 0.5cm},clip,scale=0.32]{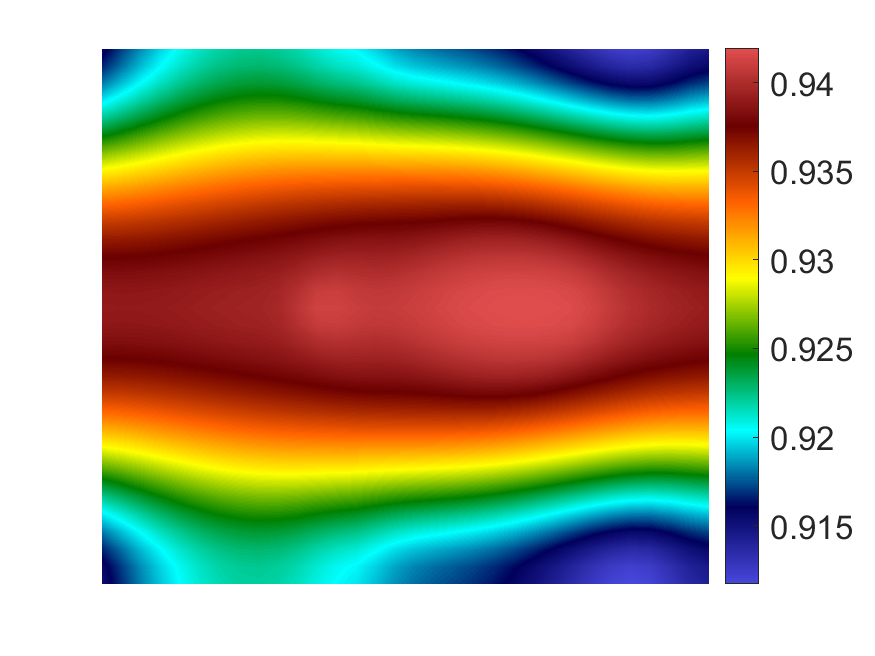}  
    &
    \includegraphics[trim={2.0cm 0.0cm 0.0cm 0.5cm},clip,scale=0.32]{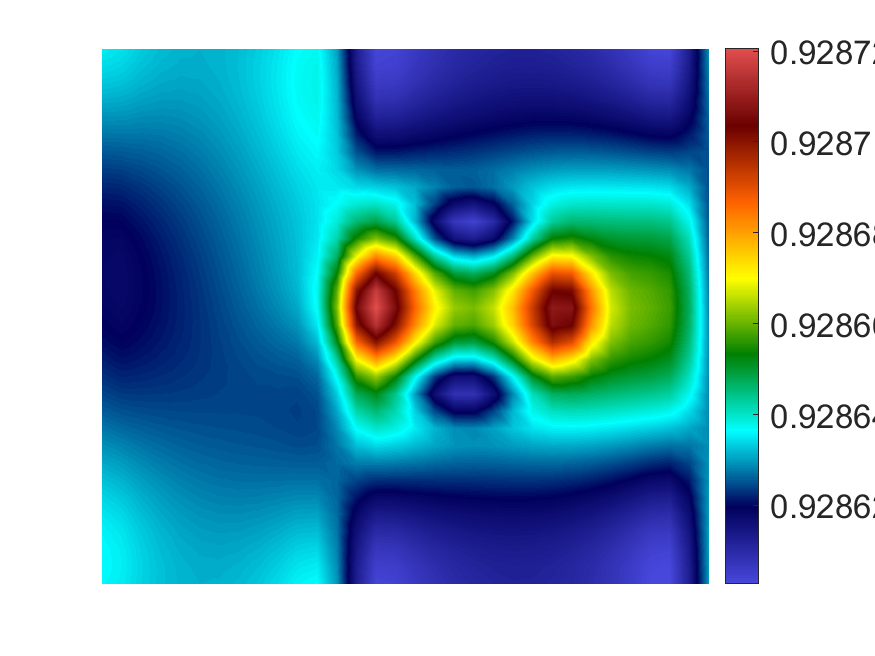}
    &
    \includegraphics[trim={2.0cm 0.0cm 0.0cm 0.5cm},clip,scale=0.32]{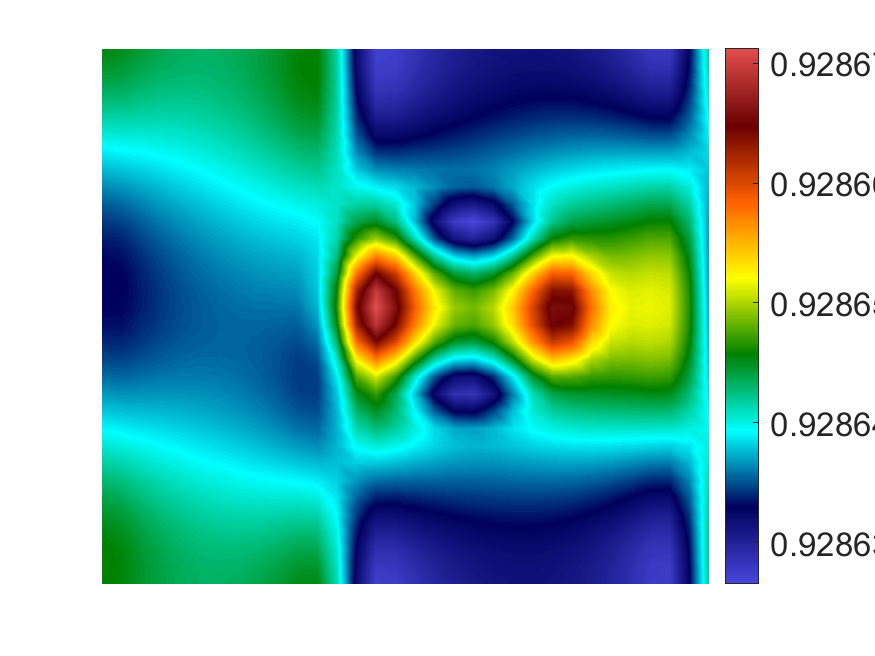} \\[-0.5em]
    \includegraphics[trim={2.0cm 0.0cm 0.0cm 0.5cm},clip,scale=0.32]{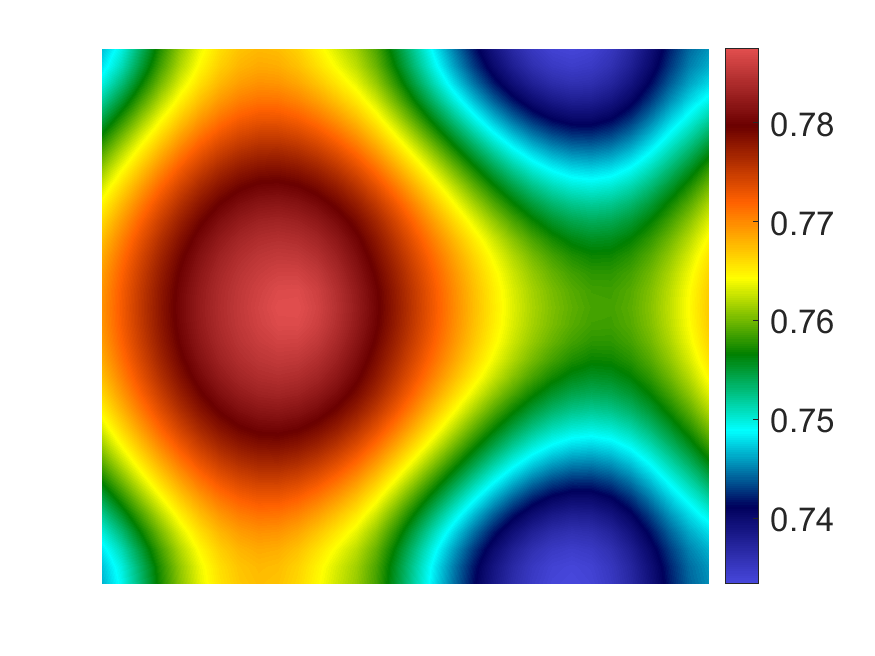}  
    &
    \includegraphics[trim={2.0cm 0.0cm 0.0cm 0.5cm},clip,scale=0.32]{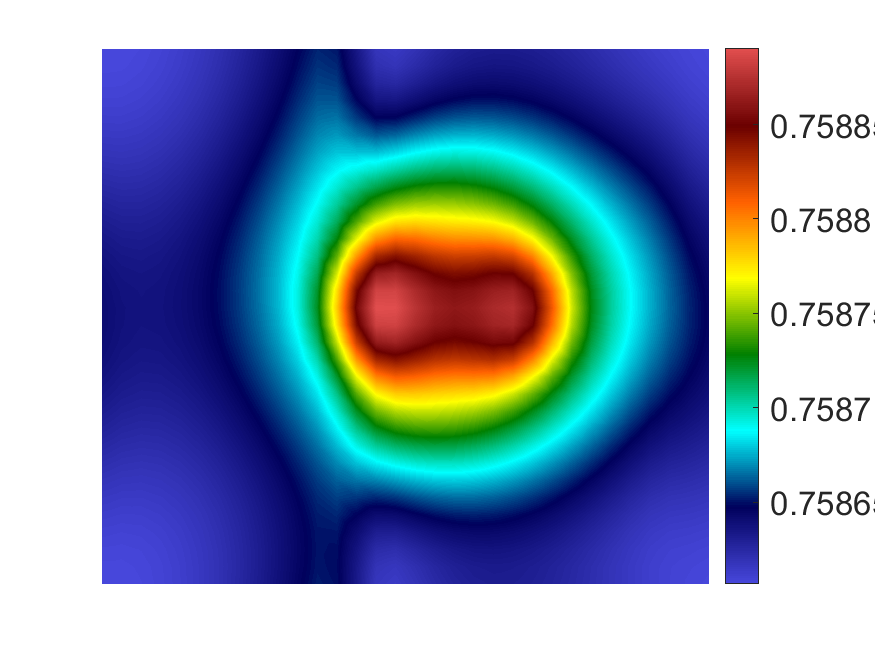}
    &
    \includegraphics[trim={2.0cm 0.0cm 0.0cm 0.5cm},clip,scale=0.32]{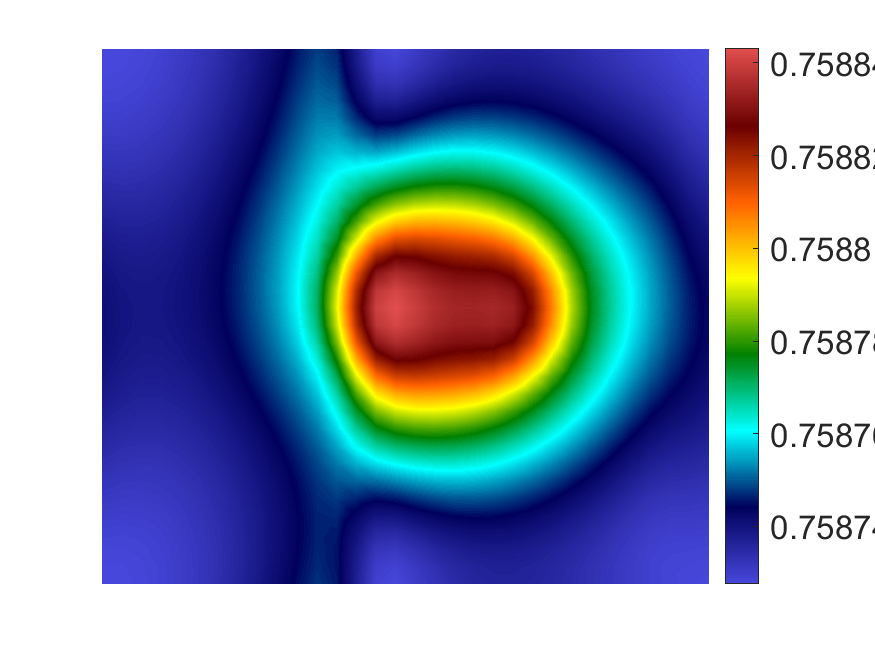} \\[-0.5em]
    \includegraphics[trim={2.0cm 0.0cm 0.0cm 0.5cm},clip,scale=0.32]{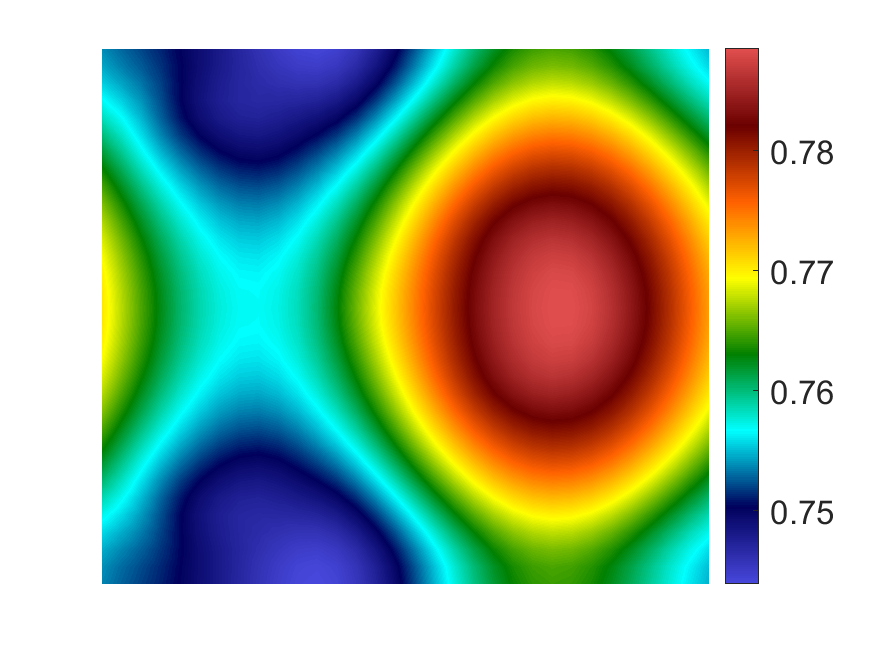}  
    &
    \includegraphics[trim={2.0cm 0.0cm 0.0cm 0.5cm},clip,scale=0.32]{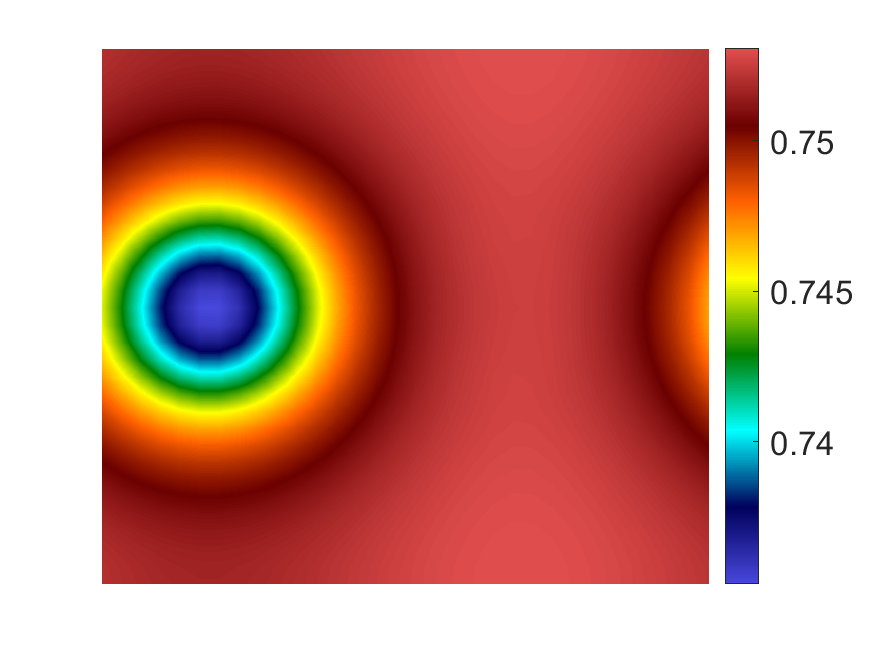}
    &
    \includegraphics[trim={2.0cm 0.0cm 0.0cm 0.5cm},clip,scale=0.32]{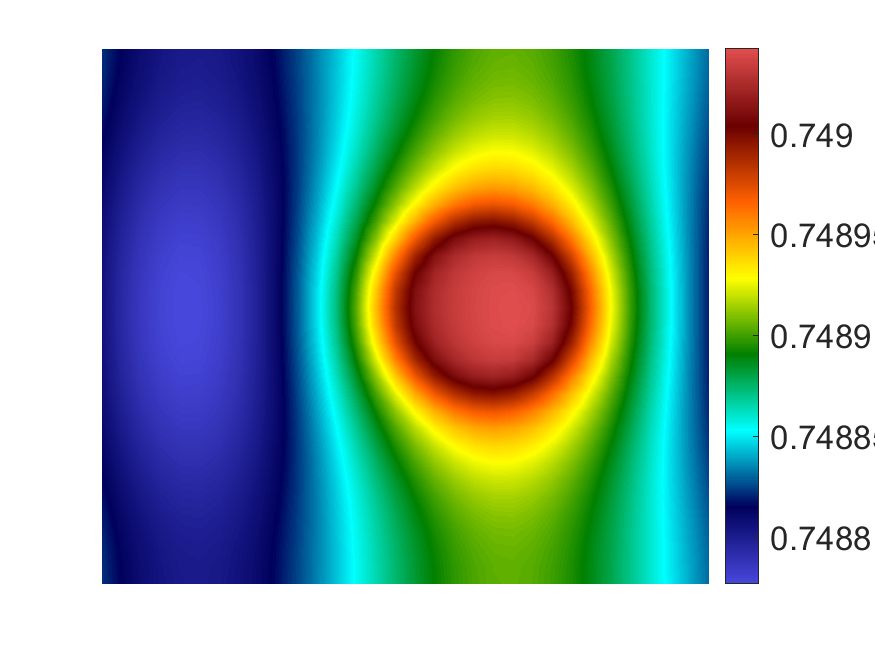} \\
\end{tabular}
    \caption{Snapshots of the inverse temperature $\theta$ at the times $t\in\{1.5,7.5,10\}$ with the three different test cases. (Upper row) Initial temperature profile A (Middle row) Initial temperature profile B (Lower row) Initial temperature profile C.\label{fig:evotheta}}
\end{figure}

\section*{Acknowledgments}
Financial support by the German Science Foundation (DFG, Project number 441153493) within the Priority Program SPP~2256: \textit{Variational Methods for Predicting Complex Phenomena in Engineering Structures and Materials} (project BR~7093/1-2 and Xu~121/13-2) and via TRR~146: \textit{Multiscale Simulation Methods for Soft Matter Systems} (project C3) is gratefully acknowledged. Part of the research was conducted during a research stay of the first author at RICAM/JKU Linz.

\section{Conclusion \& Outlook}
In this work, we have considered a non-isothermal coupled phase-field model with cross-kinetic coupling. We transformed the system to reveal the variational structure and introduced a semi-discrete and a fully-discrete structure-preserving approximation. Furthermore, we have considered the stability of the discrete solutions using the relative entropy framework. In principle, such estimates allow us to perform a rather straightforward error analysis. The expected rates are illustrated by a numerical example. Rigorous error analysis will be part of further research. \changesone{Add outlook}

\clearpage
\appendix

\section{Proof full discrete stability estimate}\label{app:1}
In this section, we will expand the fully discrete relative entropy. We divide the temporal jump of the relative entropy as follows
\begin{align*}
 \mathcal{W}_{\lambda}(\rho,\theta,\eta|\hat\rho,\hat\theta,\hat\eta)\vert_{t^n}^{t^{n+1}} = \mathcal{W}(\rho,\theta,\eta|\hat\rho,\hat\theta,\hat\eta)\vert_{t^n}^{t^{n+1}} + \frac{\lambda}{2}(\norm{\rho_h-\hat\rho_h}_0^2 + \norm{\eta_h-\hat\eta_h}_0^2)\vert_{t^n}^{t^{n+1}}.  
\end{align*}

The structure of this appendix is as follows
\begin{itemize}
    \item The first term from the above decomposition will be expanded in paragraph \ref{subs:expansion}
    \item The resulting relative dissipation is estimated in paragraph \ref{subs:dissip}
    \item In paragraph \ref{subs:quadpart} we deal with the quadratic terms from above.
    \item Finally in \ref{subs:collection} we collect all parts together
\end{itemize}

To increase readability we will neglect the index $h$.

\subsection{Expansion relative entropy}\label{subs:expansion}

Let us compute the evolution of the discrete relative energy, adding suitable zeros and rearranging terms we find
\begin{align*}
&\mathcal{W}(\rho,\theta,\eta|\hat\rho,\hat\theta,\hat\eta)\vert_{t^n}^{t^{n+1}} \\
=& \la \psi(\rho^{n+1},\theta^{n+1},\eta^{n+1}) - \psi(\hat\rho^{n+1},\hat\theta^{n+1},\hat\eta^{n+1}) - \partial_\theta\psi(\rho^{n+1},\theta^{n+1},\eta^{n+1})(\theta^{n+1}-\hat \theta^{n+1}) \\
&- \partial_\rho\psi(\hat\rho^{n+1},\hat\theta^{n+1},\hat\eta^{n+1})(\rho^{n+1}-\hat\rho^{n+1})
- \partial_\eta\psi(\hat\rho^{n+1},\hat\theta^{n+1},\hat\eta^{n+1})(\eta^{n+1}-\hat\eta^{n+1}),1\ra \\
& + \frac{\gamma}{2}\norm{\nabla\rho^{n+1}}_0^2 + \frac{\gamma}{2}\norm{\nabla\eta^{n+1}}_0^2 - \frac{\gamma}{2}\norm{\nabla\hat\rho^{n+1}}_0^2 - \frac{\gamma}{2}\norm{\nabla\hat\eta^{n+1}}_0^2 \\
&-\la \psi(\rho^{n},\theta^{n},\eta^{n}) - \psi(\hat\rho^{n},\hat\theta^{n},\hat\eta^{n}) - \partial_\theta\psi(\rho^{n},\theta^{n},\eta^{n})(\theta^{n}-\hat \theta^{n}) \\
&- \partial_\rho\psi(\hat\rho^{n},\hat\theta^{n},\hat\eta^{n})(\rho^{n}-\hat\rho^{n})
- \partial_\eta\psi(\hat\rho^{n},\hat\theta^{n},\hat\eta^{n})(\eta^{n}-\hat\eta^{n}),1\ra \\
&- \frac{\gamma}{2}\norm{\nabla\rho^{n}}_0^2 - \frac{\gamma}{2}\norm{\nabla\eta^{n}}_0^2 + \frac{\gamma}{2}\norm{\nabla\hat\rho^{n}}_0^2 +\frac{\gamma}{2}\norm{\nabla\hat\eta^{n}}_0^2\\
=& -\la d^{n+1}(e-\hat e), \theta^{n+1}-\hat \theta^{n+1}\ra  +\la d^{n+1}(\rho-\hat\rho),\mu_{\rho}^{n+1}-\hat\mu_{\rho}^{n+1}+r_{2}^{n+1} \ra \\
&+ \la d^{n+1}(\eta-\hat\eta),\mu_{\eta}^{n+1}-\hat\mu_{\eta}^{n+1}+r_{5}^{n+1}\ra \\
&- \la e(\rho^{n},\theta^{n},\eta^{n}), d^{n+1}(\theta^{n+1}-\hat \theta^{n+1}) \ra - \la d^{n+1}\hat e, \theta^{n+1}-\hat \theta^{n+1}\ra - \frac{\gamma}{2}\norm{\nabla d^{n+1}(\rho-\hat\rho)}_0^2  \\
& -\frac{\gamma}{2}\norm{\nabla d^{n+1}(\eta-\hat\eta)}_0^2- \la d^{n+1}(\rho-\hat\rho),\partial_\rho\tilde\psi - \partial_\rho\tilde{\hat\psi}\ra - \la d^{n+1}(\eta-\hat\eta),\partial_\eta\tilde\psi- \partial_\eta\tilde{\hat\psi} \ra \\
& +\la \psi(\rho^{n+1},\theta^{n+1},\eta^{n+1}) - \psi(\hat\rho^{n+1},\hat\theta^{n+1},\hat\eta^{n+1}) \\
&- \partial_\rho\psi(\hat\rho^{n+1},\hat\theta^{n+1},\hat\eta^{n+1})(\rho^{n+1}-\hat\rho^{n+1})
- \partial_\eta\psi(\hat\rho^{n+1},\hat\theta^{n+1},\hat\eta^{n+1})(\eta^{n+1}-\hat\eta^{n+1}),1\ra \\
&-\la \psi(\rho^{n},\theta^{n},\eta^{n}) - \psi(\hat\rho^{n},\hat\theta^{n},\hat\eta^{n}) - \partial_\rho\psi(\hat\rho^{n},\hat\theta^{n},\hat\eta^{n})(\rho^{n}-\hat\rho^{n})
- \partial_\eta\psi(\hat\rho^{n},\hat\theta^{n},\hat\eta^{n})(\eta^{n}-\hat\eta^{n}),1\ra \\
=&-\la d^{n+1}(e-\hat e), \theta^{n+1}-\hat \theta^{n+1}\ra  +\la d^{n+1}(\rho-\hat\rho),\mu_{\rho}^{n+1}-\hat\mu_{\rho}^{n+1}+r_{2}^{n+1} \ra \\
&+ \la d^{n+1}(\eta-\hat\eta),\mu_{\eta}^{n+1}-\hat\mu_{\eta}^{n+1}+r_{5}^{n+1}\ra  + \mathcal{R}.
\end{align*}

\subsection{Remainder $\mathcal{R}$}
In order to rewrite and estimate the remainder term $\mathcal{R}$, we introduce the following notation
\begin{align*}
  \psi(\rho^{n+1},\theta^{n+1},\eta^{n+1})=:\psi^{n+1}, \psi(\rho^{n},\theta^{n},\eta^{n})=:\psi^{n}, \psi(\rho^{n},\theta^{n+1},\eta^{n})=:\psi^{n,n+1,n}  
\end{align*}
In this new notation, the remainder term can be written as
\begin{align*}
 \mathcal{R}  = &\la \psi^{n+1} - \psi^n - \partial_\theta\psi^n d^{n+1}\theta - \partial_\rho\psi_{vex}^{n+1}d^{n+1}\rho - \partial_\rho\psi_{cav}^{n,n+1,n}d^{n+1}\rho - \partial_\eta\psi_{vex}^{n+1}d^{n+1}\eta- \partial_\eta\psi_{cav}^{n,n+1,n}d^{n+1}\eta,1 \ra   \\
 &+\la \hat\psi^{n+1} - \hat\psi^n - \partial_\theta\psi^n d^{n+1}\hat\theta^{n+1} - \partial_\rho\hat\psi_{vex}^{n+1}d\hat\rho - \partial_\rho\hat\psi_{cav}^{n,n+1,n}d^{n+1}\hat\rho - \partial_\eta\hat\psi_{vex}^{n+1}d^{n+1}\hat\eta- \partial_\eta\hat\psi_{cav}^{n,n+1,n}d^{n+1}\hat\eta,1 \ra \\
&+ \la \partial_\theta\psi^n -\partial_\theta\hat\psi^n,d^{n+1}\hat\theta \ra- \la \partial_\theta\hat\psi^{n+1} -\partial_\theta\hat\psi^n,(\theta^{n+1}-\hat\theta^{n+1}) \ra \\
&-\la \partial_\rho\hat\psi^{n+1}, \rho^{n+1}-\hat\rho^{n+1}\ra + \la \partial_\rho\hat\psi^n, \rho^n-\hat\rho^{n}\ra -\la \partial_\eta\hat\psi^{n+1}, \eta^{n+1}-\hat\eta^{n+1}\ra + \la \partial_\eta\hat\psi^n, \eta^n-\hat\eta^{n}\ra\\
&+ \la \partial_\rho\hat\psi_{vex}^{n+1} +\partial_\rho\hat\psi_{cav}^{n,n+1,n} , d^{n+1}\rho\ra + \la  \partial_\eta\hat\psi_{vex}^{n+1} +\partial_\eta\hat\psi_{cav}^{n,n+1,n}, d^{n+1}\eta\ra  \\
&+ \la \partial_\rho\psi_{vex}^{n+1} +\partial_\rho\psi_{cav}^{n,n+1,n} - 2\partial_\rho\hat\psi_{vex}^{n+1} +2\partial_\rho\hat\psi_{cav}^{n,n+1,n}, \hat\rho^{n+1}-\hat\rho^n\ra \\
&+ \la \partial_\eta\psi_{vex}^{n+1} +\partial_\eta\psi_{cav}^{n,n+1,n} - 2\partial_\eta\hat\psi_{vex}^{n+1} +2\partial_\eta\hat\psi_{cav}^{n,n+1,n}, \hat\eta^{n+1}-\hat\eta^n\ra \\
=& (i) + \ldots + (xii)
\end{align*}
where we numbered the different inner products. For simplicity we introduce $\ww = (\rho,\eta)^\top$
By adding and subtracting $\psi^{n,n+1,n},\hat\psi^{n,n+1,n}$ we find that
\begin{align*}
2((i) + (ii))  = 
&\la \partial_{\theta\theta}\psi^{n,\xi_1,n}d^{n+1}\theta,d^{n+1}\theta \ra + \la (H_{\ww}^{cav,\chi_1,{n+1},\chi_2}-H^{vex,\chi_3,{n+1},\chi_4}_{\ww})d^{n+1}\ww,d^{n+1}\ww \ra \\
- &\la \partial_{\theta\theta}\hat\psi^{n,\hat\xi_1,n}d^{n+1}\hat\theta,d^{n+1}\hat\theta \ra - \la (\hat H_{\ww}^{cav,\hat\chi_1,{n+1},\hat\chi_2}-\hat H^{vex,\hat\chi_3,{n+1},\hat\chi_4}_{\ww})d^{n+1}\hat\ww,d^{n+1}\hat\ww \ra  \\
& \qquad \xi_1\in(\theta^n,\theta^{n+1}) , \chi_1,\chi_3\in (\rho^n,\rho^{n+1}) , \chi_2,\chi_4\in (\eta^n,\eta^{n+1}),  \\
& \qquad \hat\xi_1\in(\hat\theta^n,\hat\theta^{n+1}) , \hat\chi_1\hat\chi_3\in (\hat\rho^n,\hat\rho^{n+1}) , \hat\chi_2,\hat\chi_4\in (\hat\eta^n,\hat\eta^{n+1}). 
\end{align*}
This can be further rearranged as follows
\begin{align*}
&2((i) + (ii))  = \\
  =&\la \partial_{\theta\theta}\psi^{n,\xi_1,x}d^{n+1}(\theta-\hat\theta),d^{n+1}(\theta-\hat\theta) \ra \\
&- \la (H^{vex,\chi_3,{n+1},\chi_4}_{\ww} -H_{\ww}^{cav,\chi_1,{n+1},\chi_2}) d^{n+1}(\ww-\hat\ww),d^{n+1}(\ww-\hat\ww)\ra \\
 &- \la d^{n+1}\hat\theta,\partial_{\theta\theta}\hat\psi^{n,\hat\xi_1,n}d^{n+1}\hat\theta + \partial_{\theta\theta}\psi^{n,\xi_1,n}d^{n+1}\hat\theta-2\partial_{\theta\theta}\psi^{n,\xi_1,n}d^{n+1}\theta \ra \\
 &- \la d^{n+1}\hat\ww,\hat H^{cav,n+1,\hat\chi_2}_{\ww}d^{n+1}\hat\ww + H^{cav,\chi_1,n+1,\chi_2}_{\ww}d^{n+1}\hat\ww-2H^{cav,\chi_1,{n+1},\chi_2}_{\ww}d^{n+1}\ww\ra \\
  &+ \la d^{n+1}\hat\ww,\hat H^{vex,\hat\chi_3,{n+1},\hat\chi_4}_{\ww}d^{n+1}\hat\ww + H^{vex,\chi_3,{n+1},\chi_4}_{\ww}d^{n+1}\hat\ww-2H^{vex,\chi_3,{n+1},\chi_4}_{\ww}d^{n+1}\ww\ra.
\end{align*}
%
For the remaining terms we find by adding suitable zeros
\begin{align*}
&(iii) + \ldots + (xii) \\
=& \la \partial_\theta\psi(\rho^n,\theta^n,\eta^n|\hat\rho^n,\hat\theta^n,\hat\eta^n), d^{n+1}\hat\theta\ra  + \la \partial_\theta\tilde\psi(\rho,\theta^{n+1},\eta|\hat\rho,\hat\theta^{n+1},\hat\eta^n), d^{n+1}\hat\rho\ra   \\
&+ \la \partial_\eta\tilde\psi(\rho,\theta^{n+1},\eta|\hat\rho,\hat\theta^{n+1},\hat\eta^n, d^{n+1}\hat\eta\ra  \\
& + \la \partial_{\theta\theta}\hat\psi^n(\theta^n-\hat\theta^n) + \partial_{\theta\rho}\hat\psi^n(\rho^n-\hat\rho^n) +  \partial_{\theta\eta}\hat\psi(\eta^n-\hat\eta^n), d^{n+1} \hat\theta\ra \\
& + \la \partial_{\rho\theta}\hat\psi_{vex}^{n+1}(\theta^{n+1}-\hat\theta^{n+1}) + \partial_{\rho\theta}\hat\psi_{cav}^{n,n+1,n}(\theta^{n+1}-\hat\theta^{n+1}) + \partial_{\rho\rho}\hat\psi_{vex}^{n+1}(\rho^{n+1}-\hat\rho^{n+1})   \\
& \quad + \partial_{\rho\rho}\hat\psi_{cav}^{n,n+1,n}(\rho^{n}-\hat\rho^{n})+ \partial_{\rho\eta}\hat\psi_{vex}^{n+1}(\eta^{n+1}-\hat\eta^{n+1}) + \partial_{\rho\eta}\hat\psi_{cav}^{n,n+1,n}(\eta^{n}-\hat\eta^{n}) , d^{n+1}\hat\rho\ra\\
& + \la \partial_{\eta\theta}\hat\psi^{n+1}(\theta^{n+1}-\hat\theta^{n+1}) + \partial_{\eta\theta}\hat\psi^{n,n+1,n}(\theta^{n+1}-\hat\theta^{n+1}) + \partial_{\eta\rho}\hat\psi_{vex}^{n+1}(\rho^{n+1}-\hat\rho^{n+1})   \\
& \quad + \partial_{\eta\rho}\hat\psi_{cav}^{n,n+1,n}(\rho^{n}-\hat\rho^{n})+ \partial_{\eta\eta}\hat\psi_{vex}^{n+1}(\eta^{n+1}-\hat\eta^{n+1}) + \partial_{\eta\eta}\hat\psi_{cav}^{n,n+1,n}(\eta^{n}-\hat\eta^{n}) , d^{n+1}\hat\eta\ra\\
&- \la \partial_\theta\hat\psi^{n+1} -\partial_\theta\hat\psi^n,\theta^{n+1}-\hat\theta^{n+1} \ra \\
&+ \la \partial_\rho\tilde\psi(\hat\rho,\hat\theta^{n+1},\hat\eta), d^{n+1}(\rho-\hat\rho)\ra + \la \partial_\eta\tilde\psi(\hat\rho,\hat\theta^{n+1},\hat\eta), d^{n+1}(\eta-\hat\eta)\ra \\
&-\la \partial_\rho\hat\psi^{n+1}, \rho^{n+1}-\hat\rho^{n+1}\ra + \la \partial_\rho\hat\psi^n, \rho^n-\hat\rho^{n}\ra -\la \partial_\eta\hat\psi^{n+1}, \eta^{n+1}-\hat\eta^{n+1}\ra + \la \partial_\eta\hat\psi^n, \eta^n-\hat\eta^{n}\ra. \\
\end{align*}
The first three inner products are already part of the results. Hence, we concentrate on the rest. For the next step, we perform the following Taylor expansion and algebraic manipulation
\begin{align*}
 &- \la \partial_\theta\hat\psi^{n+1} -\partial_\theta\hat\psi^n,\theta^{n+1}-\hat\theta^{n+1}  \ra 
= - \la \partial_{\theta\theta}\hat\psi^{\hat\omega_1}d^{n+1}\hat\theta + \partial_{\theta\rho}\hat\psi^{\hat\omega_1}d^{n+1}\hat\rho + \partial_{\theta\eta}\hat\psi^{\hat\omega_1}d^{n+1}\hat\eta, \theta^{n+1}-\hat\theta^{n+1} \ra,\\
&\qquad\qquad \hat\omega_1 \in (\hat\theta^n,\hat\theta^{n+1}) \\
&\la \partial_\rho\tilde\psi(\hat\rho,\hat\theta^{n+1},\hat\eta), d^{n+1}(\rho-\hat\rho)\ra -\la \partial_\rho\hat\psi^{n+1}, \rho^{n+1}-\hat\rho^{n+1}\ra + \la \partial_\rho\hat\psi^n, \rho^n-\hat\rho^{n}\ra \\
= & \la \partial_\rho\tilde\psi(\hat\rho,\hat\theta^{n+1},\hat\eta)-\partial_\rho\hat\psi^{n+1}, d^{n+1}(\rho-\hat\rho)\ra + \la \partial_\rho\hat\psi^n-\partial_\rho\hat\psi^{n+1}, \rho^n-\hat\rho^n\ra \\
& \\
&\la \partial_\eta\tilde\psi(\hat\rho,\hat\theta^{n+1},\hat\eta), d^{n+1}(\eta-\hat\eta)\ra -\la \partial_\eta\hat\psi^{n+1}, \eta^{n+1}-\hat\eta^{n+1}\ra + \la \partial_\eta\hat\psi^, \eta^n-\hat\eta^{n}\ra \\
= & \la \partial_\eta\tilde\psi(\hat\rho,\theta^{n+1},\hat\eta)-\partial_\eta\hat\psi^{n+1}, d^{n+1}(\eta-\hat\eta)\ra + \la \partial_\eta\psi^n-\partial_\eta\hat\psi^{n+1}, \eta^n-\hat\eta^n\ra
\end{align*}
With these insights, some rearrangement and expansion of the last two equalities we
\begin{align*}
&(iii) + \ldots + (xii) \\
&= \la \partial_\theta\psi(\rho^n,\theta^n,\eta^n|\hat\rho^n,\hat\theta^n,\hat\eta^n), d^{n+1}\hat\theta\ra  + \la \partial_\theta\tilde\psi(\rho,\theta^{n+1},\eta|\hat\rho,\hat\theta^{n+1},\hat\eta^n), d^{n+1}\hat\rho\ra   \\
&+ \la \partial_\theta\tilde\psi(\rho,\theta^{n+1},\eta|\hat\rho,\hat\theta^{n+1},\hat\eta^n, d^{n+1}\hat\eta\ra  \\
& -\la \partial_{\theta\theta}\hat\psi^nd^{n+1}(\theta-\hat\theta) + (\partial_{\theta\theta}\hat\psi^{\hat\omega_1}-\partial_{\theta\theta}\hat\psi^n)(\theta^{n+1}-\hat\theta^{n+1}), d^{n+1} \hat\theta\ra \\
& +\la \partial_{\theta\rho}\hat\psi^n(\rho^n-\hat\rho^n) +  \partial_{\theta\eta}\hat\psi^n(\eta^n-\hat\eta^n), d^{n+1} \hat\theta\ra -\la \partial_{\theta\rho}\hat\psi^{\hat\omega_1}d^{n+1}\hat\rho^{n+1} +  \partial_{\theta\eta}\hat\psi^{\hat\omega_1}d^{n+1}\hat\eta, \theta^{n+1}-\hat\theta^{n+1}\ra \\
& + \la \partial_{\rho\theta}\hat\psi_{vex}^{n+1}(\theta^{n+1}-\hat\theta^{n+1}) + \partial_{\rho\theta}\hat\psi_{cav}^{n,n+1,n}(\theta^{n+1}-\hat\theta^{n+1}) + \partial_{\rho\rho}\hat\psi_{vex}^{n+1}(\rho^{n+1}-\hat\rho^{n+1})   \\
&  \quad+ \partial_{\rho\rho}\hat\psi_{cav}^n(\rho^{n}-\hat\rho^{n})+ \partial_{\rho\eta}\hat\psi_{vex}^{n+1}(\eta^{n+1}-\hat\eta^{n+1}) + \partial_{\rho\eta}\hat\psi_{cav}^n(\eta^{n}-\hat\eta^{n}) , d^{n+1}\hat\rho\ra\\
& + \la \partial_{\eta\theta}\hat\psi_{vex}^{n+1}(\theta^{n+1}-\hat\theta^{n+1})+ \partial_{\eta\theta}\hat\psi_{cav}^{n,n+1,n}(\theta^{n+1}-\hat\theta^{n+1}) + \partial_{\eta\rho}\hat\psi_{vex}^{n+1}(\rho^{n+1}-\hat\rho^{n+1})   \\
& \quad+ \partial_{\eta\rho}\hat\psi_{cav}^n(\rho^{n}-\hat\rho^{n})+ \partial_{\eta\eta}\hat\psi_{vex}^{n+1}(\eta^{n+1}-\hat\eta^{n+1}) + \partial_{\eta\eta}\hat\psi_{cav}^n(\eta^{n}-\hat\eta^{n}) , d^{n+1}\hat\eta\ra\\
&- \la \partial_{\rho\rho}\hat\psi_{cav}^{\hat\omega_2}d^{n+1}\hat\rho + \partial_{\rho\eta}\hat\psi_{cav}^{\hat\omega_2}d^{n+1}\hat\eta, d^{n+1}(\rho-\hat\rho)\ra - \la \partial_{\eta\rho}\hat\psi_{cav}^{\hat\omega_3}d^{n+1}\hat\rho + \partial_{\eta\eta}\hat\psi_{cav}^{\hat\omega_3}d^{n+1}\hat\eta, d^{n+1}(\eta-\hat\eta)\ra \\
& - \la \partial_{\rho\theta}\hat\psi^{\hat\omega_4}d^{n+1}\hat\theta + \partial_{\rho\rho}\hat\psi^{\hat\omega_4}d^{n+1}\hat\rho + \partial_{\rho\eta}\hat\psi^{\hat\omega_4}d^{n+1}\hat\eta ,\rho^n-\hat\rho^n\ra \\
& - \la \partial_{\eta\theta}\psi^{\hat\omega_5}d^{n+1}\hat\theta + \partial_{\eta\rho}\hat\psi^{\hat\omega_5}d^{n+1}\hat\rho + \partial_{\eta\eta}\hat\psi^{\hat\omega_5}d^{n+1}\hat\eta ,\eta^n-\hat\eta^n\ra, \\
& \omega_2,\omega_3\in \Big[(\hat\rho^n,\hat\eta^n)^\top,(\hat\rho^{n+1},\hat\eta^{n+1})^\top\Big], \qquad \omega_4, \omega_5 \in \Big[(\hat\rho^n,\hat\theta^n,\hat\eta^n)^\top,(\hat\rho^{n+1},\hat\theta^{n+1},\hat\eta^{n+1})^\top\Big].
\end{align*}
We are now in the position to regroup the terms into the final form. Algebraic manipulation yields
\begin{align*}
 &(iii)-(xii) =\\
 =& \la \partial_\theta\psi(\rho^n,\theta^n,\eta^n|\hat\rho^n,\hat\theta^n,\hat\eta^n), d^{n+1}\hat\theta\ra  + \la \partial_\theta\tilde\psi(\rho,\theta^{n+1},\eta|\hat\rho,\hat\theta^{n+1},\hat\eta^n), d^{n+1}\hat\rho\ra   \\
&+ \la \partial_\theta\tilde\psi(\rho,\theta^{n+1},\eta|\hat\rho,\hat\theta^{n+1},\hat\eta^n, d^{n+1}\hat\eta\ra  \\
 &-\la \partial_{\theta\theta}\hat\psi^nd^{n+1}(\theta-\hat\theta) + (\partial_{\theta\theta}\hat\psi^{\hat\omega_1}-\partial_{\theta\theta}\hat\psi^n)(\theta^{n+1}-\hat\theta^{n+1}), d^{n+1} \hat\theta\ra \\   
 & +\la (\partial_{\rho\theta}\hat\psi^n-\partial_{\rho\theta}\hat\psi^{\hat\omega_4})(\rho^n-\hat\rho^n) + (\partial_{\eta\theta}\hat\psi^n-\partial_{\eta\theta}\hat\psi^{\hat\omega_5})(\eta^n-\hat\eta^n), d^{n+1}\hat\theta\ra \\
 & +   \la (\partial_{\rho\theta}\hat\psi_{vex}^{n+1} + \partial_{\rho\theta}\hat\psi_{cav}^{n,n+1,n}-\partial_{\rho\theta}\hat\psi^{\hat\omega_1})d^{n+1}\hat\rho, \theta^{n+1}-\hat\theta^{n+1}\ra \\
  & +  \la (\partial_{\eta\theta}\hat\psi_{vex}^{n+1} + \partial_{\eta\theta}\hat\psi_{cav}^{n,n+1,n}-\partial_{\eta\theta}\hat\psi^{\hat\omega_1})d^{n+1}\hat\eta, \theta^{n+1}-\hat\theta^{n+1}\ra \\
  & + \la \partial_{\rho\rho}\hat\psi_{vex}^{n+1}(\rho^{n+1}-\hat\rho^{n+1}) + \partial_{\rho\rho}\hat\psi_{cav}^{n,n+1,n}(\rho^{n}-\hat\rho^{n})  + \partial_{\rho\eta}\hat\psi_{vex}^{n+1}(\eta^{n+1}-\hat\eta^{n+1}) + \partial_{\rho\eta}\hat\psi_{cav}^{n,n+1,n}(\eta^{n}-\hat\eta^{n}) \\
& \quad - \partial_{\rho\rho}\hat\psi_{cav}^{\hat\omega_2}d^{n+1}(\rho-\hat\rho)\ra - \partial_{\rho\eta}\tilde\psi_{cav}^{\hat\omega_2}d^{n+1}(\eta-\hat\eta)  - \partial_{\rho\rho}\psi^{\hat\omega_4}(\rho^n-\hat\rho^n) - \partial_{\rho\eta}\psi^{\hat\omega_5}(\eta^n-\hat\eta^n) , d^{n+1}\hat\rho\ra \\
& + \la  \partial_{\eta\rho}\hat\psi_{vex}^{n+1}(\rho^{n+1}-\hat\rho^{n+1}) + \partial_{\eta\rho}\hat\psi_{cav}^{n,n+1,n}(\rho^{n}-\hat\rho^{n})  + \partial_{\eta\eta}\hat\psi_{vex}^{n+1}(\eta^{n+1}-\hat\eta^{n+1}) + \partial_{\eta\eta}\hat\psi_{cav}^{n,n+1,n}(\eta^{n}-\hat\eta^{n}) \\
& \quad- \partial_{\rho\eta}\tilde\psi_{cav}^{\hat\omega_2}d^{n+1}(\rho-\hat\rho)\ra - \partial_{\eta\eta}\tilde\psi_{cav}^{\hat\omega_2}d^{n+1}(\eta-\hat\eta)  - \partial_{\rho\eta}\psi^{\hat\omega_4}(\rho^n-\hat\rho^n) - \partial_{\eta\eta}\psi^{\hat\omega_5}(\eta^n-\hat\eta^n)  , d^{n+1}\hat\eta\ra \\
 =&(a) + \ldots + (f).
\end{align*}
We consider the last two inner products $(e), (f)$ and by adding a suitable zero, we find
\begin{align*}
 &(e) + (f) \\
 =& \la \partial_{\rho\rho}\hat\psi_{vex}^{n+1}(\rho^{n+1}-\hat\rho^{n+1}) + \partial_{\rho\rho}\hat\psi_{cav}^{n,n+1,n}(\rho^{n}-\hat\rho^{n}) + \partial_{\rho\eta}\psi_{vex}^{n+1}(\eta^{n+1}-\hat\eta^{n+1}) + \partial_{\rho\eta}\hat\psi_{cav}^{n,n+1,n}(\eta^{n}-\hat\eta^{n}) \\
& \quad - \partial_{\rho\rho}\hat\psi_{cav}^{\hat\omega_2}d^{n+1}(\rho-\hat\rho)\ra - \partial_{\rho\eta}\hat\psi_{cav}^{\hat\omega_2}d^{n+1}(\eta-\hat\eta)  - \partial_{\rho\rho}\hat\psi^{\hat\omega_4}(\rho^n-\hat\rho^n) - \partial_{\rho\eta}\hat\psi^{\hat\omega_5}(\eta^n-\hat\eta^n) , d^{n+1}\hat\rho\ra \\
 &+ \la  \partial_{\eta\rho}\psi_{vex}^{n+1}(\rho^{n+1}-\hat\rho^{n+1}) + \partial_{\eta\rho}\psi_{cav}^{n,n+1,n}(\rho^{n}-\hat\rho^{n}) + \partial_{\eta\eta}\psi_{vex}^{n+1}(\eta^{n+1}-\hat\eta^{n+1}) + \partial_{\eta\eta}\psi_{cav}^{n,n+1,n}(\eta^{n}-\hat\eta^{n}) \\
& \quad- \partial_{\rho\eta}\psi_{cav}^{\hat\omega_2}d^{n+1}(\rho-\hat\rho)\ra - \partial_{\eta\eta}\psi_{cav}^{\hat\omega_2}d^{n+1}(\eta-\hat\eta) - \partial_{\rho\eta}\hat\psi^{\hat\omega_4}(\rho^n-\hat\rho^n) - \partial_{\eta\eta}\hat\psi^{\hat\omega_5}(\eta^n-\hat\eta^n)  , d^{n+1}\hat\eta\ra \\
=&  \la \partial_{\rho\rho}\hat\psi_{vex}^{n+1}(\rho-\hat\rho) + \partial_{\rho\eta}\hat\psi_{vex}^{n+1}d^{n+1}(\eta-\hat\eta) + \Big[ \partial_{\rho\rho}\hat\psi_{vex}^{n+1} + \partial_{\rho\rho}\hat\psi_{cav}^{n,n+1,n} - \partial_{\rho\rho}\hat\psi^{\hat\omega_4} \Big](\rho^{n}-\hat\rho^{n})  \\
& \quad + \Big[ \partial_{\rho\eta}\hat\psi_{vex}^{n+1} + \partial_{\rho\eta}\hat\psi_{cav}^{n,n+1,n} - \partial_{\rho\eta}\hat\psi^{\hat\omega_5} \Big](\eta^{n}-\hat\eta^{n})  - \partial_{\rho\rho}\hat\psi_{cav}^{\hat\omega_2}d^{n+1}(\rho-\hat\rho)\ra - \partial_{\rho\eta}\hat\psi_{cav}^{\hat\omega_2}d^{n+1}(\eta-\hat\eta) , d^{n+1}\hat\rho \ra \\
&+\la \partial_{\rho\eta}\hat\psi_{vex}^{n+1}d^{n+1}(\rho-\hat\rho) + \partial_{\eta\eta}\hat\psi_{vex}^{n+1}d^{n+1}(\eta-\hat\eta) + \Big[ \partial_{\rho\eta}\hat\psi_{vex}^{n+1} + \partial_{\rho\eta}\hat\psi_{cav}^{n,n+1,n} - \partial_{\rho\eta}\hat\psi^{\hat\omega_4} \Big](\rho^{n}-\hat\rho^{n})  \\
& \quad + \Big[ \partial_{\eta\eta}\hat\psi_{vex}^{n+1} + \partial_{\eta\eta}\hat\psi_{cav}^{n,n+1,n} - \partial_{\eta\eta}\hat\psi^{\hat\omega_5} \Big](\eta^{n}-\hat\eta^{n})  - \partial_{\eta\rho}\hat\psi_{cav}^{\hat\omega_2}d^{n+1}(\rho-\hat\rho)\ra - \partial_{\eta\eta}\hat\psi_{cav}^{\hat\omega_2}d^{n+1}(\eta-\hat\eta) , d^{n+1}\hat\eta \ra \\
\leq &\la (H^{vex,n+1}_{\ww}-H^{cav,\hat\omega_2}_{\ww})d^{n+1}(\ww-\hat\ww),d^{n+1}\hat\ww \ra  \\
& +\tau(\norm{\rho^n-\hat\rho^n}_0 + \norm{\eta^n-\hat\eta^n}_0)\norm{d^{n+1}_\tau\hat\rho}_{0,\infty}(\norm{\hat\u^{n+1} - \omega_4}_{0,\infty} + \norm{\hat\u^n - \omega_4}_{0,\infty}) \\
&+ \tau(\norm{\rho^n-\hat\rho^n}_0 +\norm{\eta^n-\hat\eta^n}_0)\norm{d^{n+1}_\tau\hat\eta}_{0,\infty}(\norm{\hat\u^{n+1} - \omega_5}_{0,\infty} + \norm{\hat\u^{n} - \omega_5}_{0,\infty}).
\end{align*}
Here we used the notation $\mathbf{u}=(\rho,\theta,\eta)^\top$. Similarly, we estimate
\begin{align*}
 &(a) + \ldots +(d) = \\
 &-\la\partial_{\theta\theta}\hat\psi^nd^{n+1}(\theta-\hat\theta) + (\partial_{\theta\theta}\hat\psi^ {\hat\omega_1}-\partial_{\theta\theta}\hat\psi^n)(\theta^{n+1}-\hat\theta^{n+1}), d^{n+1} \hat\theta\ra \\    
 &+\la (\partial_{\rho\theta}\hat\psi^n-\partial_{\rho\theta}\hat\psi^{\hat\omega_4})(\rho^n-\hat\rho^n) + (\partial_{\eta\theta}\hat\psi^n-\partial_{\eta\theta}\hat\psi^{\hat\omega_5})(\eta^n-\hat\eta^n), d^{n+1}\hat\theta\ra  \\
  \leq& -\la \partial_{\theta\theta}\hat\psi^nd^{n+1}(\theta-\hat\theta), d^{n+1} \hat\theta\ra +\tau\norm{d_\tau^{n+1}\hat\theta}_{0,\infty}(\norm{\theta^{n+1}-\hat\theta^{n+1}}_0^2\norm{\hat\u^n-\omega_1}_{0,\infty} \\ &\quad+ \norm{\rho^n-\hat\rho^n}_0^2\norm{\hat\u^n-\omega_4}_{0,\infty} 
 + \norm{\eta^n-\hat\eta^n}_0^2\norm{\hat\u^n-\omega_5}_{0,\infty}  ).
\end{align*}
In the following we will combine this with $(i),(ii)$ and after rearrangement we find
\begin{align*}
\mathcal{R} \leq& -\mathcal{D}_{num}^{n+1}(\rho-\hat\rho,\theta-\hat\theta,\eta-\hat\eta) +\la \partial_\theta\psi(\rho^n,\theta^n,\eta|\hat\rho^n,\hat\theta^n,\hat\eta^n), d^{n+1}\hat\theta\ra    \\
& + \la \partial_\theta\tilde\psi(\rho,\theta^{n+1},\eta|\hat\rho,\hat\theta^{n+1},\hat\eta^n), d^{n+1}\hat\rho\ra+ \la \partial_\theta\tilde\psi(\rho,\theta^{n+1},\eta|\hat\rho,\hat\theta^{n+1},\hat\eta^n, d^{n+1}\hat\eta\ra  \\
&+ \tfrac{1}{2}\la d^{n+1}\hat\theta, 2\partial_{\theta\theta}\psi^{n,\xi_1,n}d^{n+1}\theta-2\partial_{\theta\theta}\hat\psi^nd^{n+1}(\theta-\hat\theta)  - \partial_{\theta\theta}\hat\psi^{n,\hat\xi_1,n}d^{n+1}\hat\theta - \partial_{\theta\theta}\psi^{n,\xi_1,n}d^{n+1}\hat\theta  \ra\\
&+ \tfrac{1}{2}\la d^{n+1}\hat\ww,2H^{cav,\chi_1,{n+1},\chi_2}_{\ww}d^{n+1}\ww -2H^{cav,\hat\omega_2}_{\ww}d^{n+1}(\ww-\hat\ww)\\
& \quad- \hat H^{cav,\hat\chi_1,{n+1},\hat\chi_2}_{\ww}d^{n+1}\hat\ww - H^{cav,\chi_1,n+1,\chi_2}_{\ww})d^{n+1}\hat\ww\ra \\
  &- \tfrac{1}{2}\la d^{n+1}\hat\ww, 2H^{vex,\chi_3,n+1,\chi_4}_{\ww}d^{n+1}\ww - 2\hat H^{vex,n+1}_{\ww}d^{n+1}(\ww-\hat\ww) \\
  &\quad -\hat H^{vex(\hat\chi_3,n+1,\hat\chi_4}_{\ww}d^{n+1}\hat\ww - H^{vex,\chi_3,n+1,\chi_4}_{\ww}d^{n+1}\hat\ww\ra \\
 & +\tau\norm{d_\tau^{n+1}\hat\theta}_{0,\infty}(\norm{\theta^{n+1}-\hat\theta^{n+1}}_0^2\norm{\hat\u^n-\omega_1}_{0,\infty} \\ &\quad+ \norm{\rho^n-\hat\rho^n}_0^2\norm{\hat\u^n-\omega_4}_{0,\infty} 
 + \norm{\eta^n-\hat\eta^n}_0^2\norm{\hat\u^n-\omega_5}_{0,\infty}  ) \\
& +\tau(\norm{\rho^n-\hat\rho^n}_0 + \norm{\eta^n-\hat\eta^n}_0)\norm{d^{n+1}_\tau\hat\rho}_{0,\infty}(\norm{\hat\u^{n+1} - \omega_4}_{0,\infty} + \norm{\hat\u^n - \omega_4}_{0,\infty}) \\
&+ \tau(\norm{\rho^n-\hat\rho^n}_0 +\norm{\eta^n-\hat\eta^n}_0)\norm{d^{n+1}_\tau\hat\eta}_{0,\infty}(\norm{\hat\u^{n+1} - \omega_5}_{0,\infty} + \norm{\hat\u^{n} - \omega_5}_{0,\infty}) \\
   \leq &\mathcal{D}_{num}^{n+1}(\rho-\hat\rho,\theta-\hat\theta,\eta-\hat\eta)  +\la \partial_\theta\psi(\rho^n,\theta^n,\eta|\hat\rho^n,\hat\theta^n,\hat\eta^n), d^{n+1}\hat\theta\ra     \\
&+ \la \partial_\theta\tilde\psi(\rho,\theta^{n+1},\eta|\hat\rho,\hat\theta^{n+1},\hat\eta^n), d^{n+1}\hat\rho\ra+ \la \partial_\theta\tilde\psi(\rho,\theta^{n+1},\eta|\hat\rho,\hat\theta^{n+1},\hat\eta^n, d^{n+1}\hat\eta\ra  \\
&+ \tau\norm{d^{n+1}_\tau\hat\theta}_{0,\infty}\norm{d^{n+1}\theta}_0\norm{(\rho^n,\xi_1,\eta^n)^\top-\hat\u^n}_0 \\
  &+ \tau^2\norm{d^{n+1}_\tau\hat\theta}_{0,\infty}^2(\norm{(\hat\rho^n,\hat\xi_1,\hat\eta^n)^\top-\hat\u^n}_0+\norm{(\rho^n,\xi_1,\eta^n)^\top-\hat\u^n}_0) \\
& +\tau\norm{d^{n+1}_\tau\hat\ww}_{0,\infty}\norm{d^{n+1}\ww}_0\norm{(\chi_1,\chi_2))^\top-\hat\omega_2}_0 \\
  &+ \tau^2\norm{d^{n+1}_\tau\hat\ww}_{0,\infty}^2(\norm{(\chi_1,\chi_2)^\top-\hat\omega_2}_0+\norm{(\hat\chi_1,\hat\chi_2)^\top-\hat\omega_2}_0) \\
  &+ \tau\norm{d^{n+1}_\tau\hat\ww}_{0,\infty}^2\norm{(\chi_3,\theta^{n+1},\chi_4)^\top-\hat\u^{n+1}}_0 \\
  &+ \tau^2\norm{d^{n+1}_\tau\hat\ww}_{0,\infty}^2(\norm{(\chi_3,\theta^{n+1},\chi_4)^\top-\hat\u^{n+1}}_0+\norm{(\hat\chi_3,\hat\theta^{n+1},\hat\chi_4)^\top-\hat\u^{n+1}}_0) \\
   & +\tau\norm{d_\tau^{n+1}\hat\theta}_{0,\infty}(\norm{\theta^{n+1}-\hat\theta^{n+1}}_0^2\norm{\hat\theta^n-\hat\omega_1}_{0,\infty} \\ &\quad+ \norm{\rho^n-\hat\rho^n}_0^2\norm{\hat\u^n-\hat\omega_4}_{0,\infty} 
 + \norm{\eta^n-\hat\eta^n}_0^2\norm{\hat\u^n-\hat\omega_5}_{0,\infty}  ) \\
& +\tau(\norm{\rho^n-\hat\rho^n}_0 + \norm{\eta^n-\hat\eta^n}_0)\norm{d^{n+1}_\tau\hat\rho}_{0,\infty}(\norm{\hat\u^{n+1} - \hat\omega_4}_{0,\infty} + \norm{\hat\u^n - \hat\omega_4}_{0,\infty}) \\
&+ \tau(\norm{\rho^n-\hat\rho^n}_0 +\norm{\eta^n-\hat\eta^n}_0)\norm{d^{n+1}_\tau\hat\eta}_{0,\infty}(\norm{\hat\u^{n+1} - \hat\omega_5}_{0,\infty} + \norm{\hat\u^{n} - \hat\omega_5}_{0,\infty}) 
\end{align*}
with the numerical dissipation
\begin{align*}
 &\mathcal{D}_{num}^{n+1}(\rho-\hat\rho,\theta-\hat\theta,\eta-\hat\eta) \\
 & = \frac{\gamma_\rho}{2}\norm{\nabla d^{n+1}(\rho-\hat\rho)}_0^2 + \frac{\gamma_\eta}{2}\norm{\nabla d^{n+1}(\eta-\hat\eta)}_0^2 + \la \partial_{\theta\theta}\psi^{n,\xi_1,x}d^{n+1}(\theta-\hat\theta),d^{n+1}(\theta-\hat\theta) \ra \\
&- \la (H^{vex,\chi_3,{n+1},\chi_4}_{\ww} -H_{\ww}^{cav,\chi_1,{n+1},\chi_2}) d^{n+1}(\ww-\hat\ww),d^{n+1}(\ww-\hat\ww)\ra \leq 0 .
\end{align*}
We estimate all remaining norms, which by construction yields
\begin{align*}
&\norm{d^{n+1}\theta}_0 \leq \norm{\theta^{n+1}-\hat\theta^{n+1}}_0 + \norm{\theta^{n}-\hat\theta^{n}}_0 + \tau\norm{d^{n+1}_\tau\hat\theta}_0\\
&\norm{(\rho^n,\xi_1,\eta^n)^\top-\hat\u^n}_0 \leq \norm{\ww^n-\hat\ww^n}_0 + C\norm{\theta^n-\hat\theta^n}_0 + C\norm{\theta^{n+1}-\hat\theta^{n+1}}_0 + \tau\norm{d^{n+1}_\tau\hat\theta^{n+1}}_0  \\
  & \norm{(\hat\rho^n,\hat\xi_1,\hat\eta^n)^\top-\hat\u^n}_0 \leq  C\tau\norm{d^{n+1}_\tau\hat\theta}_0 \\
& \norm{d^{n+1}\ww}_0 \leq \norm{\ww^{n+1}-\hat\ww^{n+1}}_0 + \norm{\ww^{n}-\hat\ww^{n}}_0 + \tau\norm{d^{n+1}_\tau\hat\ww}_0\\
&\norm{(\chi_1,\chi_2))^\top-\hat\omega_2}_0 \leq C(\norm{\ww^{n}-\hat\ww^n}_0 + \norm{\ww^{n+1}-\hat\ww^{n+1}}_0 + \tau\norm{d^{n+1}_\tau\hat\ww}_0)\\
  &\norm{(\hat\chi_1,\hat\chi_2)^\top-\hat\omega_2}_0 \leq C\tau\norm{d_\tau^{n+1}\hat\ww}_0^2 \\
  &\norm{(\chi_3,\theta^{n+1},\chi_4)^\top-\hat\u^{n+1}}_0 \leq C(\norm{\ww^{n}-\hat\ww^n}_0 + \norm{\ww^{n+1}-\hat\ww^{n+1}}_0 + \tau\norm{d^{n+1}_\tau\hat\ww}_0)\\
  &\norm{(\hat\chi_3,\hat\theta^{n+1},\hat\chi_4)^\top-\hat\u^{n+1}}_0 \leq C\tau(\norm{d_\tau^{n+1}\hat\ww}_0 + \norm{d_\tau^{n+1}\hat\theta}_0)\\
  &\norm{\hat\theta^n-\hat\omega_1}_{0,\infty} \leq C\tau\norm{d^{n+1}_\tau \hat\theta^{n+1}}_{0,\infty} \\
  &\norm{\hat\u^{n}-\hat\omega_4}_{0,\infty} \leq C\tau(\norm{d_\tau^{n+1}\hat\ww}_{0,\infty} + \norm{d_\tau^{n+1}\hat\theta}_{0,\infty}) \\
 & \norm{\hat\u^{n}-\hat\omega_5}_{0,\infty} \leq C\tau(\norm{d_\tau^{n+1}\hat\ww}_{0,\infty} + \norm{d_\tau^{n+1}\hat\theta}_{0,\infty})\\
&\norm{\hat\u^{n+1} - \hat\omega_4}_{0,\infty} \leq C\tau(\norm{d_\tau^{n+1}\hat\ww}_{0,\infty} + \norm{d_\tau^{n+1}\hat\theta}_{0,\infty}) \\
& \norm{\hat\u^{n+1} - \hat\omega_5}_{0,\infty}   \leq C\tau(\norm{d_\tau^{n+1}\hat\ww}_{0,\infty} + \norm{d_\tau^{n+1}\hat\theta}_{0,\infty}).
\end{align*}

Using the above estimates yields
\begin{align*}
 \mathcal{R} \leq&  -\mathcal{D}_{num}^{n+1}(\rho-\hat\rho,\theta-\hat\theta,\eta-\hat\eta) +\la \partial_\theta\psi(\rho^n,\theta^n,\eta|\hat\rho^n,\hat\theta^n,\hat\eta^n), d^{n+1}\hat\theta\ra    \\
&+ \la \partial_\theta\tilde\psi(\rho,\theta^{n+1},\eta|\hat\rho,\hat\theta^{n+1},\hat\eta^n), d^{n+1}\hat\rho\ra + \la \partial_\theta\tilde\psi(\rho,\theta^{n+1},\eta|\hat\rho,\hat\theta^{n+1},\hat\eta^n, d^{n+1}\hat\eta\ra \\
&+ C\tau(\mathcal{W}(\rho,\theta,\eta|\hat\rho,\hat\theta,\hat\eta)^n + \mathcal{W}(\rho,\theta,\eta|\hat\rho,\hat\theta,\hat\eta)^{n+1}) + C\tau^3  .
\end{align*}

The remaining terms are discrete versions of the remainder in the relative entropy ansatz. However, they can be estimated in a straightforward manner such that
\begin{align*}
 \mathcal{R} &\leq  -\mathcal{D}_{num}^{n+1}(\rho-\hat\rho,\theta-\hat\theta,\eta-\hat\eta)+ C\tau(\mathcal{W}^n + \mathcal{W}^{n+1}) + C\tau^3  .
\end{align*}

\subsection{Dissipative contribution}\label{subs:dissip}

Let us consider now the dissipative contributions, i.e.
\begin{align*}
 \mathcal{D} =&  -\la d^{n+1}(e-\hat e), \theta^{n+1}-\hat \theta^{n+1}\ra  +\la d^{n+1}(\rho-\hat\rho),\mu_{\rho}^{n+1}-\hat\mu_{\rho}^{n+1} + r_{2}^{n+1}\ra \\
&+ \la d^{n+1}(\eta-\hat\eta),\mu_{\eta}^{n+1}-\hat\mu_{\eta}^{n+1}+ r_{5}^{n+1}\ra   
\end{align*}

Inserting into the $v_{1}=\mu_{\rho}^{n+1}-\hat\mu_{\rho}^{n+1} + r_{2}^{n+1},\xi=\theta^{n+1}-\hat \theta^{n+1}, w_{1}=\mu_{\eta}^{n+1}-\hat\mu_{\eta}^{n+1}+ r_{5}^{n+1}$ into \eqref{eq:pg1}--\eqref{eq:pg5} and \eqref{eq:ppg1}--\eqref{eq:ppg5} yields
\begin{align*}
  \mathcal{D} =&  -\la \LL_{11}\nabla (\mu_{\rho}^{n+1}-\hat\mu_{\rho}^{n+1})- \LL_{12}\nabla(\theta^{n+1}-\hat\theta^{n+1})+ (\mu_{\eta}^{n+1}-\hat\mu_{\eta}^{n+1})\LL_{13},\nabla(\mu_{\rho}^{n+1}-\hat\mu_{\rho}^{n+1} + r_{2}^{n+1}) \ra \\
  &+\la \LL_{12}\nabla (\mu_{\rho}^{n+1}-\hat\mu_{\rho}^{n+1})- \LL_{22}\nabla(\theta^{n+1}-\hat\theta^{n+1})+ (\mu_{\eta}^{n+1}-\hat\mu_{\eta}^{n+1})\LL_{23},\nabla(\theta^{n+1}-\hat\theta^{n+1} \ra \\
  &-\la\LL_{13}\cdot\nabla(\mu_{\rho}^{n+1}-\hat\mu_{\rho}^{n+1})-\LL_{23}\cdot\nabla(\theta^{n+1}-\hat\theta^{n+1})+\LL_{33}(\mu_{\eta}^{n+1}-\hat\mu_{\eta}^{n+1}),\mu_{\eta}^{n+1}-\hat\mu_{\eta}^{n+1}+ r_{5}^{n+1}\ra \\
  &+ \la r_{1}^{n+1}, \mu_{\rho}^{n+1}-\hat\mu_{\rho}^{n+1} + r_{2}^{n+1}\ra - \la r_{3}^{n+1}, \theta_{h}^{n+1}-\hat\theta_{h}^{n+1}\ra + \la r_{4}^{n+1}, \mu_{\eta}^{n+1}-\hat\mu_{\eta}^{n+1} + r_{5}^{n+1}\ra \\
  \leq& -(1-2\delta)\tau\mathcal{D}_\LL(\mu_{\rho}^{n+1}-\hat\mu_{\rho}^{n+1},\theta^{n+1}-\hat\theta^{n+1},\mu_{\eta}^{n+1}-\hat\mu_{\eta}^{n+1}) + C\tau\la \tilde\psi_\rho-\tilde{\hat\psi}_\rho, 1\ra^2 + C\tau\norm{\theta^{n+1}-\hat\theta^{n+1}}_0^2   \\
  & +C(\LL)\tau(\norm{r_{1}^{n+1}}_{-1}^2 + \norm{r_{2}^{n+1}}_1^2+ \norm{r_{3}^{n+1}}_{-1}^2 + \norm{r_{4}^{n+1}}_0^2 + \norm{r_{5}^{n+1}}_0^2.  
\end{align*}

\subsection{Quadratic part:} \label{subs:quadpart}
Due to the non-convex nature, the relative entropy is stabilised by an $L^2$ term for $\rho-\hat\rho, \eta-\hat\eta$. The temporal change in time is easily computed as
\begin{align*}
 &\frac{\lambda}{2}(\norm{\rho^{n+1}-\hat\rho^{n+1}}_0^2 + \norm{\eta^{n+1}-\hat\eta^{n+1}}_0^2 - \norm{\rho^{n}-\hat\rho^{n}}_0^2 - \norm{\eta^{n}-\hat\eta^{n}}_0^2   ) \\
 =& \lambda\la d^{n+1}(\rho-\hat\rho),\rho^{n+1}-\hat\rho^{n+1} \ra + \lambda\la d^{n+1}(\eta-\hat\eta),\eta^{n+1}-\hat\eta^{n+1} \ra \\
 &- \frac{\lambda}{2}\norm{d^{n+1}(\rho-\hat\rho)}_0^2 - \frac{\lambda}{2}\norm{d^{n+1}(\eta-\hat\eta)}_0^2 \\
 =& (i) + (ii) + (iii) + (iv).
\end{align*}
For the first and second term we insert $v_{1}=\rho^{n+1}-\hat\rho^{n+1}$ into \eqref{eq:pg1}, \eqref{eq:ppg1} and $w_{1}=\eta^{n+1}-\hat\eta^{n+1}$ into \eqref{eq:pg4}, \eqref{eq:ppg4} which yields
\begin{align*}
 (i) + (ii) =& -\lambda\tau\la \LL_{11}\nabla (\mu_{\rho}^{n+1}-\hat\mu_{\rho}^{n+1})- \LL_{12}\nabla(\theta^{n+1}-\hat\theta^{n+1})+ (\mu_{\eta}^{n+1}-\hat\mu_{\eta}^{n+1})\LL_{13},\nabla(\rho^{n+1}-\hat\rho^{n+1}) \ra \\
 & -\lambda\tau\la\LL_{13}\cdot\nabla(\mu_{\rho}^{n+1}-\hat\mu_{\rho}^{n+1})-\LL_{23}\cdot\nabla(\theta^{n+1}-\hat\theta^{n+1})+\LL_{33}(\mu_{\eta}^{n+1}-\hat\mu_{\eta}^{n+1}),\eta^{n+1}-\hat\eta^{n+1}\ra \\
 & + \lambda\tau\la r_{1}^{n+1},\rho^{n+1}-\hat\rho^{n+1} \ra + \lambda\tau\la r_{4}^{n+1},\eta^{n+1}-\hat\eta^{n+1} \ra \\
 \leq & \delta\tau\mathcal{D}_\LL(\mu_{\rho}^{n+1}-\hat\mu_{\rho}^{n+1},\theta^{n+1}-\hat\theta^{n+1},\mu_{\eta}^{n+1}-\hat\mu_{\eta}^{n+1})  \\
 &+ C(\LL,\lambda)\tau(\norm{\rho^{n+1}-\hat\rho^{n+1}}_1^2 + \norm{\eta^{n+1}-\hat\eta^{n+1}}_0^2)+ C\tau(\norm{r_{1}^{n+1}}_{-1}^2 + \norm{r_{4}^{n+1}}_0^2).
\end{align*}

\subsection{Summation of all estimates}\label{subs:collection}
Here we will collect all estimates from above and recall Lemma \ref{eq:rel_ent_equiv}, i.e. that the $H^1$- norm of $\rho^{n+1}-\hat\rho^{n+1}$ and $\eta^{n+1}-\hat\eta^{n+1}$ can be bounded by the relative entropy.

Summing the above results together and setting $\delta=1/6$ we find
\begin{align}
 \mathcal{W}_{\lambda}(\rho,\theta,\eta|\hat\rho,\hat\theta,\eta)\vert_{t^n}^{t^{n+1}} &+ \frac{\tau}{2}\mathcal{D}_\LL(\mu_{\rho}^{n+1}-\hat\mu_{\rho}^{n+1},\theta^{n+1}-\hat\theta^{n+1},\mu_{\eta}^{n+1}-\hat\mu_{\eta}^{n+1}) \notag\\
 &+ \tilde{\mathcal{D}}_{num}^{n+1}(\rho-\hat\rho,\theta-\hat\theta,\eta-\hat\eta)\label{eq:relentfullcolla}\\
 \leq& C\tau\mathcal{W}_{\lambda}(\rho,\theta,\eta|\hat\rho,\hat\theta,\eta)^{n+1} + C\tau\mathcal{W}_{\lambda}(\rho,\theta,\eta|\hat\rho,\hat\theta,\eta)^{n} + C\tau^3 \notag\\
&+\tau C(\LL)(\norm{r_{1}^{n+1}}_{-1}^2 + \norm{r_{2}^{n+1}}_1^2+ \norm{r_{3}^{n+1}}_{-1}^2 + \norm{r_{4}^{n+1}}_0^2 + \norm{r_{5}^{n+1}}_0^2 \notag
\end{align}
with relative numerical dissipation
\begin{align*}
 \tilde{\mathcal{D}}_{num}^{n+1}(\rho-\hat\rho,\theta-\hat\theta,\eta-\hat\eta) :=&\mathcal{D}_{num}^{n+1}(\rho-\hat\rho,\theta-\hat\theta,\eta-\hat\eta) \\
 &+ \frac{\lambda}{2}\norm{d^{n+1}(\rho-\hat\rho)}_0^2 + \frac{\lambda}{2}\norm{d^{n+1}(\eta-\hat\eta)}_0^2 .
\end{align*}

The final results follow from the discrete Gronwall Lemma setting with $\tau$ sufficiently small.

\bibliographystyle{abbrv}	
\bibliography{lit.bib}

\end{document}